\documentclass[reqno,a4paper, 12pt]{amsart}

\usepackage[a4paper=true,pdfpagelabels]{hyperref}
\usepackage{graphicx}

\usepackage[ansinew]{inputenc}
\usepackage{amsfonts,epsfig}
\usepackage{latexsym}
\usepackage{amsmath}
\usepackage{amssymb}

\newtheorem{theorem}{Theorem}[section]
\newtheorem{lemma}[theorem]{Lemma}
\newtheorem{corollary}[theorem]{Corollary}
\newtheorem{proposition}[theorem]{Proposition}

\newtheorem{lettertheorem}{Theorem}

\theoremstyle{definition}

\theoremstyle{remark}

\numberwithin{equation}{section}

\newcommand{\D}{\mathbb{D}}
\newcommand{\DD}{\widehat{\mathcal{D}}}
\newcommand{\N}{\mathbb{N}}
\newcommand{\R}{\mathbb{R}}

\newcommand{\e}{\varepsilon}
\newcommand{\ep}{\varepsilon}

\renewcommand{\phi}{\varphi}

\newcommand{\T}{\mathbb{T}}

\newcommand{\B}{\mathcal{B}}
\newcommand{\BMOA}{\mathord{\rm BMOA}}

\def\a{\alpha}       \def\b{\beta}        \def\g{\gamma}
       \def\De{{\Delta}}    \def\e{\varepsilon}
     \def\om{\omega}      
       \def\t{\theta}       
         \def\r{\rho}         \def\z{\zeta}
                  \def\vp{\varphi}

\def\R{{\mathcal R}}
\def\I{{\mathcal I}}
\def\Inv{{\mathcal Inv}}
\newcommand{\CC}{\mathcal{C}}

\renewcommand{\H}{\mathcal{H}}

\newcommand{\op}{\mathrm{o}}

\def\dm{\,d(\omega\otimes m)}

\newenvironment{Prf}{\noindent{\emph{Proof of}}}{\hfill$\Box$ }
\begin{document}

\title[Small weighted Bergman spaces ]{Small weighted Bergman spaces }
\author{Jos\'e \'Angel Pel\'aez}

\address{Departamento de An\'alisis Matem\'atico, Universidad de M\'alaga, Campus de
Teatinos, 29071 M\'alaga, Spain} \email{japelaez@uma.es}

\thanks{The  author is supported in part by the Ram\'on y Cajal program
of MICINN (Spain), Ministerio de Edu\-ca\-ci\'on y Ciencia, Spain,
MTM2011-25502 and MTM2014-52865-P, from La Junta de Andaluc{\'i}a, (FQM210) and
(P09-FQM-4468).}

\date{\today}


\begin{abstract}
 This paper is based on the course \lq\lq Weighted Hardy-Bergman spaces\rq\rq\, I delivered in
 the Summer School  \lq\lq Complex and Harmonic Analysis and Related Topics\rq\rq at the Mekrij\"arvi research station of  University
of Eastern Finland, June $2014$. The main  purpose of this survey is to present recent progress
on the theory of Bergman spaces $A^p_\om$, induced by radial weights $\om$  satisfying the doubling property $\int_r^1\om(s)\,ds\le C\int_{\frac{1+r}{2}}^1\om(s)\,ds$.
\end{abstract}

\maketitle



\section{Introduction}
Let $\H(\D)$ denote the space  of all analytic functions in the unit disc $\D=\{z:|z|<1\}$.
For $f\in \H (\D)$ and $0<r<1$,\, set
    \begin{equation*}
    \begin{split}
    M_p(r,f)&=\left(\frac{1}{2\pi}\int_{0}^{2\pi} |f(re^{it})|^p\,dt\right)^{1/p},\quad
    0<p<\infty,\\
    M_\infty(r,f)&=\sup_{|z|=r}|f(z)|.
    \end{split}
    \end{equation*}
For $0<p\le \infty $, the Hardy space $H^p$ consists of
$f\in \H(\mathbb D)$ such that $\Vert f\Vert _{H^p}=
\sup_{0<r<1}M_p(r,f)<\infty$.
A  nonnegative integrable function $\om$ on the unit disc $\D$ is called a weight.
 It is radial if $\omega(z)=\omega(|z|)$ for all $z\in\D$.
 For
$0<p<\infty$ and a weight $\omega$, the weighted Bergman
space $A^p_\omega$ is the space of  $f\in\H(\D)$ for
which
    $$
    \|f\|_{A^p_\omega}^p=\int_\D|f(z)|^p\omega(z)\,dA(z)<\infty,
    $$
where $dA(z)=\frac{dx\,dy}{\pi}$ is the normalized
Lebesgue area measure on $\D$. That is, $A^p_\om=L^p_\om\cap \H(\D)$ where $L^p_\om$ is the corresponding
weighted Lebesgue space.
As usual, we write~$A^p_\alpha$ for the standard weighted
Bergman space induced by
the  radial weight $(1-|z|^2)^\alpha$, where
$-1<\alpha<\infty$ \cite{DurSchus,HKZ,Zhu}.  We denote $dA_\a=(\a+1)(1-|z|^2)^\alpha\,dA(z)$ and $\om(E)=\int_{E}\om(z)\,dA(z)$ for short.
 We recall that the Bloch space $\mathcal{B}$ \cite{ACP}  consists of $f\in\H(\D)$ such that
    $$
    \|f\|_{\mathcal{B}}=\sup_{z\in\D}|f'(z)|(1-|z|^2)+|f(0)|<\infty.
    $$

The Carleson square $S(I)$ based on an
interval $I\subset\T$ is the set $S(I)=\{re^{it}\in\D:\,e^{it}\in I,\,
1-|I|\le r<1\}$, where $|E|$ denotes the Lebesgue measure of $E\subset\T$. We associate to
each $a\in\D\setminus\{0\}$ the interval
$I_a=\{e^{i\t}:|\arg(a e^{-i\t})|\le\frac{1-|a|}{2}\}$, and denote
$S(a)=S(I_a)$.

 The theory of standard Bergman spaces $A^p_\alpha$
  has evolved enormously throughout the last decades,
  although important problems such as a description of zero sets or a characterization of invariant subspaces remain open,
  see \cite{DurSchus,HKZ,Zhu} for details.
  \par With respect to a general weighted Bergman space $A^p_\om$, a fact which affects the way of approaching  a good number
  of problems is whether or not $\om$ is  radial. Roughly speaking, we can say that  the
  theory of weighted Bergman spaces $A^p_\om$ induced by  non-radial weights is at early stages
  and essential facts are unknown. For instance,  if $\om$ is a radial weight, one can easily prove that polynomials are dense in $A^p_\om$,
   but this  does not remain true
 for a general weight. For example, the weight
 $$\om(z)=|S(z)|^2=\left|\exp\left(-\frac{1+z}{1-z}\right)\right|^2=\exp\left(-\frac{1-|z|^2}{|1-z|^2}\right)$$
 satisfies that polynomials are not dense in $A^2_\om$ \cite[p. 138]{DurSchus}.  Concerning embeddings,  the sharp inequality
 $M_p(r,f)\lesssim  \|f\|_{\mathcal{B}}\left(\log\frac{e}{1-r}\right)^{p/2}$
 and known results on lacunary series \cite{ClMg},
show that  $\mathcal{B}\subset A^p_\om$ if and only if
$\int_0^1 \om(r)\left(\log\frac{e}{1-r}\right)^{\frac{p}{2}}\,dr<\infty$,
 whenever $\om$ is a radial weight. These observations lead us to the following open questions;
 \begin{enumerate}
\item Which are those weights such that the polynomials are dense in $A^p_\om$?
\item Which are those weights  such that  $\mathcal{B}\subset A^p_\om$?
 \end{enumerate}
   \par Despite these and other obstacles,  some progress has been achieved on the theory of weighed Bergman spaces $A^p_\om$ induced by
   non-radial weights \cite{AlCo,Asserda-Hichame2014,ArPau,BB,PelRat}.
 \medskip
\par   As for the Bergman spaces $A^p_\om$ induced by radial weights it is worth noticing that some advances
 have been obtained on Bergman spaces $A^p_\om$, in the case when $\om$  belongs to  certain classes
   of radial weights, see \cite{DurSchus, HKZ, PelRat,Zhu} and the references therein.
  However, many questions  such that the existence of a
 (strong or weak) factorization of $A^p_\om$-functions or the boundedness of  the  Bergman projection $P_\om$ on $A^p_\om$ \cite{PRprojections},
  are not understood yet. In this paper,
we will be specially concerned to the theory of Bergman spaces $A^p_\om$ induced by  radial weights $\om$ such that $\int_r^1\om(s)\,ds\le C\int_{\frac{1+r}{2}}^1\om(s)\,ds$.
  We shall write $\DD$ for this class of radial weights. A primary motivation for this study is the so called \lq\lq transition phenomena\rq\rq
  from the standard Bergman spaces $A^p_\a$ to the  Hardy space $H^p$. That is,
 in many respects the Hardy space $H^p$ is the limit of $A^p_\a$,
as $\a\to-1$, but it  is a
very rough estimate since most  of the finer function-theoretic
properties of the classical weighted Bergman space $A^p_\a$ are not
carried over to the Hardy space $H^p$. Plenty of results in  \cite{PelRat,PelRatMathAnn,PRprojections} show that spaces $A^p_\om$ induced by
rapidly increasing weights (Section~\ref{sec:weights} below for a  definition),  lie \lq\lq closer\rq\rq to $H^p$
than any $A^p_\a$. Here we will present some of them. 
 Moreover, many tools used in the theory of the classical
Bergman spaces fail to work in $A^p_\om$, $\om\in\DD$, so frequently we have to employ
 appropriate techniques for $A^p_\om$, $\om\in\DD$,  which usually work  on standard Bergman spaces and even on Hardy spaces.

\par The paper is organized as follows; Section~\ref{sec:weights} contains the definition of  classes of radial
weights that are considered in these notes, shows relations
between them, and contains several descriptions of the class $\DD$.  In Section~\ref{sec:measures} we characterize $q$-Carleson measures
for $A^p_\om$, $\om\in\DD$. This result has been recently proved in \cite{PelRatMathAnn}. For the range $q\ge p$, we offer a different proof from that
in \cite{PelRatMathAnn}. Here we follow  ideas from \cite[Chapter~$2$]{PelRat} and in particular we prove the  pointwise
estimate
$$ |f(z)|^\alpha\le C(\a,\om) \sup_{I:\,z\in S(I)}\frac{1}{\om\left(S(I)
    \right)}\int_{S(I)}|f(\xi)|^\alpha\om(\xi)\,dA(\xi)=CM_{\om}(|f|^\alpha)(z)$$
    for any $f\in \H(\D)$, $\a>0$, $\om\in\DD$ and $z\in\D$.
We also show some equivalent norms on $A^p_\om$  and a description  of $q$-Carleson measures for $A^p_\om$ in the case $q<p$. Most of these
last results are presented without a detailed proof.
Section~\ref{sec:factorization} contains the main result in \cite[Chapter~$3$]{PelRat}. There,
by using a probabilistic method introduced by
Horowitz~\cite{HorFacto},  we prove that if $\om$ is a weight (not
necessarily radial) such that
    \begin{equation}\label{222}
    \om(z)\asymp\om(\zeta),\quad z\in\Delta(\zeta,r),\quad \z\in\D,\
    \end{equation}
where $\Delta(\zeta,r)$ denotes a pseudohyperbolic disc, and
polynomials are dense in $A^p_\om$, then each $f\in A^p_\omega$
can be represented in the form $f=f_1\cdot f_2$, where $f_1\in
A^{p_1}_\omega$, $f_2\in A^{p_2}_\omega$ and $\frac{1}{p_1}+
\frac{1}{p_2}=\frac{1}{p}$, and the following norm estimates hold
    \begin{equation*}\label{P1}
    \|f_1\|_{A^{p_1}_\omega}^p\cdot\|f_2\|_{A^{p_2}_\omega}^p\le\frac{p}{p_1}\|f_1\|_{A^{p_1}_\omega}^{p_1}+\frac{p}{p_2}\|f_2\|_{A^{p_2}_\omega}^{p_2}\le
    C(p_1,p_2,\omega)\|f\|_{A^p_\omega}^p.
    \end{equation*}
\par In Section~\ref{sec:zeros}, by mimicking the corresponding  proofs in \cite[Section~$3.2$]{PelRat}, we prove that
 whenever $\om\in\DD$, the union of two $A^p_\om$-zero sets is not an $A^p_\om$-zero set.
\par In Section~\ref{sec:integral} we
characterize those analytic symbols $g$ on $\D$ such that the integral operator $T_g(f)(z)=\int_0^z f(\z)g'(\z)\,d\z$ is bounded
from $A^p_\om$ into $A^q_\om$,  where $0<p, q<\infty$. Finally, in Section~\ref{sec:composition} we
deal with composition operators $C_\vp(f)= f\circ\vp$, where $f\in \H(\D)$ and $\vp$ is an analytic self-map $\vp$ of $\D$.
We recall a recent description \cite{PelRatToeplitz} of bounded and compact composition operators,  from $A^p_\om$
into $A^q_v$, when $\om\in\DD$ and $v$ a radial weight. In the case $q<p$, Theorem~\ref{Theorem:introduction-bounded-composition-operators} (below)
 gives a characterization of bounded (and compact) composition operators that differs from the one in the existing literature~\cite{SmithYang98}
 in the classical case $C_\vp:A^p_\a\to A^q_\b$. Here we extend this last result in order to describe
 bounded (and compact) composition operators from  $A^p_\om$
into $A^q_v$, where $\om$ is a regular weight (see Section~\ref{sec:weights} below for a definition) and $v$ a radial weight.
As far as we know, this result is new.
\medskip

    Throughout these notes, the letter $C=C(\cdot)$ will denote an
absolute constant whose value depends on the parameters indicated
in the parenthesis, and may change from one occurrence to another.
We will use the notation $a\lesssim b$ if there exists a constant
$C=C(\cdot)>0$ such that $a\le Cb$, and $a\gtrsim b$ is understood
in an analogous manner. In particular, if $a\lesssim b$ and
$a\gtrsim b$, then we will write $a\asymp b$.

\section{Radial Weights. Preliminary results}\label{sec:weights}
We recall that $\DD$ is the class of radial weights such that $\widehat{\om}(z)=\int_{|z|}^1\om(s)\,ds$ is doubling, that is, there exists $C=C(\om)\ge1$ such that $\widehat{\om}(r)\le C\widehat{\om}(\frac{1+r}{2})$ for all $0\le r<1$.
We call a radial weight $\om$ regular, denoted by $\om\in\R$, if $\om\in\DD$
and  $\om(r)$ behaves as its integral average over $(r,1)$, that is,
    \begin{equation*}
    \om(r)\asymp\frac{\int_r^1\om(s)\,ds}{1-r},\quad 0\le r<1.
    \end{equation*}
    As to concrete examples, we mention that every standard weight as well as  those given in
\cite[(4.4)--(4.6)]{AS} are regular.
It is clear that  $\om\in\R$ if and only if for each $s\in[0,1)$ there exists a
constant $C=C(s,\omega)>1$ such that
    \begin{equation}\label{eq:r2}
    C^{-1}\om(t)\le \om(r)\le C\om(t),\quad 0\le r\le t\le
    r+s(1-r)<1,
    \end{equation}
and 
\begin{equation}\label{eq:r1}
   \frac{\int_r^1\om(s)\,ds}{1-r}\lesssim \om(r),\quad0\le r<1.
    \end{equation}
The definition of regular weights used here is slightly more general than that in \cite{PelRat}, but the main
 properties  are essentially the same by Lemma~\ref{Lemma:weights-in-D-hat} below and \cite[Chapter~1]{PelRat}.
 \par A radial continuous weight $\om$ is called rapidly increasing, denoted by $\om\in\I$, if
    \begin{equation*}
    \lim_{r\to 1^-}\frac{\int_r^1\om(s)\,ds }{\om(r)(1-r)}=\infty.
    \end{equation*}
It follows from \cite[Lemma~1.1]{PelRat} that $\I\subset\DD$.
Typical examples of rapidly increasing
weights are
    \begin{equation*}\label{eq:def-of-v_alpha}
    v_\a(r)=\left((1-r)\left(\log\frac{e}{1-r}\right)^\a\right)^{-1},\quad 1<\a<\infty.\index{$v_\a(r)$}
    \end{equation*}
Despite their name,  rapidly increasing weights may admit a
strong oscillatory behavior. Indeed,  the weight
    \begin{equation*}\label{pesomalo1}
    \omega(r)=\left|\sin\left(\log\frac{1}{1-r}\right)\right|v_\alpha(r)+1,\quad
    1<\alpha<\infty,
    \end{equation*}
belongs to $\I$ but it does not satisfy
\eqref{eq:r2} \cite[p.~7]{PelRat}.
 Due to this fact, occasionally we consider the class $\widetilde{\I}$   of those weights  $\om\in\I$ satisfying \eqref{eq:r2}.
\par A radial continuous weight $\om$ is called rapidly decreasing 
if
    $
    \lim_{r\to 1^-}\frac{\int_r^1\om(s)\,ds }{\om(r)(1-r)}=0.$
The exponential type weights $
    \om_{\gamma,\alpha}(r)=(1-r)^{\gamma}\exp
    \left(\frac{-c}{(1-r)^\alpha}\right), \,\gamma\ge0,\,
    \alpha,c>0,
    $
are  rapidly decreasing.
 It is worth  mentioning that
the pseudohyperbolic metric is not the right one to describe problems on $A^p_\om$ in this case. Roughly speaking,
the substitute of a pseudohyperbolic disc of center $z$ and radius $r<1$ is constructed by writing  $\om=e^{-\varphi}$,
where $\Delta\varphi>0$, and considering the disc $D\left(z,\frac{c}{\sqrt{\Delta\varphi(z)}}\right)$.
\par The  weighted Bergman spaces $A^p_\om$ induced by rapidly decreasing weights  are similar, but not identical, to weighted  Fock spaces \cite{MarMasOrtGFA2003}. See \cite{Asserda-Hichame2014,ArPau,CP2,CP,oleinik,PP,SeiYouJGA2011} for progress on the theory of these spaces.
For further information on any of these classes, see~\cite[Chapter~1]{PelRat} and the references therein.

\medskip\par
The main aim of this section is to obtain different characterizations and properties of the classes of weights  $\DD$ and $\R$.
We shall go further and  in the next result (and only there in these notes)  $\om$ is assumed to be a finite positive Borel measure on $[0,1)$ and $\widehat{\om}(z)=\int_{|z|}^1\,d\om(t)$ for all $z\in\D$. If there exists $C=C(\om)>0$ such that $\widehat{\om}(r)\le C\widehat{\om}(\frac{1+r}{2})$ for all $r\in[0,1)$, we denote $\om\in\DD$. We write $\dm(z)=d\t\,rd\om(r)/\pi$ for $z=re^{i\t}\in\D$, and
    $$
    \om_x=\int_0^1r^{x}\,d\om(r),\quad x>-1.
    $$
For each $K>1$, let $\r_n=\r_n(\om,K)$ be the sequence defined by $\widehat{\om}(\r_n)=\widehat{\om}(0)K^{-n}$.
\par The following characterizations of the class $\DD$ will be frequently used from here on.

\begin{lemma}\label{Lemma:weights-in-D-hat}
Let $\om$ be a finite positive Borel measure on $[0,1)$. Then the following conditions are equivalent:
\begin{itemize}
\item[\rm(i)] $\om\in\DD$;
\item[\rm(ii)] There exist $C=C(\om)\ge 1$ and $\b=\b(\om)>0$ such that
    \begin{equation*}
    \begin{split}
    \widehat{\om}(r)\le C\left(\frac{1-r}{1-t}\right)^{\b}\widehat{\om}(t),\quad 0\le r\le t<1;
    \end{split}
    \end{equation*}
\item[\rm(iii)] There exist $C=C(\om)>0$ and $\gamma=\gamma(\om)>0$ such that
    \begin{equation*}
    \begin{split}
    \int_0^t\left(\frac{1-t}{1-s}\right)^\g\,d\om(s)
    \le C\widehat{\om}(t),\quad 0\le t<1;
    \end{split}
    \end{equation*}
\item[\rm(iv)] There exist constants $C_0=C_0(\om)>0$ and $C=C(\om)>0$ such that
\begin{equation}\label{cero}
 \widehat{\om}(0)\le C_0\widehat{\om}\left(\frac{1}{2}\right)
\end{equation}
and
    \begin{equation}
    \begin{split}\label{4}
    \int_0^ts^{\frac1{1-t}}\,d\om(s)\le C\widehat{\om}(t),\quad 0\le t<1;
    \end{split}
    \end{equation}
\item[\rm(v)] There exist constants $C_0=C_0(\om)>0$ and $C=C(\om)>0$ such that \eqref{cero} holds and
    \begin{equation}
    \begin{split}\label{5}
    \widehat{\om}(r)\le Cr^{-\frac{1}{1-t}}\widehat{\om}(t),\quad 0\le r\le t<1;
   \end{split}
    \end{equation}
\item[\rm(vi)] Condition \eqref{cero} and
the asymptotic equality
    \begin{equation}\label{seis}
    \int_0^1s^x\,d\om(s)\asymp\widehat{\om}\left(1-\frac1x\right),\quad x\in[1,\infty),
    \end{equation}
are valid;
\item[\rm(vii)] There exists $\lambda=\lambda(\om)\ge0$ such that
    $$
    \int_\D\frac{\dm(z)}{|1-\overline{\z}z|^{\lambda+1}}\asymp\frac{\widehat{\om}(\zeta)}{(1-|\z|)^\lambda},\quad \z\in\D;
    $$
\item[\rm(viii)]  Conditions \eqref{cero} and
$\om^\star(z)\asymp\widehat{\om}(z)(1-|z|)$ as $|z|\ge \frac12$,
hold. Here and on the following
    $$
    \omega^\star(z)=\int_{|z|}^1\log\frac{s}{|z|}s\,d\omega(s),\quad z\in\D\setminus\{0\};
    $$
\item[\rm(ix)] Condition \eqref{cero} holds and there exists $C=C(\om)>0$ such that $\om_{n}\le C\om_{2n}$ for all $n\in\N$;
\item[\rm(x)] Condition \eqref{cero} holds and there exist $C=C(\om)>0$ and $\eta=\eta(\om)>0$ such that
    \begin{equation*}
    \begin{split}
    \om_x\le C\left(\frac{y}{x}\right)^{\eta}\om_y,\quad 0<x\le y<\infty;
    \end{split}
    \end{equation*}
\item[\rm(xi)] There exist $K=K(\om)>1$ and $C=C(\om,K)>1$ such that $1-\r_n(\om,K)\ge C(1-\r_{n+1}(\om,K))$ for all $n\in\N\cup\{0\}$.
\end{itemize}
Moreover, if $\om\in\DD$, there exists $C=C(\om)>0$ such that
    $$
    \int_0^r\frac{dt}{\widehat{\om}(t)(1-t)}\ge\frac{C}{\widehat{\om}(r)},\quad r\in\left[\frac12,1\right).
    $$
\end{lemma}

\par Before presenting the proof of Lemma~\ref{Lemma:weights-in-D-hat}, let us observe that condition \eqref{cero} holds for any weight (absolutely continuous measure) such that $\om>0$ on an interval
contained in $[1/2,1)$, so it is not a real restriction for an admissible weight but a consequence of working in the general setting of positive Borel measures.
\medskip

\begin{Prf}{\em{Lemma~\ref{Lemma:weights-in-D-hat}.}}
We will prove (i)$\Leftrightarrow$(ii), (i)$\Leftrightarrow$(iii)$\Rightarrow$(iv)$\Rightarrow$(v)$\Rightarrow$(i), (iv)$\Leftrightarrow$(vi), (iii)$\Rightarrow$(vii)$\Rightarrow$(i)$\Leftrightarrow$(viii), and since (i) and (vi) together imply (ix), finally (ix)$\Rightarrow$(vi),
(ix)$\Leftrightarrow$(x), and (ii)$\Leftrightarrow$(xi).

Let $\om\in\DD$. If $0\le r\le t<1$ and $r_n=1-2^{-n}$ for all $n\in\N\cup\{0\}$, then there exist $k$ and $m$ such that $r_k\le r<r_{k+1}$ and $r_m\le t<r_{m+1}$. Therefore
    \begin{equation*}
    \begin{split}
    \widehat{\om}(r)&\le\widehat{\om}(r_k)
    \le C\widehat{\om}(r_{k+1})
    \le\cdots
    \le C^{m-k+1}\widehat{\om}(r_{m+1})
    \le C^{m-k+1}\widehat{\om}(t)\\
    &=C^22^{(m-k-1)\log_2 C}\widehat{\om}(t)\le C^2\left(\frac{1-r}{1-t}\right)^{\log_2 C}\widehat{\om}(t),\quad 0\le r\le t<1,
    \end{split}
    \end{equation*}
and hence (ii) is satisfied. Since the choice $t=\frac{1+r}{2}$ in (ii) gives $\widehat{\om}(r)\le C2^\b\widehat{\om}(\frac{1+r}{2})$ for all $r\in[0,1)$, we have shown that $\om\in\DD$ if and only if (ii) is satisfied.

Let $\om\in\DD$. If $0\le t<1$ and $r_n=1-2^{-n}$ for all $n\in\N\cup\{0\}$, then there exists $m$ such that $r_m\le t<r_{m+1}$. Therefore
    \begin{equation*}
    \begin{split}
    \int_0^t\left(\frac{1-t}{1-s}\right)^\g\,d\om(s)
    &\le \int_0^{r_{m+1}}\left(\frac{1-t}{1-s}\right)^\g\,d\om(s)
    =\sum_{n=0}^m\int_{r_n}^{r_{n+1}}\left(\frac{1-t}{1-s}\right)^\g\,d\om(s)\\
    &\le\sum_{n=0}^m\left(\frac{1-r_m}{1-r_{n+1}}\right)^\g\left(\widehat{\om}(r_n)-\widehat{\om}(r_{n+1})\right)\\
    &\le\sum_{n=0}^m\frac{C}{2^{\gamma(m-n-1)}}\widehat{\om}(r_{n+1})
  \\
    &\le\widehat{\om}(r_{m+1})2^{2\gamma}\sum_{n=0}^m\left(\frac{C}{2^{\gamma}}\right)^{m-n+1}
    \le\widehat{\om}(t)2^{2\gamma}\sum_{j=1}^\infty\left(\frac{C}{2^{\gamma}}\right)^{j},
    \end{split}
    \end{equation*}
and we deduce (iii) for $\gamma=\gamma(\om)>\frac{\log C}{\log 2}$. 
 Conversely, if (iii) is satisfied and $0\le r\le t<1$, then
    \begin{equation*}
    \begin{split}
    C\widehat{\om}(t)&\ge\int_0^t\left(\frac{1-t}{1-s}\right)^\g\,d\om(s)
    =(1-t)^\g\int_0^t\left(\int_0^s\g(1-x)^{-\g-1}\,dx+1\right)\,d\om(s)\\
    &=(1-t)^\g\g\int_0^t(1-x)^{-\g-1}\int_x^t\,d\om(s)\,dx+(1-t)^\g\int_0^t\,d\om(s)\\
    &=(1-t)^\g\g\int_0^t(1-x)^{-\g-1}\left(\widehat{\om}(x)-\widehat{\om}(t)\right)\,dx+(1-t)^\g\int_0^t\,d\om(s)\\
    &\ge(1-t)^\g\g\int_0^r(1-x)^{-\g-1}\left(\widehat{\om}(x)-\widehat{\om}(t)\right)\,dx+(1-t)^\g\int_0^t\,d\om(s)\\
    &\ge(1-t)^\g\g\widehat{\om}(r)\int_0^r(1-x)^{-\g-1}\,dx-(1-t)^\g\widehat{\om}(t)\g\int_0^t(1-x)^{-\g-1}\,dx\\
    &\quad+(1-t)^\g\int_0^t\,d\om(s)\\
    &=\left(\frac{1-t}{1-r}\right)^\gamma\widehat{\om}(r)-(1-t)^\g\widehat{\om}(r)-\widehat{\om}(t)+(1-t)^\g\widehat{\om}(t)+(1-t)^\g\int_0^t\,d\om(s)\\
    &=\left(\frac{1-t}{1-r}\right)^\gamma\widehat{\om}(r)-\widehat{\om}(t)+(1-t)^\g(\widehat{\om}(0)-\widehat{\om}(r))\\
    &\ge\left(\frac{1-t}{1-r}\right)^\gamma\widehat{\om}(r)-\widehat{\om}(t),\quad 0\le r\le t<1.
    \end{split}
    \end{equation*}
Therefore (ii), and thus also (i), is valid.

The proof of \cite[Lemma~1.3]{PelRat} shows that (iii) implies (iv).
We include a proof  for the sake of completeness. Condition \eqref{cero} follows trivially from (i).
 A simple calculation shows that for all $s\in (0,1)$ and $x>1$,
    $$
    s^{x-1}(1-s)^\gamma\le\left(\frac{x-1}{x-1+\gamma}\right)^{x-1}\left(\frac{\gamma}{x-1+\gamma}\right)^\gamma
    \le\left(\frac{\gamma}{x-1+\gamma}\right)^\gamma.
    $$
 Therefore (iii), with
$t=1-\frac{1}{x}$, yields
   \begin{equation*}
    \begin{split}
    \int_0^{1-\frac{1}{x}} s^{x}\om(s)\,ds
    &\le\left(\frac{\gamma x}{x-1+\gamma}\right)^\gamma\int_0^{1-\frac{1}{x}}\frac{\om(s)}{x^\gamma(1-s)^\gamma}s\,ds
    \\ & \lesssim
    \int_{1-\frac{1}{x}}^1\om(s)\,ds,\quad x>1,
     \end{split}
    \end{equation*}
 which gives \eqref{4}.
On the other hand,
if (iv) is satisfied and $0\le r\le t<1$, then
    \begin{equation*}
    \begin{split}
    C\widehat{\om}(t)&\ge\int_0^ts^\frac{1}{1-t}\om(s)\,ds
    =\int_0^t\frac{x^{\frac{t}{1-t}}}{1-t}\int_x^td\om(s)\,dx
    =\int_0^t\frac{x^{\frac{t}{1-t}}}{1-t}\left(\widehat{\om}(x)-\widehat{\om}(t)\right)\,dx\\
    &=\int_0^t\frac{x^{\frac{t}{1-t}}}{1-t}\widehat{\om}(x)\,dx-\widehat{\om}(t)\int_0^t\frac{x^{\frac{t}{1-t}}}{1-t}\,dx\\
    &\ge\int_0^r\frac{x^{\frac{t}{1-t}}}{1-t}\widehat{\om}(x)\,dx-\widehat{\om}(t)\int_0^t\frac{x^{\frac{t}{1-t}}}{1-t}\,dx\\
    &\ge\widehat{\om}(r)\int_0^r\frac{x^{\frac{t}{1-t}}}{1-t}\,dx-\widehat{\om}(t)\int_0^t\frac{x^{\frac{t}{1-t}}}{1-t}\,dx
    =\widehat{\om}(r)r^\frac{1}{1-t}-\widehat{\om}(t)t^\frac{1}{1-t},
    \end{split}
    \end{equation*}
and thus
    $$
    r^\frac{1}{1-t}\widehat{\om}(r)\le\left(C+t^\frac1{1-t}\right)\widehat{\om}(t),\quad 0\le r\le t<1,
    $$
which is  \eqref{5}.
Now, by choosing $t=\frac{1+r}{2}$,
\eqref{5} implies
\begin{equation}\label{6}
\widehat{\om}(r)
\le A^{-1}r^{\frac{2}{1-r}}\widehat{\om}(r)
\le A^{-1}(C+1)\widehat{\om}\left(\frac{1+r}{2}\right),\quad \frac12\le r<1,
 \end{equation}
where $A=\min_{r\in\left[\frac12,1\right)}r^{\frac2{1-r}}>0$.
Now, by combining \eqref{cero} and \eqref{6} we deduce
$$ \widehat{\om}(s)\le \widehat{\om}(0)\le C_1\widehat{\om}\left(\frac12\right)
\lesssim \widehat{\om}\left(\frac34\right)\le \widehat{\om}\left(\frac{1+s}{2}\right),\quad 0\le s\le \frac{1}{2},$$
which together with \eqref{6} gives
$\om\in\DD$.

By integrating only from $0$  to $1-\frac{1}{x}$ on the left of \eqref{seis}, we see that (vi)$\Rightarrow$(iv). Conversely,
(iv) implies
\begin{equation*}
\begin{split}\label{uno}
& \widehat{\om}\left(1-\frac{1}{x}\right)\le \widehat{\om}(0)\le C_1\widehat{\om}\left(\frac12\right)\le 4 C_1\int_{\frac12}^1 s^2d\om(s)
\le 4C_1 \int_{0}^1 s^2d\om(s)
\\ & \le 4C_1 \left(\int_0^{1-\frac1x}s^x\,d\om(s)+\int_{1-\frac1x}^1s^x\,d\om(s)\right)
\lesssim \widehat{\om}\left(1-\frac1x\right),\quad 1\le x\le 2,
\end{split}\end{equation*}
which gives \eqref{seis} for $1\le x\le 2$. Moreover,
(iv) implies
    \begin{equation*}
    \begin{split}
    \widehat{\om}\left(1-\frac1x\right)&\asymp\int_{1-\frac1x}^1s^x\,d\om(s)
    \le\int_0^1s^x\,d\om(s)
    =\int_0^{1-\frac1x}s^x\,d\om(s)+\int_{1-\frac1x}^1s^x\,d\om(s)\\
    &\lesssim\widehat{\om}\left(1-\frac1x\right)+\widehat{\om}\left(1-\frac1x\right)\asymp\widehat{\om}\left(1-\frac1x\right),\quad 2\le x<\infty,
    \end{split}
    \end{equation*}
and thus (vi) is satisfied.

Now, let us see (iii) implies (vii). If $|\zeta|\le\frac{1}{2}$, (vii) is equivalent to
\begin{equation}\label{d}
\widehat{\om}(0)\lesssim \dm(\D)=\int_0^1 s\,d\om(s)\lesssim \widehat{\om}(1/2),
\end{equation}
which clearly follows from (i). Moreover,
\begin{equation*}\label{1111}
\begin{split}
\int_\D\frac{\dm(z)}{|1-\overline{\z}z|^{\lambda+1}}
&\asymp\int_0^1\frac{s\,d\om(s)}{(1-|\z|s)^{\lambda}}
=\left(\int_0^{|\z|}+\int_{|\z|}^1\right)\frac{s\,d\om(s)}{(1-|\z|s)^{\lambda}}\\
&\asymp\frac{\widehat{\om}(\z)}{(1-|\z|)^\lambda}+\int_{0}^{|\z|}\frac{s\,d\om(s)}{(1-|\z|s)^{\lambda}},
\quad |\z|\ge\frac12,
\end{split}
\end{equation*}
so by using (iii)
\begin{equation*}
\begin{split}
\frac{\widehat{\om}(\z)}{(1-|\z|)^\lambda} &\le
\frac{\widehat{\om}(\z)}{(1-|\z|)^\lambda}+\int_{0}^{|\z|}\frac{s\,d\om(s)}{(1-|\z|s)^{\lambda}}
\\ &\le \frac{\widehat{\om}(\z)}{(1-|\z|)^\lambda}+\int_{0}^{|\z|}\frac{d\om(s)}{(1-s)^{\lambda}}
\lesssim \frac{\widehat{\om}(\z)}{(1-|\z|)^\lambda},\quad |\z|\ge\frac12,
\end{split}
\end{equation*}
and hence (iii)$\Rightarrow$(vii). Assuming (vii), in particular we have \eqref{d}, which implies
$$\widehat{\om}(x)\le \widehat{\om}(0)\asymp \widehat{\om}(1/2)\le 2 \int_{1/2}^1 s\,d\om(s)\le 2 \int_{x}^1 s\,d\om(s),\quad 0\le x\le \frac12.$$
So
\begin{equation}\label{7}
\widehat{\om}(x)\asymp \int_{x}^1 s\,d\om(s)=\widehat{\om_1}(x),\quad 0\le x<1.
\end{equation}
Moreover,
 for $0< r\le t\in[\frac12,1)$, (vii) yields
    \begin{equation*}
    \begin{split}
    \frac{\widehat{\om}(t)}{(1-t)^\lambda}&\gtrsim\int_0^t\frac{sd\om(s)}{(1-ts)^\lambda}
    =\int_0^t\left(\int_0^s\frac{\lambda t}{(1-tx)^{\lambda+1}}dx+1\right)\,sd\om(s)\\
    &=\int_0^t\frac{\lambda t}{(1-tx)^{\lambda+1}}\left(\widehat{\om_1}(x)-\widehat{\om_1}(t)\right)\,dx+\int_0^t\,sd\om(s)\\
    &=\int_0^t\frac{\lambda t}{(1-tx)^{\lambda+1}}\widehat{\om_1}(x)\,dx-\widehat{\om_1}(t)\int_0^t\frac{\lambda t}{(1-tx)^{\lambda+1}}\,dx+
    \int_0^t\,sd\om(s)\\
    &\ge\widehat{\om_1}(r)\int_0^r\frac{\lambda t}{(1-tx)^{\lambda+1}}\,dx-\frac{\widehat{\om_1}(t)}{(1-t^2)^{\lambda}}+\widehat{\om_1}(0)\\
    &\ge\widehat{\om_1}(r)\frac{1}{(1-tr)^\lambda}-\frac{\widehat{\om_1}(t)}{(1-t)^\lambda},
    \end{split}
    \end{equation*}
and thus bearing in mind \eqref{7}
    $$
    \widehat{\om}(r)\lesssim \frac{(1-tr)^\lambda}{(1-t)^\lambda}\widehat{\om}(t),\quad 0<r\le t\in\left[\frac12,1\right).
    $$
By choosing $t=\frac{1+r}{2}$ we deduce $\om\in\DD$.

The inequalities $1-t\le-\log t\le(1-t)/t$ show that $\om^\star(r)\asymp\int_r^1(s-r)\,d\om(s)$ for $r\ge\frac12$, and hence $\om^\star(r)\lesssim\widehat{\om}(r)(1-r)$ for all $r\ge\frac12$ and any $\om$. Moreover, if $\om\in\DD$, then
    $$
    \om^\star(r)\gtrsim\int_{\frac{1+r}{2}}^1(s-r)\,d\om(s)\ge\left(\frac{1+r}{2}-r\right)\widehat{\om}\left(\frac{1+r}{2}\right)
    \asymp\widehat{\om}(r)(1-r),
    $$
and thus (i)$\Rightarrow$(viii). Conversely, assume that there exists $C=C(\om)>0$ such that
    $$
    \widehat{\om}(r)(1-r)\le C\int_r^1(s-r)\,d\om(s),\quad \frac12\le r<1,
    $$
and let $r_p=\frac{p+r}{p+1}$, where $p>0$. Then
    \begin{equation*}
    \begin{split}
    \widehat{\om}(r)(1-r)&\le C\int_r^{r_p}(s-r)\,d\om(s)+C\int_{r_p}^1(s-r)\,d\om(s)\\
    &\le C\widehat{\om}(r)\left(r_p-r\right)+C(1-r)\widehat{\om}(r_p),
    \end{split}
    \end{equation*}
and hence
    $$
    \widehat{\om}(r)\le\frac{C(p+1)}{1+p-Cp}\widehat{\om}(r_p),\quad \frac12\le r<1.
    $$
If $C<2$ we may take $p=1$ and deduce $\om\in\DD$. For otherwise, fix $p>0$ sufficiently small and use the argument employed in the proof of (i)$\Rightarrow$(ii) together with $1-r_p=(1-r)/(1+p)\asymp1-r$ to obtain
 $$\widehat{\om}(r)\lesssim \widehat{\om}\left(\frac{1+r}{2}\right),\quad \frac12\le r<1.$$
 This together with \eqref{cero}, gives
 $\om\in\DD$. Thus (viii)$\Rightarrow$(i).

It is clear that (i) and (vi) together imply (ix). Conversely, assume (ix) is satisfied.
Let $A=\sup_{n}\left(1-\frac{1}{n+1}\right)^{n}$ and fix $k$ large enough such that
$C^{k}A^{2^k}<1$. Then
    \begin{equation*}
    \begin{split}
    \om_n&\le C\om_{2n}
    \le C^k \om_{2^kn}
    =C^k\left(\int_0^{1-\frac{1}{n+1}}+\int_{1-\frac{1}{n+1}}^1\right)r^{2^{k}n}\,d\om(r)\\
    &\le C^kA^{2^k}\om_n+C^k\widehat{\om}\left(1-\frac{1}{n+1}\right),\quad n\in\N,
    \end{split}
    \end{equation*}
and hence
    $$
    \om_n\le\frac{C^k}{1-C^kA^{2^k}}\widehat{\om}\left(1-\frac{1}{n+1}\right).
    $$
So, if $n\le x<n+1$, we deduce
    $$
    \int_0^1s^x\,d\om(s)\le\om_n\lesssim\widehat{\om}\left(1-\frac{1}{n+1}\right)\le\widehat{\om}\left(1-\frac{1}{x}\right),
    $$
and (vi) follows.

Assume now (ix) and let $1\le x\le y<\infty$. Then there exist $n,m\in\N\cup\{0\}$ such that $n\le x\le n+1$ and $2^m n\le y\le 2^{m+1}n$. Then (ix) gives
    \begin{equation*}
    \begin{split}
    \om_x&\le\om_n\le C^{m+1}\om_{2^{m+1}n}\le2^{(m+1)\log_2C}\om_y\\
    &\le\left(\frac{2y}{n+1}\frac{n+1}{n}\right)^{\log_2C}\om_y
    \le C^2\left(\frac{y}{x}\right)^{\log_2C}\om_y,
    \end{split}
    \end{equation*}
and (x) follows. The choice $y=2n=2x$ gives (x)$\Rightarrow$(ix).

Assume there exist $K=K(\om)>1$ and $C=C(\om)>1$ such that $1-\r_n\ge C(1-\r_{n+1})$ for all $n\in\N\cup\{0\}$. Let $0\le r\le t<1$ and fix $n,k\in\N\cup\{0\}$ such that $\r_n\le r<\r_{n+1}$ and $\r_k\le t<\r_{k+1}$. Then
    \begin{equation*}
    \begin{split}
    1-r&\ge1-\r_{n+1}\ge C(1-\r_{n+2})\ge\cdots\ge C^{k-n-1}(1-\r_{k})\\
    &\ge C^{-2}\left(\frac{K^{-n}}{K^{-(k+1)}}\right)^{\log_K C}(1-t)
    \ge C^{-2}\left(\frac{\widehat{\om}(r)}{\widehat{\om}(t)}\right)^{\log_K C}(1-t),
    \end{split}
    \end{equation*}
and hence
    $$
    \widehat{\om}(r)\le C^\frac{2}{\log_K C}\left(\frac{1-r}{1-t}\right)^\frac1{\log_KC}\widehat{\om}(t),\quad 0\le r\le t<1,
    $$
and thus (ii) is satisfied. Conversely, by choosing $t=\r_{n+1}$ and $r=\r_n$ in (ii), we deduce $1-\r_{n+1}\le\left(\frac{C}{K}\right)^\frac1\b(1-\r_n)$, and (xi) follows by choosing $K>C$.

Moreover, if $\om\in\DD$, there exists $C=C(\om)>0$ such that
    $$
    \int_0^r\frac{dt}{\widehat{\om}(t)(1-t)}\ge\int_{2r-1}^r\frac{dt}{\widehat{\om}(t)(1-t)}
    \ge\frac{1}{\widehat{\om}(2r-1)}\log2\ge
    \frac{\log 2}{C\widehat{\om}(r)},\quad r\in\left[\frac12,1\right).
    $$
The proof of the lemma is now complete.
\end{Prf}
\medskip
\par Let $1<p_0,p_0'<\infty$ such that
$\frac{1}{p_0}+\frac{1}{p'_0}=1$, and let $\eta>-1$. A weight
$\om:\D\to(0,\infty)$ satisfies the \emph{Bekoll\'e-Bonami
$B_{p_0}(\eta)$-condition}\index{Bekoll\'e-Bonami
weight},\index{$B_{p_0}(\eta)$} denoted by $\om\in B_{p_0}(\eta)$,
if there exists a constant $C=C(p_0,\eta,\omega)>0$ such that
    \begin{equation}\label{eq:BB}
    \begin{split}
    &\left(\int_{S(I)}\om(z)(1-|z|)^{\eta}\,dA(z)\right)
    \left(\int_{S(I)}\om(z)^{\frac{-p'_0}{p_0}}(1-|z|)^{\eta}\,dA(z)\right)^{\frac{p_0}{p'_0}}\\
    &\le C|I|^{(2+\eta)p_0}
    \end{split}
    \end{equation}
for every interval $I\subset \T$. Bekoll\'e and Bonami introduced
these weights in~\cite{Bek,BB}, and showed that $\frac{\om(z)}{(1-|z|)^\eta}\in
B_{p_0}(\eta)$ if and only if the Bergman
projection\index{Bergman projection}\index{$P_\eta(f)$}
    $$
    P_\eta(f)(z)=(\eta+1)\int_\D\frac{f(\xi)}{(1-\overline{\xi}z)^{2+\eta}}(1-|\xi|^2)^\eta\,dA(\xi)
    $$
is bounded from $L^{p_0}_\om$ to $A^{p_0}_\om$~\cite{BB}.
\medskip\par
The next lemma shows that a  radial weight $\om$ that
satisfies \eqref{eq:r2} is regular if and only if it is a
Bekoll\'e-Bonami weight. Moreover, Part (iii) quantifies in a
certain sense the self-improving integrability of radial weights.

\begin{lemma}\label{le:RAp}
\begin{itemize}
\item[{\rm(i)}] If $\om\in\R$, then for each $p_0>1$ there exists
$\eta=\eta(p_0,\omega)>-1$ such that $\frac{\om(z)}{(1-|z|)^\eta}$
belongs to $B_{p_0}(\eta)$.

\item[{\rm(ii)}] If $\om$ is a  radial weight such that
\eqref{eq:r2} is satisfied and $\frac{\om(z)}{(1-|z|)^\eta}$
belongs to $B_{p_0}(\eta)$ for some $p_0>0$ and $\eta>-1$, then
$\om\in\R$.

\item[{\rm(iii)}] For each radial weight $\om$ and $0<\a<1$,
define
    $$
    \widetilde{\om}(r)=\left(\int_r^1\om(s)\,ds\right)^{-\a}\om(r),\quad
    0\le r<1.
    $$
Then $\widetilde{\om}$ is also a weight and
$\frac{\int_r^1 \widetilde{\om}(s)\,ds}{(1-r)\widetilde{\om}(r)}=\frac1{1-\a}\frac{\int_r^1 \om(s)\,ds}{(1-r)\om(r)}$ for all
$0\le r<1$.
\end{itemize}
\end{lemma}

\begin{proof}
(i) Since each regular weight is radial, it suffices to show that
there exists a constant $C=C(p,\eta,\omega)>0$ such that
    \begin{equation}\label{eq:rb1}
    \left (\int_{1-|I|}^1\om(t)\,dt\right )\left
    (\int_{1-|I|}^1\om(t)^{\frac{-p'_0}{p_0}}(1-t)^{p'_0\eta}\,dt\right
    )^{\frac{p_0}{p'_0}}\le C|I|^{(1+\eta)p_0}
    \end{equation}
for every interval $I\subset \T$. To prove~\eqref{eq:rb1}, set
$s_0=1-|I|$ and $s_{n+1}=s_n+s(1-s_n)$, where $s\in (0,1)$ is
fixed. Take $p_0$ and $\eta$ such that $\eta>\frac{\log
C}{p_0\log\frac{1}{1-s}}>0$, where the constant $C=C(s,\omega)>1$
is from~\eqref{eq:r2}. Then \eqref{eq:r2} yields
    \begin{equation*}
    \begin{split}
    \int_{1-|I|}^1\om(t)^{\frac{-p'_0}{p_0}}(1-t)^{p'_0\eta}\,dt
    &\le\sum_{n=0}^\infty(1-s_n)^{p'_0\eta}\int_{s_n}^{s_{n+1}}\om(t)^{\frac{-p'_0}{p_0}}\,dt\\
    &\le C^{\frac{p'_0}{p_0}}\sum_{n=0}^\infty (1-s_n)^{p'_0\eta+1}\om(s_{n})^{\frac{-p'_0}{p_0}}\\
    &\le|I|^{p'_0\eta+1}\om(1-|I|)^{\frac{-p'_0}{p_0}}\\
    &\quad\cdot\sum_{n=0}^\infty (1-s)^{n(p'_0\eta+1)}C^{(n+1)\frac{p'_0}{p_0}}\\
    &=C(p_0,\eta,s,\omega)|I|^{p'_0\eta+1}\om(1-|I|)^{\frac{-p'_0}{p_0}},
    \end{split}
    \end{equation*}
which together with \eqref{eq:r1} gives \eqref{eq:rb1}.

(ii) The asymptotic inequality $\frac{\int_r^1 \om(s)\,ds}{\om(r)}\lesssim(1-r)$ follows by
\eqref{eq:rb1} and further appropriately modifying the argument in
the proof of (i). Since the assumption \eqref{eq:r2} gives
$\frac{\int_r^1 \om(s)\,ds}{\om(r)} \gtrsim(1-r)$, we deduce $\om\in\R$.

(iii) If $0\le r<t<1$, then an integration by parts yields
    \begin{equation*}
    \begin{split}
    \int_r^t\frac{\om(s)}{\left(\int_{s}^1\om(v)\,dv\right)^{\a}}\,ds
    &=\left(\int_{r}^1\om(v)\,dv\right)^{1-\a}-\left(\int_{t}^1\om(v)\,dv\right)^{1-\a}\\
    &\quad+\a\int_r^t
    \frac{\om(s)}{\left(\int_{s}^1\om(v)\,dv\right)^{\a}}\,ds,
    \end{split}
    \end{equation*}
from which the assertion follows by letting $t\to1^-$.
\end{proof}

\section{Carleson measures}\label{sec:measures}

For a given Banach space (or a complete metric
space) $X$ of analytic functions on $\D$, a positive Borel measure
$\mu$ on $\D$ is called a \emph{$q$-Carleson measure for $X$}
\index{Carleson measure} if the identity operator\index{identity
operator} $I_d:\, X\to L^q(\mu)$ is bounded.
We shall obtain a description 
of $q$-Carleson measures for the weighted
Bergman space $A^p_{\om}$, $\om\in\DD$. We shall offer a
detailed proof for the case $q\ge p$ which differs from that in
\cite{PelRatMathAnn} and follows the lines of \cite[Chapter $2$]{PelRat}.
\medskip

\subsection{Test functions and the weighted maximal function}

The next result follows from  Lemma~\ref{Lemma:weights-in-D-hat}(vii)
 and its proof.

\begin{lemma}\label{testfunctions1}
Let $0<p<\infty$ and $\omega\in\DD$. Then there is $\lambda_0(\om)$ such that for any $\lambda\ge \lambda_0$ and  each $a\in \D$
the  function $F_{a,p}(z)=\left(\frac{1-|a|^2}{1-\overline{a}z}\right)^{\frac{\lambda+1}{p}}$ is analytic in $\D$ and satisfies
    \begin{equation}\label{eq:tf1}
    |F_{a,p}(z)|\asymp 1,\quad z\in S(a),\quad a\in\D,
    \end{equation}
and
    \begin{equation}\label{eq:tf2}
    \|F_{a,p}\|_{A^p_\om}^p\asymp\om\left(S(a)\right),\quad a\in\D.
    \end{equation}
\end{lemma}

It is
known that $q$-Carleson measures for $\om\in\R$ can
be characterized either in terms of Carleson squares or
pseudohyperbolic discs \cite{OC}. However, this is no longer true when
$\om\in\DD$. So, we shall use tools from harmonic analysis.
\medskip
\par Let us consider the maximal function
    $$
    M_{\om}(\vp)(z)=\sup_{I:\,z\in S(I)}\frac{1}{\om\left(S(I)
    \right)}\int_{S(I)}|\vp(\xi)|\om(\xi)\,dA(\xi),\quad
    z\in\D,
    $$
introduced by H\"ormander~\cite{HormanderL67}. Here we must
require $\vp\in L^1_\om$ and that $\vp(re^{i\t})$ is
$2\pi$-periodic with respect to $\t$ for all $r\in(0,1)$. The function $M_{\om}(\vp)$ plays a role on $A^p_\om$
similar to that of the Hardy-Littlewood maximal function on the Hardy space $H^p$.

\par Now, we are going to get a pointwise control of $|f|$ in terms of $M_{\om}(|f|)$.
\begin{lemma}\label{le:suf1}
Let $0<s<\infty$ and $\om\in\DD$. Then there exists a
constant $C=C(s,\omega)>0$ such that
    \begin{equation}\label{eq:s3}
    |f(z)|^s\le CM_{\om}(f^s)(z),\quad z\in\D,
    \end{equation}
for all $f\in\H(\D)$.
\end{lemma}
\begin{proof}

Let  $\om\in\DD$ and let $C=C(\om)\ge 1$ and $\b=\b(\om)>0$ be those of Lemma~\ref{Lemma:weights-in-D-hat}(ii).  Write $s=\a\gamma$, where $\gamma>\beta+1+\log_2 C>1$. It  suffices to prove the assertion for
the points $re^{i\t}\in\D$ with $r>\frac12$. If $r<\rho<1$, then using that $|f|^\alpha$ is subharmonic and
H\"{o}lder's inequality
    \begin{equation*}
    \begin{split}
    |f(re^{i\t})|^\alpha
    &\le\frac{1}{2\pi}\int_{-\pi}^{\pi}\frac{1-(\frac{r}{\rho})^2}{|1-\frac{r}{\rho}e^{it}|^2}|f(\rho
    e^{i(t+\t)})|^\alpha\,dt
    \\ & \le
    \left(\frac{1}{2\pi}\int_{-\pi}^{\pi}\frac{\left(1-(\frac{r}{\rho})^2\right)^{\gamma-1}}{|1-\frac{r}{\rho}e^{it}|^\g}|f(\rho
    e^{i(t+\t)})|^{\alpha\g}\,dt\right)^{1/\g}
    \\ & \quad
    \cdot\left(\frac{1}{2\pi}\int_{-\pi}^{\pi}\frac{\left(1-(\frac{r}{\rho})^2\right)^{\gamma'-1}}{|1-\frac{r}{\rho}e^{it}|^{\g'}}\,dt\right)^{1/{\g}'},
    \end{split}
    \end{equation*}
    that is
    \begin{equation*}
    \begin{split}
    |f(re^{i\t})|^s
    &\le C(\om,s)\frac{1}{2\pi}\int_{-\pi}^{\pi}\frac{\left(1-(\frac{r}{\rho})^2\right)^{\gamma-1}}{|1-\frac{r}{\rho}e^{it}|^\g}|f(\rho
    e^{i(t+\t)})|^{s}\,dt
   \\  &= C(\om,s)\int_{-\pi}^{\pi}P_{\g}\left(\frac{r}{\rho},t\right)|f(\rho e^{i(t+\t)})|^s\,dt,
    \end{split}
    \end{equation*}
where
    $$
    P_\g(r,t)=\frac{1}{2\pi}\frac{(1-r^2)^{\gamma-1}}{|1-re^{it}|^\g},
    \quad0<r<1.
    $$

 Set
$t_n=2^{n-1}(1-r)$ and $J_n=[-t_n,t_n]$ for $n=0,1,\ldots,N+1$,
where $N$ is the largest natural number such that $t_N<\frac12$.
Further, set $G_0=J_0$, $G_n=J_n\setminus J_{n-1}$ for
$n=1,\ldots,N$, and $G_{N+1}=[-\pi,\pi]\setminus J_N$. Then
    \begin{equation*}
    \begin{split}
    |f(re^{i\t})|^s
    &\le\sum_{n=0}^{N+1}\int_{G_n}P_\g\left(\frac{r}{\rho},t\right)|f(\rho
    e^{i(t+\t)})|^s dt\\
    &\le\sum_{n=0}^{N+1}P_\g\left(\frac{r}{\rho},t_{n-1}\right)\int_{G_n}|f(\rho
    e^{i(t+\t)})|^s dt\\
    &\lesssim\frac{1}{1-\frac{r}{\rho}}\sum_{n=0}^{N+1}2^{-n\g}\int_{G_n}|f(\rho
    e^{i(t+\t)})|^s dt,
    \end{split}
    \end{equation*}
and therefore
    \begin{eqnarray*}
    &&|f(re^{i\t})|^s
    (1-r)\int_{(1+r)/2}^1\omega(\rho)\rho\,d\rho
    \le2\int_r^1|f(re^{i\t})|^s(\rho-r)\omega(\rho)\rho\,d\rho\\
    &&\lesssim\sum_{n=0}^{N+1}2^{-n\g}\int_{r}^1\int_{G_n}\left|f\left(\rho
    e^{i(t+\t)}\right)\right|^s dt\,\omega(\rho)\rho^2\,d\rho.
    \end{eqnarray*}
It follows that
    \begin{equation*}
    \begin{split}
    |f(re^{i\t})|^s
    &\lesssim\sum_{n=0}^{N}2^{-n(\gamma-1)}
    \frac{\int_{r}^1\int_{-t_n}^{t_n}\left|f\left(\rho
    e^{i(t+\t)}\right)\right|^s dt\,\omega(\rho)\rho\,d\rho}{\int_{-t_n}^{t_n}\int_{(1+r)/2}^1\omega(\rho)\rho\,
    d\rho\,dt}\\
    &\quad+2^{-N(\gamma-1)}\frac{\int_{r}^1\int_{-\pi}^{\pi}\left|f\left(\rho
    e^{it}\right)\right|^s
    dt\omega(\rho)\rho\,d\rho}{\int_{-\pi}^{\pi}\int_{(1+r)/2}^1\omega(\rho)\rho
    d\rho\,dt}\\
    &\lesssim\sum_{n=0}^{N}2^{-n(\gamma-1)}
    \frac{\int_{1-t_{n+1}}^1\int_{-t_n}^{t_n}\left|f\left(\rho
    e^{i(t+\t)}\right)\right|^s dt\,\omega(\rho)\rho\,d\rho}{\int_{-t_n}^{t_n}\int_{(1+r)/2}^1\omega(\rho)\rho\,
    d\rho\,dt}\\
    &\quad+2^{-N(\gamma-1)}\frac{\int_{0}^1\int_{-\pi}^{\pi}\left|f\left(\rho
    e^{it}\right)\right|^s
    dt\omega(\rho)\rho\,d\rho}{\int_{-\pi}^{\pi}\int_{(1+r)/2}^1\omega(\rho)\rho
    d\rho\,dt},
    \end{split}
    \end{equation*}
where the last step is a consequence of the inequalities
$0<1-t_{n+1}\le r$. Denoting the interval centered at $e^{i\t}$
and of the same length as $J_n$ by $J_n(\t)$, and applying
Lemma~\ref{Lemma:weights-in-D-hat}(ii), to the
denominators, we obtain
    \begin{equation*}
    \begin{split}
    |f(re^{i\t})|^s&\lesssim\sum_{n=0}^{N}C^n2^{-n(\gamma-1-\b)}
    \frac{\int_{S(J_n(\t))}\left|f(z)\right|^s\omega(z)\,dA(z)}{\omega(S(J_n(\t)))}\\
    &\quad+C^N2^{-N(\gamma-1-\b)}\frac{\int_{\D}|f(z)|^s\omega(z)\,dA(z)}{\omega(\D)}\\
    &\lesssim\left(\sum_{n=0}^\infty2^{-n(\gamma-1-\beta-\log_2C)}\right)M_{\om}(|f|^s)(re^{i\t})
    \lesssim M_{\om}(|f|^s)(re^{i\t}),
    \end{split}
    \end{equation*}
where in the last inequality we have used the election of $\gamma$. This finishes the proof.
\end{proof}

\subsection{Carleson measures. Case $\mathbf{0<p\le q<\infty}$.}
\par Next,   we prove our main result in this section, by combining  a weak $(1,1)$ inequality for the maximal function with the
pointwise estimate \eqref{eq:s3}.

\begin{theorem}\label{th:cm}
Let $0<p\le q<\infty$,  $\om\in\DD$ and let $\mu$ be a
positive Borel measure on $\D$. Then $\mu$ is a $q$-Carleson measure for $A^p_\omega$ if
and only if
    \begin{equation}\label{eq:s1}
    \sup_{I\subset\T}\frac{\mu\left(S(I) \right)}{\left(\om\left(S(I)
    \right)\right)^\frac{q}p}<\infty.
    \end{equation}
Moreover, if $\mu$ is a $q$-Carleson measure for $A^p_\omega$,
then the identity operator $I_d:A^p_{\om}\to L^q(\mu)$ satisfies
    $$
    \|I_d\|^q_{\left(A^p_\om, L^q(\mu)\right)}\asymp\sup_{I\subset\T}\frac{\mu\left(S(I) \right)}{\left(\om\left(S(I)
    \right)\right)^\frac{q}p}.
    $$
\end{theorem}
\begin{proof}
Let $0<p\le q<\infty$ and $\om\in\DD$, and assume first that
$\mu$ is a $q$-Carleson measure for $A^p_\om$. Consider the test
functions $F_{a,p}$\index{$F_{a,p}$} defined in
Lemma~\ref{testfunctions1}. Then the assumption together with
relations \eqref{eq:tf1} and \eqref{eq:tf2} yield
    \begin{equation*}\index{$F_{a,p}$}
    \begin{split}
    \mu(S(a))&\lesssim\int_{S(a)}|F_{a,p}(z)|^q\,d\mu(z)\le\int_{\D}|F_{a,p}(z)|^q\,d\mu(z)\lesssim\|F_{a,p}\|_{A^p_\om}^q\lesssim\om\left(S(a)\right)^\frac{q}{p}
    \end{split}
    \end{equation*}
for all $a\in\D$, and thus $\mu$ satisfies \eqref{eq:s1}.

Conversely, let $\mu$ be a positive Borel measure on $\D$ such
that \eqref{eq:s1} is satisfied. We begin with proving that there
exists a constant $K=K(p,q,\omega)>0$ such that the $L^1_\om$-weak
type inequality\index{$L^1_\om$-weak type inequality}
    \begin{equation}\label{eq:s4}\index{weighted maximal function}\index{$M_\om(\vp)$}
    \mu\left(E_{s}\right)\le
    Ks^{-\frac{q}{p}}\|\vp\|_{L^1_\om}^\frac{q}{p},\quad E_s=\left\{z\in\D: M_{\om}(\vp)(z)>s
    \right\},
    \end{equation}
is valid for all $\varphi\in L^1_\om$ and $0<s<\infty$.

If $E_s=\emptyset$, then \eqref{eq:s4} is clearly satisfied. If
$E_s\not=\emptyset$, then recall that
$I_z=\{e^{i\t}:|\arg(ze^{-i\t})|<(1-|z|)/2\}$ and $S(z)=S(I_z)$,
and define for each $\e>0$ the sets
    $$
    A_s^{\e}=\left\{z\in\D:\, \int_{S(I_z)}|\vp(\xi)|\om(\xi)\,dA(\xi)>s
    \left(\e+\om(S(z)) \right) \right\}
    $$
and
    $$
    B_s^{\e}=\left\{z\in\D:\, I_z\subset I_u\,\text{for some $u\in A_s^{\e}$}\right\}.
    $$
The sets $B_s^{\e}$ expand as $\e\to 0^+$, and
    $$\index{weighted maximal function}\index{$M_\om(\vp)$}
    E_s=\left\{z\in\D: M_{\om}(\vp)(z)>s \right\}=\bigcup_{\e>0}B_s^{\e},
    $$
so
    \begin{equation}\label{eq:s5}
    \mu(E_s)=\lim_{\e\to0^+}\mu(B_s^{\e}).
    \end{equation}
We notice that for each $\e>0$ and $s>0$ there are finitely many
points $z_n\in A_s^{\e}$ such that the arcs $I_{z_n}$ are
disjoint. Namely, if there were infinitely many points $z_n\in
A_s^{\e}$ with this property, then the definition of $A_s^{\e}$
would yield
    \begin{equation}\label{eq:s6}
    s\sum_n[\e+\om(S(z))]\le\sum_n\int_{S(I_{z_n})}|\vp(\xi)|\om(\xi)\,dA(\xi)
    \le\|\vp\|_{L^1_\om},
    \end{equation}
and therefore
    $$
    \infty=s\sum_n\e\le\|\vp\|_{L^1_\om},
    $$
which is impossible because $\vp\in L^1_\om$.

We now use Covering lemma~\cite[p.~161]{Duren1970} to find
$z_1,\ldots,z_m\in A_s^{\e}$ such that the arcs $I_{z_n}$ are
disjoint and
    $$
    A_s^{\e}\subset\bigcup_{n=1}^m\left\{z:I_z\subset J_{z_n}\right\},
    $$
where $J_z$ is the arc centered at the same point as $I_z$ and of
length $5|I_z|$. It follows easily that
    \begin{equation}\label{eq:s7}
    B_s^{\e}\subset\bigcup_{n=1}^m\left\{z:I_z\subset J_{z_n}\right\}.
    \end{equation}
But now the assumption \eqref{eq:s1} and the hypothesis $\om\in\DD$ give
    \begin{equation*}\begin{split}
    \mu\left(\left\{z:I_z\subset J_{z_n}\right\}\right)
    &=\mu\left(\left\{z:S(z)\subset S(J_{z_n}) \right\}\right)\le\mu\left( S(J_{z_n})\right)\\
    &\lesssim\left(\om\left( S(J_{z_n})\right)\right)^\frac{q}p\lesssim\left(\om\left(S(z_n)\right)\right)^\frac{q}{p},\quad
n=1,\ldots,m.
    \end{split}
    \end{equation*}
This combined with \eqref{eq:s7} and \eqref{eq:s6} yields
    $$
    \mu(B_s^{\e})\lesssim\sum_{n=1}^m\left(\om\left(S(z_n)\right)\right)^\frac{q}{p}
    \le\left(\sum_{n=1}^m\om\left(S(z_n)\right)\right)^\frac{q}{p}\le s^{-\frac{q}{p}}
    \|\vp\|_{L^1_\om}^\frac{q}{p},
    $$
which together with \eqref{eq:s5} gives \eqref{eq:s4} for some
$K=K(p,q,\omega)$.

We will now use Lemma~\ref{le:suf1} and \eqref{eq:s4} to show that
$\mu$ is a $q$-Carleson measure for $A^p_\om$. To do this, fix
$\alpha>\frac{1}{p}$ and let $f\in A^p_\om$. For $s>0$, let
$|f|^{\frac{1}{\alpha}}=\psi_{\frac{1}{\alpha},s}+\chi_{\frac{1}{\alpha},s}$,
where
    \begin{equation*}
    \psi_{\frac{1}{\alpha},s}(z)=\left\{
        \begin{array}{rl}
        |f(z)|^{\frac{1}{\alpha}},&\quad\text{if}\;|f(z)|^{\frac{1}{\alpha}}>s/(2K)\\
        0,&\quad\textrm{otherwise}
        \end{array}\right.
    \end{equation*}
and $K$ is the constant in \eqref{eq:s4}, chosen such that
$K\ge1$. Since $p>\frac{1}{\alpha}$, the function
$\psi_{\frac{1}{\alpha},s}$ belongs to $L^1_\om$ for all $s>0$.
Moreover,
    $$
    M_{\om}(|f|^{\frac{1}{\alpha}})\le M_{\om}(\psi_{\frac{1}{\alpha},s})+M_{\om}(\chi_{\frac{1}{\alpha},s})\le
    M_{\om}(\psi_{\frac{1}{\alpha},s})+\frac{s}{2K},
    $$
and therefore
    \begin{equation}\label{eq:inclusion}
    \left\{z\in\D: M_{\om}(|f|^{\frac{1}{\alpha}})(z)>s\right\}
   \subset
   \left\{z\in\D: M_{\om}(\psi_{\frac{1}{\alpha},s})(z)>s/2\right\}.
    \end{equation}
Using Lemma~\ref{le:suf1}, the inclusion \eqref{eq:inclusion},
\eqref{eq:s4} and Minkowski's inequality in continuous form
(Fubini in the case $q=p$), we finally deduce
    \begin{equation*}
    \begin{split}
    \int_{\D}|f(z)|^q\,d\mu(z)&\lesssim
    \int_{\D}\left(M_{\om}(|f|^{\frac{1}{\alpha}})(z)\right)^{q\alpha}\,d\mu(z)\\
    &=q\alpha\int_0^\infty s^{q\alpha-1}
    \mu\left(\left\{z\in\D: M_{\om}(|f|^{\frac{1}{\alpha}})(z)>s \right\} \right)\,ds\\
    &\le q\alpha\int_0^\infty s^{q\alpha-1}
    \mu\left(\left\{z\in\D: M_{\om}(\psi_{\frac{1}{\alpha},s})(z)>s/2\right\}\right)\,ds\\
    &\lesssim
    \int_0^\infty s^{q\alpha-1-\frac{q}{p}}
    \|\psi_{\frac{1}{\alpha},s}\|_{L^1_\om}^\frac{q}p\,ds\\
    &=\int_0^\infty s^{q\alpha-1-\frac{q}{p}}
    \left(\int_{\left\{z:\,|f(z)|^{\frac{1}{\alpha}}>\frac{s}{2K}\right\}}
    |f(z)|^{\frac{1}{\alpha}}\om(z)\,dA(z)\right)^\frac{q}p\,ds\\
    &\le\left(\int_{\D}|f(z)|^{\frac{1}{\alpha}}\om(z)\left(\int_0^{2K|f(z)|^{\frac{1}{\alpha}}}s^{q\alpha-1-\frac{q}{p}}
    \,ds\right)^\frac{p}{q}\,dA(z)\right)^\frac{q}{p}\\
    &\lesssim\left(\int_{\D}|f(z)|^p\om(z)\,dA(z)\right)^\frac{q}{p}.
    \end{split}
    \end{equation*}
Therefore $\mu$ is a $q$-Carleson measure for $A^p_\om$, and the
proof of Theorem~\ref{th:cm}(i) is complete.
\end{proof}

\par The next useful result follows from the proof of Theorem~\ref{th:cm}.
\begin{theorem}\label{co:maxbou}
Let $0<p\le q<\infty$ and $0<\alpha<\infty$ such that $p\alpha>1$.
Let $\om\in\DD$, and let $\mu$ be a positive Borel measure on
$\D$. Then
$[M_{\om}((\cdot)^{\frac{1}{\alpha}})]^{\alpha}:L^p_\omega\to
L^q(\mu)$ is bounded if and only if $\mu$ satisfies \eqref{eq:s1}.
Moreover,
    $$
    \|[M_{\om}((\cdot)^{\frac{1}{\alpha}})]^{\alpha}\|^q_{\left(L^p_\om,L^q(\mu)\right)}\asymp\sup_{I\subset\T}\frac{\mu\left(S(I) \right)}{\left(\om\left(S(I)
    \right)\right)^\frac{q}p}.
    $$
\end{theorem}

\par Before presenting a description of $q$-Carleson measures for $A^p_\om$, where $\om\in\DD$ and $q<p$, we shall obtain
several equivalent $A^p_\om$-norms which are useful to study this problem and some other questions throughout the manuscript.

\subsection{Equivalent norms on $A^p_\om$}
A description of $A^p_\om$ in terms of the maximal function follows from Lemma~\ref{le:suf1} and Theorem~\ref{co:maxbou}.

\begin{corollary}
Let $0<p<\infty$ and $0<\alpha<\infty$ such that $p\alpha>1$.
Let $\om\in\DD$.  Then,
$$\|f\|^p_{A^p_\om}\asymp \|[M_{\om}((f)^{\frac{1}{\alpha}})]^{\alpha}\|^p_{L^p_\om},\quad f\in\H(\D).$$
\end{corollary}
\medskip \par It is well-known that a choice of an appropriate norm is often a key step when solving
 a problem on a space of analytic functions.
For instance, in the study of the integration operator
\begin{displaymath}
    T_g(f)(z)=\int_{0}^{z}f(\zeta)\,g'(\zeta)\,d\zeta,\quad
    z\in\D,\quad g\in\H(\D),
    \end{displaymath}
one wants to get rid of the integral symbol, so one looks  for   norms
in terms of the first derivative.
The first known result in this area was proved by Hardy and Littlewood
for the standard weights \cite{Zhu}.
\begin{lettertheorem}\label{th:1}
If $0<p<\infty$ and $\a>-1$, then
\begin{equation*}
\int_\D \vert f(z)\vert \sp p(1-|z|)^{\alpha}\, dA(z)\asymp
|f(0)|^p+\int_\D \vert f'(z)\vert \sp p(1-|z|)^{p+\alpha}\, dA(z)
\end{equation*}
for all $f\in \H(\D)$.
\end{lettertheorem}
\par Later, this Littlewood-Paley type formula was extended to the following class of weights~\cite{PavP}, which includes any
differentiable decreasing weight and all the standard ones.
 See also \cite{AS0,CP,Si} for previous and further results.
The distortion function of a
radial weight $\om$ is
    $$
    \psi_{\om}(r)=\frac{1}{\om(r)}\int_{r}^1\om(s)\,ds,\quad
    0\le r<1.
    $$
    It was introduced by Siskakis~\cite{Si}.
\begin{theorem}\label{th:LPformula}
Let $0<p<\infty$ and let $\om$ be a differentiable radial weight. If $$\sup_{0<r<1}\frac{\om'(r)}{\om^2(r)}\int_r^1\om(s)\,ds<\infty,$$
 then
 $$\|f\|^p_{A^p_\om}\asymp |f(0)|^p+\int_{\D}|f'(z)|^p\psi^p_\om(|z|)\om(z)\,dA(z),\,f\in\H(\D).$$
\end{theorem}
\par See also \cite{AlCo} for  a Littlewood-Paley type formula  for $\|\cdot\|_{A^p_\om}$-norm, where $\om$ is a Bekoll\'e-Bonami weight. However,  an analogue of Theorem \ref{th:LPformula}
does not exist if $\om\in\DD$ and $p\neq 2$.
\begin{proposition}
Let $p\ne2$. Then there exists $\om\in\DD$ such that, for any
function $\vp:[0,1)\to(0,\infty)$, the relation
\begin{equation}\label{eq:NOL-P}
    \|f\|^p_{A^p_{\om}}\asymp
    \int_\D|f'(z)|^p\varphi(|z|)^p\om(z)\,dA(z)+|f(0)|^p
    \end{equation}
can not be valid for all $f\in\H(\D)$.
\end{proposition}
\begin{proof}
Let first $p>2$ and consider the weight $v_\a(r)=(1-r)^{-1}\left(\log\frac{e}{1-r}\right)^{-\a}$, where $\alpha>1$ is fixed
such that $2<2(\alpha-1)\le p$. Assume on the contrary to the
assertion that \eqref{eq:NOL-P} is satisfied for all $f\in\H(\D)$.
Applying this relation to the function $h_n(z)=z^n$, we obtain
    \begin{equation}\label{eq:NOL-P1}
    \int_0^1r^{np}v_\a(r)\,dr\asymp
    n^p\int_0^1r^{(n-1)p}\varphi(r)^pv_\a(r)\,dr,\quad
    n\in\N.
    \end{equation}
Consider now the lacunary series $h(z)=\sum_{k=0}^\infty z^{2^k}$.
It is easy to see that
    \begin{equation}\label{1}
    M_p(r,h)\asymp\left(\log\frac{1}{1-r}\right)^{1/2},\quad
    M_p(r,h')\asymp\frac{1}{1-r},\quad 0\le r<1.
    \end{equation}
By combining the relations \eqref{eq:NOL-P1}, \eqref{1} and
    $$
    \left(\frac{1}{1-r^p}\right)^p\asymp\sum_{n=1}^\infty
    n^{p-1}r^{(n-1)p},\quad\log\frac{1}{1-r^p}\asymp\sum_{n=1}^\infty
    n^{-1}r^{np},\quad0\le r<1,
    $$
we obtain
    \begin{equation*}\begin{split}
    \int_\D|h'(z)|^p\varphi(z)^pv_\a(z)\,dA(z) &\asymp
    \int_0^1\left(\frac{1}{1-r^p}\right)^p\varphi(r)^pv_\a(r)\,dr \\ &
    \asymp \int_0^1\left(\sum_{n=1}^\infty
    n^{p-1}r^{(n-1)p}\right)\varphi(r)^pv_\a(r)\,dr\\
    &\asymp \sum_{n=1}^\infty n^{p-1}\int_0^1
    r^{(n-1)p}\varphi(r)^pv_\a(r)\,dr\\
    &\asymp \sum_{n=1}^\infty n^{-1}\int_0^1
    r^{np}v_\a(r)\,dr\\
    &\asymp \int_0^1\left(\sum_{n=1}^\infty n^{-1}r^{np}\right)v_\a(r)\,dr\\
    &\asymp \int_0^1\log\frac{1}{1-r^p}\,v_\a(r)\,dr,
    \end{split}\end{equation*}
where the last integral is convergent because $\a>2$. However,
    $$
    \|h\|^p_{A^p_{v_\a}}\asymp\int_0^1\left(\log\frac{1}{1-r}\right)^{p/2}v_\a(r)\,dr=\infty,
    $$
since $p\ge2(\alpha-1)$, and therefore \eqref{eq:NOL-P} fails for
$h\in\H(\D)$. This is the desired contradiction.

If $0<p<2$, we again consider $v_\a$,\index{$v_\a(r)$} where $\a$
is chosen such that $p<2(\alpha-1)\le2$, and use an analogous
reasoning to that above to prove the assertion. Details are
omitted.
\end{proof}
\par Because of the above result we look for other equivalent norms to $\|\cdot\|_{A^p_\om}$ in terms (or involving) the derivative. In fact,
applying the Hardy-Stein-Spencer
identity~\cite{Garnett1981}
    $$
    \|f\|_{H^p}^p=\frac{p^2}{2}\int_\D|f(z)|^{p-2}|f'(z)|^2\log\frac{1}{|z|}\,dA(z)+|f(0)|^p,
    $$
 to the dilated functions $f_r(z)=f(rz)$, $0<r<1$, and integrating with respect to $r\om(r)\,dr$ we obtain such equivalent norm.
\begin{theorem}\label{ThmLittlewood-Paley}
Let $0<p<\infty$, $n\in\N$ and $f\in\H(\D)$, and let $\omega$ be a
radial weight. Then
    \begin{equation}\label{HSB}
    \|f\|_{A^p_\omega}^p=p^2\int_{\D}|f(z)|^{p-2}|f'(z)|^2\omega^\star(z)\,dA(z)+\omega(\D)|f(0)|^p,
    \end{equation}
where
 $$
    \omega^\star(z)=\int_{|z|}^1\omega(s)\log\frac{s}{|z|}s\,ds,\quad z\in\D\setminus\{0\}.
    $$
 In particular,
    \begin{equation}\label{eq:LP2}
    \|f\|_{A^2_\omega}^2=4\|f'\|_{A^2_{\omega^\star}}^2+\omega(\D)|f(0)|^2.
    \end{equation}
\end{theorem}
\par
Fefferman and Stein \cite{FC} obtained the following extension of the classical Littlewood-Paley formula for $H^2$
 \begin{equation*}
    \begin{split}
    \|f\|^p_{H^p}&\asymp\int_\T\,\left(\int_{\Gamma(e^{i\theta})}|f'(z)|^2\,dA(z)\right)^{p/2}d\theta
    +|f(0)|^p,
    \end{split}
   \end{equation*}
where
\begin{equation*}\label{eq:gammadeu}
    \Gamma(e^{i\theta})=\left\{z\in \D:\,|\t-\arg
    z|<\frac12\left(1-|z|\right)\right\},\quad
    u=e^{i\theta}\in\T.
    \end{equation*}
Usually the function $e^{i\theta}\mapsto\left(\int_{\Gamma(e^{i\theta})}|f'(z)|^2\,dA(z)\right)^{1/2}$ is called the square (Lusin)  area function.
\par In order to get an extension of this result to weighted Bergman spaces, we need to define tangential lens type regions
\begin{equation*}
    \Gamma(u)=\left\{z\in \D:\,|\t-\arg
    z|<\frac12\left(1-\frac{|z|}{r}\right)\right\},\quad
    u=re^{i\theta}\in\overline{\D}\setminus\{0\},
    \end{equation*}
induced by points in $\D$, and the tents
    \begin{equation*}
    \begin{split}
    T(z)&=\left\{u\in\D:\,z\in\Gamma(u)\right\},\quad
    z\in\D,\index{$T(z)$}
    \end{split}
    \end{equation*}
which are closely interrelated. By the same method used in the proof of Theorem~\ref{ThmLittlewood-Paley}, we get the following result.
\begin{theorem}\label{th:normacono}
Let $0<p<\infty$ and $f\in\H(\D)$, and let $\omega$ be a
radial weight. Then

\begin{equation}\label{normacono}
    \begin{split}
    \|f\|_{A^p_\omega}^p&\asymp\int_\D\,\left(\int_{\Gamma(u)}|f'(z)|^2
    \,dA(z)\right)^{\frac{p}2}\omega(u)\,dA(u)+|f(0)|^p,
    \end{split}
    \end{equation}
where the constants of comparison depend only on $p$ and $\om$.
\end{theorem}

\par It is worth mentioning that $\om^\star$ is smoother than $\om$. In fact,
    \begin{equation*}
    \om(T(z))\asymp\omega^\star(z),\quad |z|\ge\frac12.
    \end{equation*}
So, bearing in mind Lemma~\ref{Lemma:weights-in-D-hat},
    \begin{equation}\label{3}
    \omega^\star(z)\asymp\omega\left(T(z)\right)\asymp\omega\left(S(z)\right),\quad
   z\in\D,\quad\om\in\DD.
    \end{equation}
\par Before ending this section,  for a function $f$  defined in $\D$, we consider the non-tangential maximal
function of $f$ in the (punctured) unit disc by
    $$
    N(f)(u)=\sup_{z\in\Gamma(u)}|f(z)|,\quad
    u\in\D\setminus\{0\}.
    $$

\begin{lemma}\label{le:funcionmaximalangular}
Let $0<p<\infty$ and let $\om$ be a radial weight. Then there
exists a constant $C>0$ such that
    $$
    \|f\|^p_{A^p_\om}\le\|N(f)\|^p_{L^p_\om}\le C\|f\|^p_{A^p_\om}
    $$
for all $f\in\H(\D)$.
\end{lemma}
A proof can be obtained by dilating and integrating the well-known inequality \cite[Theorem~3.1 on p.~57]{Garnett1981}
 \begin{equation*}
    \|f^\star\|^p_{L^p(\T)}\le C\|f\|^p_{H^p}
    \end{equation*}
    respect to $\om$. Here, and on the sequel $
    f^\star(\z)=\sup_{z\in\Gamma(\z)}|f(z)|\,\text{for}
    \,\z\in\T.
    $
\subsection{Carleson measures. Case $0<q<p<\infty$.}
\par For several classes of weights, $q$-Carleson measures for $A^p_\om$~\cite{OC,Lu93,PP}  have been described, in the triangular case $p>q$, by using an atomic decomposition theorem in the sense of standard Bergman spaces~\cite[Theorem~2.2]{Ro:de}.
However, this approach does not seem to be adequate for the class $\DD$. A sufficient condition can be easily obtained.
 \begin{proposition}\label{pr:cmqmenorp}
 Let $0<q<p<\infty$, $\om$ a radial weight and $\mu$ be a positive Borel measure on~$\D$. If
 $$
    B_\mu(z)=\int_{\Gamma(z)}\frac{d\mu(\z)}{\om(T(\z))},\quad
    z\in\D\setminus\{0\},
    $$
belongs to $L^{\frac{p}{p-q}}_\om$, then $\mu$ is a $q$-Carleson measure for $A^p_\om$.
 \end{proposition}
 \begin{proof}
  Fubini's theorem, H\"{o}lder's inequality and Lemma \ref{le:funcionmaximalangular} yield
    \begin{equation*}
    \begin{split}
    \int_{\D}|f(z)|^q\,d\mu(z)&=\int_{\D}\left(\int_{\Gamma(\z)}\frac{|f(z)|^q\,d\mu(z)}{\om(T(z))}\right)\om(\z)\,dA(\z)\\
    &\le\int_{\D}(N(f)(\z))^qB_\mu(\zeta)\om(\z)\,dA(\z)\\
    &\le\|N(f)\|_{L^p_\om}^q \|B_\mu\|_{L^{\frac{p}{p-q}}_\om}\asymp\|f\|_{A^p_\om}^q \|B_\mu\|_{L^{\frac{p}{p-q}}_\om},
    \end{split}
    \end{equation*}
    for all $ f\in A^p_\om$.
 \end{proof}
It turns out that the reverse of the above result is true \cite[Theorem~$1$]{PelRatMathAnn} for $\om\in\DD$. However, its proof its much more involved.
As in the case $q\ge p$, methods from harmonic analysis are the appropriate ones. To some extent this is natural because the weighted Bergman space $A^p_\om$ induced by $\om\in\DD$ may lie essentially much closer to the Hardy space $H^p$ than any standard Bergman space $A^p_\a$~\cite{PelRat}. Luecking~\cite{Lu90} employed the theory of tent spaces, introduced by Coifman, Meyer and Stein~\cite{CMS} and further considered by Cohn and Verbitsky~\cite{CV}, to study the analogue problem for Hardy spaces.
In \cite{PelRatMathAnn},  an analogue of this theory for Bergman spaces is built and it is a key ingredient in the proof of the following result.

\begin{theorem}\label{Theorem:CarlesonMeasures}
Let $0<q<p<\infty$, $\om\in\DD$ and $\mu$ be a positive Borel measure on~$\D$.
Then the following conditions are equivalent:
\begin{enumerate}
\item[\rm(i)] $\mu$ is a $q$-Carleson measure for $A^p_\om$;
\item[\rm(ii)] The function
    $$
    B_\mu(z)=\int_{\Gamma(z)}\frac{d\mu(\z)}{\om(T(\z))},\quad
    z\in\D\setminus\{0\},
    $$
belongs to $L^{\frac{p}{p-q}}_\om$;
\item[\rm(iii)] $M_\om(\mu)(z)=\sup_{z\in S(a)}\frac{\mu(S(a))}{\left(\om\left(S(a)
    \right)\right)^\a}\in L^{\frac{p}{p-q}}_\om$.
\end{enumerate}
\end{theorem}
\section{Factorization of functions in $A^p_\om$}\label{sec:factorization}
 Factorization theorems  in spaces of analytic functions  are related with
 plenty of issues such as zero sets, dual spaces, Hankel operators or integral operators. We remind the reader of  the following well-known factorization
 of $H^p$-functions \cite{Duren1970}.
 \begin{lettertheorem}\label{th:fhp}
 If $f\not\equiv 0$,   $f\in H^p$, then  $f=B\cdot g$ where $B$ is the  Blaschke product of zeros of $f$ and  $g$ does not vanish on $\D$. Moreover,  $\left \| f \right \|_{H^p}=\left \| g \right \|_{H^p}$.
In particular,
  $f\not\equiv 0$,   $f\in H^1$, can be written as  $f=f_1\cdot f_2$ where $f_2$ does not vanish on $\D$. Moreover,  $\left \| f \right \|_{H^1}=\left \|f_j\right \|_{H^2}$, $j=1,2$.
 \end{lettertheorem}
Because of  the following result, Theorem \ref{th:fhp}  does not remain true for standard Bergman spaces $A^p_\alpha$ \cite{Horzeros}.
\begin{lettertheorem}
Let $0<p<q<\infty$. Then there exists an $A^p$ zero set which
is not an $A^q$ zero set. In particular, it is not possible to represent an arbitrary
$A^1$ function as the product of two functions in $A^2$, one of them nonvanishing.
\end{lettertheorem}
\par Some years later, a weak factorization result was obtained  in the context of Hardy spaces in several variables \cite{CoRoWe}.
\begin{lettertheorem}
If $f\in A^1$ function, then
$$f=\sum_{j=1}^\infty F_jG_j$$
and $\sum_{j=1}^\infty \|F_j\|_{A^2}\|G_j\|_{A^2}\le C\|f\|_{A^1}.$
\end{lettertheorem}
Essentially at the same time, Horowitz \cite{HorFacto} improved this result, obtaining a strong factorization
of $A^p_\alpha$-functions.
\begin{lettertheorem}
Assume that $0<p<\infty$, $\alpha>-1$  and $p^{-1}=p_1^{-1}+p_2^{-1}$. If $f\in A^p_\alpha$, then
there exist $f_1\in A^{p_1}_\a$ and $f_2\in A^{p_2}_\a$
such that $f=f_1\cdot f_2$ and
\begin{equation*}
\|f_1\|_{A^{p_1}_\a}^p\cdot\|f_2\|_{A^{p_2}_\a}^p\le C\|f\|_{A^p_\a}^p
    \end{equation*}
for some constant $C=C(p_1,p_2,\a)>0$.
\end{lettertheorem}
Motivated by the study of integral operators, we are interested in finding out
a large class of
 weights $\om$ which allow a (strong) factorization of $A^p_\om$-functions.
 \par Throughout these notes, we shall use the following notation. For $a\in\D$, define $\vp_a(z)=(a-z)/(1-\overline{a}z)$.
The pseudohyperbolic distance from $z$ to $w$ is defined by
$\varrho(z,w)=|\vp_z(w)|$, and the pseudohyperbolic
disc of center $a\in\D$ and radius $r\in(0,1)$ is denoted by
$\Delta(a,r)=\{z:\varrho(a,z)<r\}$.

\medskip \par A careful inspection of Horowitz's techniques lead us to consider the following class of weights.
  A weight $\om$ (not necessarily radial
neither continuous) is called invariant,  $\om\in{\mathcal Inv}$, if  for each
$r\in(0,1)$ there exists a constant $C=C(r)\ge1$ such that
    \begin{equation}\label{Eq:InvariantWeightsOldDefinition}
    C^{-1}\om(a)\le\om(z)\le C\om(a),\quad\text{ $z\in\Delta(a,r)$.}
    \end{equation}
 We note that a radial weight $\om$ belongs to  ${\mathcal Inv}$ if and only if $\om$ does not have zeros  and $\om$ satisfies the property
\eqref{eq:r2}.
Therefore, $\R\cup\widetilde{\I}\subset {\mathcal Inv}$.
 Moreover, by using results in \cite{AlCo} it is not difficult to prove that a differentiable weight $\om$
is invariant whenever
    \begin{equation*}\label{Pesos:AlemanConstantin}
    |\nabla\om(z)|(1-|z|^2)\le C\om(z),\quad z\in\D.
    \end{equation*}
\par The following result is based on the additivity of the hyperbolic distance on  geodesics.

\begin{lemma}\label{Lemma:InvariantWeights}
If $\om\in{\mathcal Inv}$, then there exists a function
$C:\D\to[1,\infty)$ such that
    \begin{equation}\label{Eq:InvariantWeightsOldDefinition}
    \om(u)\le C(z)\om(\vp_u(z)),\quad u,z\in\D,
    \end{equation}
and
    \begin{equation}\label{Eq:InvariantWeightsOldDefinitionIntegral}
    \int_\D\log C(z)\,dA(z)<\infty.
    \end{equation}

Conversely, if $\om$ is a weight does not have zeros,  satisfying
\eqref{Eq:InvariantWeightsOldDefinition} and the function $C$ is
uniformly bounded in compact subsets of $\D$, then
$\om\in{\mathcal Inv}$.
\end{lemma}
\begin{proof}
Let first $\om\in{\mathcal Inv}$. Then there exists a constant
$C\ge1$ such that
    \begin{equation}\label{111}
    C^{-1}\om(a)\le\om(z)\le C\om(a),\quad z\in \Delta_h(a,1).
    \end{equation}
     For each
$z,u\in\D$, the hyperbolic distance between $u$ and $\vp_u(z)$ is
    $$
    \varrho_h(u,\vp_u(z))=\frac12\log\frac{1+|z|}{1-|z|}.
    $$
By the additivity of the hyperbolic distance on the geodesic
joining $u$ and $\vp_u(z)$, and \eqref{111} we deduce
    $$
    \om(u)\le C^{E(\varrho_h(u,\vp_u(z)))+1}\om(\vp_u(z))\le C\left(\frac{1+|z|}{1-|z|}\right)^{\frac{\log
    C}{2}}\om(\vp_u(z)),
    $$
where $E(x)$ is the integer such that $E(x)\le x<E(x)+1$. It
follows that \eqref{Eq:InvariantWeightsOldDefinition} and
\eqref{Eq:InvariantWeightsOldDefinitionIntegral} are satisfied.

Conversely, let $\om$ be a weight satisfying
\eqref{Eq:InvariantWeightsOldDefinition} such that the function
$C$ is uniformly bounded in compact subsets of $\D$. Then, for
each $r\in(0,1)$, there exists a constant $C=C(r)>0$ such that
$\om(u)\le C(r)\om(z)$ whenever $|\vp_u(z)|<r$. Thus
$\om\in{\mathcal Inv}$.
\end{proof}

The next result plays an important role in the proof of
our factorization theorem. The proof is technical, see \cite[Lemma $3.3$]{PelRat}.

\begin{lemma}\label{Lemma:factorization}
Let $0<p<q<\infty$ and $\omega\in{\mathcal Inv}$. Let $\{z_k\}$ be
the zero set of $f\in A^p_\omega$, and let
    $$
    g(z)=|f(z)|^p\prod_{k}\frac{1-\frac{p}{q}+\frac{p}{q}|\vp_{z_k}(z)|^q}{|\vp_{z_k}(z)|^p}.
    $$
Then there exists a constant $C=C(p,q,\omega)>0$ such that
\begin{equation}\label{constante}
\|g\|_{L^1_\omega}\le C\|f\|_{A^p_\omega}^p.
\end{equation}
Moreover, the constant $C$ has the following properties:
    \begin{enumerate}
    \item[\rm(i)] If $0<p<q\le 2$, then $C=C(\omega)$, that is, $C$ is independent of $p$ and $q$.
    \item[\rm(ii)] If $2<q<\infty$ and $\frac{q}{p}\ge1+\epsilon>1$, then $C=C_1qe^{C_1q}$, where $C_1=C_1(\epsilon,\omega)$.
    \end{enumerate}
\end{lemma}

\par Now, we prove our main result in this section.

\begin{theorem}\label{Thm:FactorizationBergman}
Let $0<p<\infty$ and $\omega\in{\mathcal Inv}$ such that the
polynomials are dense in $A^p_\om$. Let $f\in A^p_\omega$, and let
$0<p_1,p_2<\infty$ such that $p^{-1}=p_1^{-1}+p_2^{-1}$. Then
there exist $f_1\in A^{p_1}_\omega$ and $f_2\in A^{p_2}_\omega$
such that $f=f_1\cdot f_2$ and
    \begin{equation}\label{Eq:NormEstimateForFactorization}
   \|f_1\|_{A^{p_1}_\omega}^p\cdot\|f_2\|_{A^{p_2}_\omega}^p\le\frac{p}{p_1}\|f_1\|_{A^{p_1}_\omega}^{p_1}+\frac{p}{p_2}\|f_2\|_{A^{p_2}_\omega}^{p_2}\le C\|f\|_{A^p_\omega}^p
    \end{equation}
for some constant $C=C(p_1,p_2,\omega)>0$.
\end{theorem}
\begin{proof}
Let $0<p<\infty$ and $\omega\in{\mathcal Inv}$ such that the
polynomials are dense in $A^p_\om$, and let $f\in A^p_\omega$.
Assume first that $f$ has finitely many zeros only. Such functions
are of the form $f=gB$, where $g\in A^p_\omega$ has no zeros and
$B$ is a finite Blaschke product. Let $z_1,\ldots,z_m$ be the
zeros of $f$ so that $B=\prod_{k=1}^mB_k$, where
$B_k=\frac{z_k}{|z_k|}\vp_{z_k}$. Write $B= B^{(1)}\cdot B^{(2)}$,
where the factors $B^{(1)}$ and $B^{(2)}$ are random subproducts
of $B_0,B_1,\ldots,B_m$, where $B_0\equiv1$. Setting
$f_j=\left(\frac{f}{B}\right)^\frac{p}{p_j}B^{(j)}$, we have
$f=f_1\cdot f_2$. We now choose $B^{(j)}$ probabilistically. For a
given $j\in\{1,2\}$, the factor $B^{(j)}$ will contain each $B_k$
with the probability $p/p_j$. The obtained $m$ random variables
are independent, so the expected value of $|f_j(z)|^{p_j}$ is
    \begin{equation}\begin{split}\label{eq:esp2}
    E(|f_j(z)|^{p_j})&=\left|\frac{f(z)}{B(z)}\right|^p\prod_{k=1}^m
    \left(1-\frac{p}{p_j}+\frac{p}{p_j}|\vp_{z_k}(z)|^{p_j}\right)\\
    &=\left|f(z)\right|^p\prod_{k=1}^m
    \frac{\left(1-\frac{p}{p_j}\right)+\frac{p}{p_j}|\vp_{z_k}(z)|^{p_j}}{|\vp_{z_k}(z)|^p}
    \end{split}\end{equation}
for all $z\in\D$ and $j\in\{1,2\}$. Now, bearing in mind
\eqref{eq:esp2} and Lemma~\ref{Lemma:factorization}, we find a
constant $C_1=C_1(p, p_1,\omega)>0$ such that
   \begin{equation*}\begin{split}
   \left\|E\left(f_1^{p_1}\right)\right\|_{L^{1}_\omega}
    & =
    \int_\D\left[
   \left|f(z)\right|^p\prod_{k=1}^m
    \frac{\left(1-\frac{p}{p_1}\right)+\frac{p}{p_1}|\vp_{z_k}(z)|^{p_1}}{|\vp_{z_k}(z)|^p}\right]\,\om(z)dA(z)
    \\ & =
    \int_\D\left[
   \left|f(z)\right|^p\prod_{k=1}^m
    \frac{\frac{p}{p_2}+\left(1-\frac{p}{p_2}\right)|\vp_{z_k}(z)|^{p_1}}{|\vp_{z_k}(z)|^p}\right]\,\om(z)dA(z)\le  C_{1}\|f\|_{A^p_\omega}^p.
   \end{split}\end{equation*}
Analogously, by \eqref{eq:esp2} and
Lemma~\ref{Lemma:factorization} there exists a constant
$C_2=C_{2}(p,p_2,\om)>0$ such that
    \begin{equation*}\begin{split}
   \left\|E\left(f_2^{p_2}\right)\right\|_{L^{1}_\omega}
    & =
    \int_\D\left[
   \left|f(z)\right|^p\prod_{k=1}^m
    \frac{\left(1-\frac{p}{p_2}\right)+\frac{p}{p_2}|\vp_{z_k}(z)|^{p_2}}{|\vp_{z_k}(z)|^p}\right]\,\om(z)dA(z)\le
    C_2\|f\|_{A^p_\omega}^p.
      \end{split}\end{equation*}
By combining the two previous inequalities, we obtain
 \begin{equation}\begin{split}\label{eq:lfn3}
    \left\|E\left(\frac{p}{p_1}f_1^{p_1}\right)\right\|_{L^{1}_\omega}
    +\left\|E\left(\frac{p}{p_2}f_2^{p_2}\right)\right\|_{L^{1}_\omega}
    \le\left(\frac{p}{p_1}C_{1}+\frac{p}{p_2}C_{2}\right)\|f\|_{A^p_\omega}^p.
      \end{split}\end{equation}
On the other hand,
      \begin{equation}\begin{split}\label{eq:lfn4}
   &  \left\|E\left(\frac{p}{p_1}f_1^{p_1}\right)\right\|_{L^{1}_\omega}
   +\left\|E\left(\frac{p}{p_2}f_2^{p_2}\right)\right\|_{L^{1}_\omega}
   \\ & =\frac{p}{p_1}\int_\D \left|\frac{f(z)}{B(z)}\right|^p
   \prod_{k=1}^m
    \left(\frac{p}{p_2}+\left(1-\frac{p}{p_2}\right)|\vp_{z_k}(z)|^{p_1}\right)\,\om(z)dA(z)
    \\ &   \quad+\frac{p}{p_2}\int_\D \left|\frac{f(z)}{B(z)}\right|^p
   \prod_{k=1}^m
    \left(\left(1-\frac{p}{p_2}\right)+\frac{p}{p_2}|\vp_{z_k}(z)|^{p_2}\right)
   \,\om(z)dA(z)\\
   &=\int_\D I(z)\omega(z)\,dA(z),
      \end{split}\end{equation}
where
    \begin{equation*}\begin{split}
    I(z)=&\left|\frac{f(z)}{B(z)}\right|^p\Bigg[\frac{p}{p_1}\cdot\prod_{k=1}^m
    \left(\frac{p}{p_2}+\left(1-\frac{p}{p_2}\right)|\vp_{z_k}(z)|^{p_1}\right)
    \\
    & \quad +\frac{p}{p_2}\cdot\prod_{k=1}^m
    \left(\left(1-\frac{p}{p_2}\right)+\frac{p}{p_2}|\vp_{z_k}(z)|^{p_2}\right)\Bigg].
    \end{split}\end{equation*}
It is clear that the $m$ zeros of $f$ must be distributed to the
factors $f_1$ and $f_2$, so if $f_1$ has $n$ zeros, then $f_2$ has
the remaining $(m-n)$ zeros. Therefore
 \begin{equation}\label{idez}
    I(z)=\sum_{f_{l_1}\cdot f_{l_2}=f} \left(\left(1-\frac{p}{p_2}\right)^n\left(\frac{p}{p_2}\right)^{m-n}\left[\frac{p}{p_1}
|f_{l_1}(z)|^{p_1}+\frac{p}{p_2} |f_{l_2}(z)|^{p_2}
\right]\right).
 \end{equation}
This sum consists of $2^m$ addends, $f_{l_1}$ contains
$\left(\frac{f}{B}\right)^\frac{p}{p_1}$ and $n$ zeros of $f$, and
$f_{l_2}$ contains $\left(\frac{f}{B}\right)^\frac{p}{p_2}$ and
the remaining $(m-n)$ zeros of $f$, and thus $f=f_{l_1}\cdot
f_{l_2}$. Further, for a fixed $n=0,1,\ldots,m$, there are
$({m\atop n})$ ways to choose $f_{l_1}$ (once $f_{l_1}$ is chosen,
$f_{l_2}$ is determined). Consequently,
    \begin{equation}\begin{split}\label{eq:lfn5}
    \sum_{f_{l_1}\cdot f_{l_2}=f} \left(1-\frac{p}{p_2}\right)^n\left(\frac{p}{p_2}\right)^{m-n}
    =\sum_{n=0}^m\left({m\atop n}\right)\left(1-\frac{p}{p_2}\right)^n\left(\frac{p}{p_2}\right)^{m-n}=1.
    \end{split}\end{equation}
Now, by joining \eqref{eq:lfn3}, \eqref{eq:lfn4} and \eqref{idez},
we deduce
    \begin{equation*}\begin{split}
    &\sum_{f_{l_1}\cdot f_{l_2}=f}
    \left(1-\frac{p}{p_2}\right)^n\left(\frac{p}{p_2}\right)^{m-n}
    \left[\frac{p}{p_1}\|f_{l_1}\|_{A^{p_1}_\om}^{p_1}+\frac{p}{p_2}
    \|f_{l_2}\|_{A^{p_2}_\om}^{p_2} \right]\\
    &\quad\le\left(\frac{p}{p_1}C_{1}+\frac{p}{p_2}C_{2}\right)\|f\|_{A^p_\omega}^p.
    \end{split}\end{equation*}
This together with \eqref{eq:lfn5} shows that there must exist a
concrete factorization $f=f_1\cdot f_2$ such that
    \begin{equation}
    \begin{split}
    \label{eq:lfn6}
    \frac{p}{p_1}\|f_{1}\|_{A^{p_1}_\om}^{p_1}+\frac{p}{p_2}
    \|f_{2}\|_{A^{p_2}_\om}^{p_2} \le
    C(p_1,p_2,\om)\|f\|_{A^p_\omega}^p.
    \end{split}
    \end{equation}
By combining this with the inequality
     $$
     x^\alpha\cdot y^\beta\le \alpha x+\beta y,\quad x,y\ge 0,\quad \alpha+\beta=1,
     $$
we finally obtain \eqref{Eq:NormEstimateForFactorization} under
the hypotheses that $f$ has finitely many zeros only.

To deal with the general case, we first prove that every
norm-bounded family in $A^p_\omega$ is a normal family of analytic
functions. If $f\in A^p_\omega$, then
    \begin{equation}\label{Eq:RadialGrowthNEW}
    \begin{split}
    \|f\|_{A^p_\omega}^p&\ge\int_{D(0,\frac{1+\rho}{2})\setminus
    D(0,\rho)}|f(z)|^p\omega(z)\,dA(z)\\
    &\gtrsim
    M_p^p(\rho,f)\left(\min_{|z|\le\frac{1+\rho}{2}}\om(z)\right),\quad 0\le\rho<1,
    \end{split}
    \end{equation}
from which the well-known relation $M_\infty(r,f)\lesssim
M_p(\frac{1+r}{2},f)(1-r)^{-1/p}$ yields
    \begin{equation}\label{20NEW}
    M^p_\infty(r,f)
    \lesssim\frac{\|f\|_{A^p_\omega}^p}{(1-r)\left(\min_{|z|\le\frac{3+r}{4}}\om(z)\right)},\quad 0\le r<1.
    \end{equation}
Therefore every norm-bounded family in $A^p_\omega$ is a normal
family of analytic functions by Montel's theorem.

Finally, assume that $f\in A^p_\omega$ has infinitely many zeros.
Since polynomials are dense in $A^p_\omega$ by the assumption, we
can choose a sequence $f_l$ of functions with finitely many zeros
that converges to $f$ in norm, and then, by the previous argument,
we can factorize each $f_l=f_{l,1}\cdot f_{l,2}$ as earlier. Now,
since every norm-bounded family in~$A^p_\omega$ is a normal family
of analytic functions, by passing to subsequences of $\{f_{l,j}\}$
with respect to $l$ if necessary, we have $f_{l,j}\to f_j$, where
the functions $f_j$ form the desired bounded factorization $f=
f_1\cdot f_2$ satisfying \eqref{Eq:NormEstimateForFactorization}.
This finishes the proof.
\end{proof}

\par At first glance the next result might seem a bit artificial.
However, it turns out to be a key ingredient in the proof of Proposition~\ref{PropSmallIndeces} (below) where we get the
uniform boundedness of a certain family of integral operators, which  is usually established by using interpolation theorems.

\begin{corollary}\label{cor:FactorizationBergman}\index{factorization}
Let $0<p<2$ and $\omega\in{\mathcal Inv}$ such that the
polynomials are dense in $A^p_\om$.\index{${\mathcal
Inv}$}\index{invariant weight} Let $0<p_1\le2<p_2<\infty$ such
that $\frac1p=\frac1{p_1}+\frac1{p_2}$ and $p_2\ge2p$. If $f\in
A^p_\omega$, then there exist $f_1\in A^{p_1}_\omega$ and $f_2\in
A^{p_2}_\omega$ such that $f=f_1\cdot f_2$ and
    \begin{equation}\label{Eqco:NormEstimateForFactorization}
    \|f_1\|_{A^{p_1}_\omega}\cdot \|f_2\|_{A^{p_2}_\omega}\le C\|f\|_{A^p_\om}
    \end{equation}
for some constant $C=C(p_1,\om)>0$.
\end{corollary}
\par It can be proved mimicking the proof of of Theorem~\ref{Thm:FactorizationBergman},
but  paying special attention to the constants coming
from Lemma~\ref{Lemma:factorization}, see \cite[Corollary~3.4]{PelRat} for details.
\par Before ending this section, let us observe that there are non-radial weights satisfying the hypotheses of our factorization result for $A^p_\om$.
\begin{lemma}
Let $f$ be a non-vanishing univalent function in $\D$, $0<\gamma<1$ and $\om=|f|^\g$. Then   the
polynomials are dense in $A^p_\om$ for all $p\ge 1$.
\end{lemma}
\begin{proof}
Since $f$ is univalent and zero-free, so is $1/f$, and hence both
$f$ and $1/f$ belong to $A^p$ for all $0<p<1$. By choosing
$\delta>0$ such that $\gamma(1+\delta)<1$ we deduce that both
$\om$ and $\frac{1}{\om}$ belong to $L^{1+\delta}$. Therefore the
polynomials are dense in~$A^p_\om$ by \cite[Theorem~2]{Hedberg}.
\end{proof}
Finally, let us consider the class of weights that appears in a
paper by Abkar~\cite{Abkar1} concerning norm approximation by polynomials in  weighted Bergman spaces. A~function $u$ defined on $\D$ is said to be
\emph{superbiharmonic} if $\Delta^2u\ge
0$, where $\Delta$ stands for the Laplace
operator
    $$
    \Delta=\Delta_z=\frac{\partial^2}{\partial z\partial \overline{z}}=\frac{1}{4} \left(\frac{\partial^2}{\partial^2 x}
         +\frac{\partial^2}{\partial^2 y}
         \right)
    $$
in the complex plane $\mathbb{C}$. The superbihamonic weights play
an essential role in the study of invariant subspaces of the
Bergman space $A^p$.

\begin{lettertheorem}\label{th:Abkar}
Let $\om$ be a superbiharmonic weight such that
    \begin{equation}\label{Abkarii}
    \lim_{r\to 1^-}\int_{\T}\om(r\zeta)\,dm(\zeta)=0.
    \end{equation}
Then the polynomials are dense in~$A^p_\om$.
\end{lettertheorem}

The proof of Theorem~\ref{th:Abkar} relies on showing that these
type of weights $\om$ satisfy \begin{equation}\label{conditionrz}
    \om(z)\le C(\om)\om(rz),\quad r_0\le r<1,\quad r_0\in (0,1),
    \end{equation}
which asserts that polynomials are dense on $A^p_\om$.
In \cite[Lemma $1.11$]{PelRat} it is proved the following.

\begin{lemma}
Every superbihamonic weight that satisfies $\lim_{r\to 1^-}\int_{\T}\om(r\zeta)\,dm(\zeta)=0$, is
invariant and the polynomials are dense in~$A^p_\om$.
\end{lemma}

\section{Zero sets}\label{sec:zeros}
For a given space $X$ of analytic functions in $\D$, a sequence
$\{z_k\}$ is called an $X$-zero set, if there exists a
function $f$ in $X$ such that $f$ vanishes precisely on the points
$\{z_k\}$ and nowhere else.  A sequence $\{z_k\}$  is a $H^p$-zero set if and only if   satisfies the Blaschke condition
$\sum_k(1-|z_k|)<\infty$.
\subsection{The Bergman-Nevanlinna class}
\par Using Lemma~\ref{Lemma:weights-in-D-hat}, Jensen's formula and the elementary factors from the classical Weierstrass factorization for the theory of entire functions, it can be proved
the following \cite[Proposition $3.16$]{PelRat}. The weighted Bergman-Nevanlinna class
  consists of those analytic functions in $\D$ for
which
    $$
    \int_\D\log^+|f(z)|\om(z)\,dA(z)<\infty.
    $$
\begin{theorem}
Let $\om\in\DD$. Then $\{z_k\}$ is a zero set of the Bergman-Nevanlinna class
 if
and only if
    \begin{equation*}\label{70}
    \sum_k\left[(1-|z_k|)\widehat{\om}(z_k)\right]=\sum_k\left[(1-|z_k|)\int_{|z_k|}^1\om(s)\,ds\right]<\infty.
    \end{equation*}
\end{theorem}

As far as we know, it is still an open
problem to find a complete description of zero sets of functions
in the Bergman spaces $A^p=A^p_0$, but the gap between the known
necessary and sufficient conditions is very small.  We refer
to~\cite[Chapter~4]{DurSchus}, \cite[Chapter~4]{HKZ} and \cite{Kor,Lzeros96,S1,S2}.
The analogous question is also unsolved for classical Dirichlet spaces $\mathcal{D}^2_\alpha$, $0\le \alpha<1$, of $f\in \H(\D)$ such that
$$\|f\|^2_{\mathcal{D}^2_\alpha}=|f(0)|^2+\int_\D|f'(z)|^2(1-|z|)^\a\,dA(z)<\infty.$$
The most important results are the ones given by
Carleson in \cite{Ctesis}, \cite{C}, and by Shapiro and Shields in
\cite{SS}. Some progress was achieved in \cite{PPzeros}.
\subsection{$A^p_\om$ zeros sets}
\par Our
results on zeros set of $A^p_\om$ follow the line of those due to
Horowitz~\cite{Horzeros,Horzeros1,Horzeros2}. Roughly speaking we
will study basic properties of unions, subsets and the dependence
on $p$ of the zero sets of functions in $A^p_\om$.  By using ideas and estimates obtained in the
proof of Theorem \ref{Thm:FactorizationBergman} we get our first result in this section, see \cite[Theorem $3.5$]{PelRat}.

\begin{theorem}
Let $0<p<\infty$ and $\omega\in{\mathcal Inv}$. Let $\{z_k\}$ be an arbitrary
subset of the zero set of $f\in A^p_\omega$, and let
    $$
    H(z)=\prod_{k}B_{k}(z)(2-B_{k}(z)),\quad
    B_k=\frac{z_k}{|z_k|}\vp_{z_k},
    $$
with the convention $z_k/|z_k|=1$ if $z_k=0$. Then there exists a
constant $C=C(\omega)>0$ such that $\|f/H\|_{A^p_\omega}^p\le
C\|f\|_{A^p_\omega}^p$. In particular, each subset of an
$A^p_\omega$-zero set is an $A^p_\omega$-zero set.
\end{theorem}

Now we turn to work with radial weights. The first of them will be used to show that $A^p_\om$-zero
sets\index{$A^p_\om$-zero set} depend on $p$.

\begin{theorem}\label{Theorem:ZerosBergman1}
Let $0<p<\infty$ and let $\omega$ be a radial weight. Let $f\in
A^p_\omega$, $f(0)\ne0$, and let $\{z_k\}$ be its zero sequence
repeated according to multiplicity and ordered by increasing
moduli. Then
    \begin{equation}\label{j10}
    \prod_{k=1}^n\frac{1}{|z_k|}=\op\left(\left(\int_{1-\frac{1}n}^1\omega(r)\,dr\right)^{-\frac1p}\right),\quad
    n\to \infty.
    \end{equation}
\end{theorem}

\begin{proof}
Let $f\in A^p_\om$ and $f(0)\ne0$. By multiplying Jensen's formula
\begin{equation}\label{Eq:Jensen-Formula}
    \log|f(0)|+\sum_{k=1}^n\log\frac{r}{|z_k|}=\frac{1}{2\pi}\int_0^{2\pi}\log|f(re^{i\t})|d\t,\quad0<r<1,
    \end{equation}
     by~$p$, and
applying the arithmetic-geometric mean inequality, we obtain
    \begin{equation}\label{44}
    |f(0)|^p\prod_{k=1}^n\frac{r^p}{|z_k|^p}\le M^p_p(r,f)
    \end{equation}
for all $0<r<1$ and $n\in\N$. Moreover,
    \begin{equation}\label{j20}
    \lim_{r\to 1^-}M^p_p(r,f)\int_r^1\om(s)\,ds\le \lim_{r\to 1^-}
    \int_r^1M^p_p(s,f)\om(s)\,ds=0,
    \end{equation}
so taking $r=1-\frac{1}{n}$ in~\eqref{44}, we deduce
    $$
    \prod_{k=1}^n\frac{1}{|z_k|}\lesssim M_p\left(1-\frac{1}{n},f\right)
    =\op\left(\left(\int_{1-\frac{1}n}^1\omega(r)\,dr\right)^{-\frac1p}\right),\quad
    n\to \infty,
    $$
as desired.
\end{proof}

The next result  shows that
condition~\eqref{j10} is a sharp necessary condition for $\{z_k\}$
to be an $A^p_\om$-zero set.

\begin{theorem}\label{th:zerosqp}
Let $0<q<\infty$ and $\om\in\DD$. Then there exists
$f\in\cap_{p<q}A^p_\om$ such that its zero sequence $\{z_k\}$,
repeated according to multiplicity and ordered by increasing
moduli, does not satisfy~\eqref{j10} with $p=q$. In particular,
there is a $\cap_{p<q} A^p_\om$-zero set which is not an
$A^q_\om$-zero set.
\end{theorem}

\begin{proof} The proof uses ideas
from~\cite[Theorem~3]{GNW}, see also \cite{Horzeros1,Horzeros2}.
Define
    \begin{equation}\label{zqp2}
    f(z)=\prod_{k=1}\sp\infty F_k(z),\quad z\in\D,
    \end{equation}
where
    $$
    F_k(z)=\frac{ 1+a_k z^{2^k}}{1+a^{-1}_k z^{2^k}},\quad z\in\D,\quad
    k\in\N,
    $$
and
    $$
    a_k=\left(\frac{\int_{1-2^{-k}}^1\om(s)\,ds}{\int_{1-2^{-(k+1)}}^1\om(s)\,ds}\right)^{1/q},\quad
    k\in\N.
    $$
By Lemma~\ref{Lemma:weights-in-D-hat} there exists a constant
$C_1=C_1(q,\om)>0$ such that
    \begin{equation}\label{zqp1}
    1<a_k\le C_1<\infty,\quad k\in\N.
    \end{equation}
Therefore $\limsup_{k\to \infty}(a_k-a_k^{-1})^{2^{-k}}\le
\limsup_{k\to \infty}a_k^{2^{-k}}=1$, and hence the product
in~\eqref{zqp2} defines an analytic function in $\D$. The zero set
of $f$ is the union of the zero sets of the functions $F_k$, so
$f$ has exactly $2^k$ simple zeros on the circle
$\left\{z:|z|=a_k^{-2^{-k}}\right\}$ for each $k\in\N$. Let
$\{z_j\}_{j=1}^\infty$ be the sequence of zeros of $f$ ordered by
increasing moduli, and denote $N_n=2+2^2+\cdots+2^n$. Then $2^n\le
N_n\le 2^{n+1}$, and hence
    \begin{equation*}
    \prod_{k=1}^{N_n}\frac{1}{|z_k|}\ge\prod_{k=1}^n a_k
    =\left(\frac{\int_{\frac12}^1\om(s)\,ds}{\int_{1-2^{-(n+1)}}^1\om(s)\,ds}\right)^{1/q}
    \ge\left(\frac{\int_{\frac12}^1\om(s)\,ds}{\int_{1-\frac{1}{N_n}}^1\om(s)\,ds}\right)^{1/q}.
    \end{equation*}
It follows that $\{z_j\}_{j=1}^\infty$ does not
satisfy~\eqref{j10}, and thus $\{z_j\}_{j=1}^\infty$ is not an
$A^q_\om$-zero set by Theorem~\ref{Theorem:ZerosBergman1}.

We turn to prove that the function $f$ defined in \eqref{zqp2}
belongs to $A^p_\om$ for all $p\in(0,q)$. Set $r_n=e\sp{-2^{-n}}$
for $n\in\N$, and observe that
    \begin{equation}\label{zqp4}
    |f(z)|=\left| \prod_{k=1}^n a_k \frac{a_k^{-1}+z^{2^k}}{1+a^{-1}_k z^{2^k}}\right| \left| \prod_{j=1}^\infty \frac{ 1+a_{n+j} z^{2^{n+j}}}{1+a^{-1}_{n+j}
    z^{2^{n+j}}}\right|.
    \end{equation}
The function $h_1(x)=\frac{\a+x}{1+\a x}$ is increasing on $[0,1)$
for each $\a\in[0,1)$, and therefore
    \begin{equation}\label{67}
    \begin{split}
    \left|\frac{ 1+a_{n+j} z^{2^{n+j}}}{1+a^{-1}_{n+j} z^{2^{n+j}}}\right|&=a_{n+j}
    \left|\frac{ a^{-1}_{n+j}+ z^{2^{n+j}}}{1+a^{-1}_{n+j} z^{2^{n+j}}}\right|
    \le a_{n+j} \frac{ a^{-1}_{n+j}+ |z|^{2^{n+j}}}{1+a^{-1}_{n+j} |z|^{2^{n+j}}}\\
    &\le\frac{1+a_{n+j}\left(\frac1e\right)^{2^{j}}}{1+a^{-1}_{n+j}
    \left(\frac1e\right)^{2^{j}}},\quad |z|\le r_n,\quad
    j,\,n\in\N.
    \end{split}
    \end{equation}
Since $h_2(x)=\frac{1+x\alpha}{1+x^{-1}\alpha}$ is increasing on
$(0,\infty)$ for each $\alpha\in(0,\infty)$, \eqref{zqp1} and
\eqref{67} yield
    \begin{equation}
    \begin{split}\label{zqp5}
    \left|\prod_{j=1}^\infty\frac{1+a_{n+j} z^{2^{n+j}}}{1+a^{-1}_{n+j} z^{2^{n+j}}}\right|
    &\le\prod_{j=1}^\infty\frac{ 1+a_{n+j} \left(\frac1e\right)^{2^{j}}}{1+a^{-1}_{n+j}\left(\frac1e\right)^{2^{j}}}
    \le\prod_{j=1}^\infty\frac{1+C_1\left(\frac1e\right)^{2^{j}}}{1+C_1^{-1}\left(\frac1e\right)^{2^{j}}}
    =C_2<\infty,
    \end{split}
    \end{equation}
whenever $|z|\le r_n$ and $n\in\N$. So, by using \eqref{zqp4},
\eqref{zqp5}, Lemma~\ref{Lemma:weights-in-D-hat} and the inequality
$e^{-x}\ge 1-x$, $x\ge0$, we obtain
    \begin{equation}
    \begin{split}\label{zqp6}
    |f(z)|&\le C_2\prod_{k=1}^n
    a_k\lesssim\left(\frac{1}{\int_{1-2^{-(n+1)}}^1\om(s)\,ds}\right)^{1/q}\lesssim\left(\frac{1}{\int_{1-2^{-n}}^1\om(s)\,ds}\right)^{1/q}\\
    &\le \left(\frac{1}{\int_{r_n}^1\om(s)\,ds}\right)^{1/q},\quad
    |z|\le r_n,\quad n\in\N.
    \end{split}
    \end{equation}
Let now $|z|\ge1/\sqrt{e}$ be given and fix $n\in\N$ such that
$r_n\le|z|<r_{n+1}$. Then \eqref{zqp6}, the inequality $1-x\le
e^{-x}\le 1-\frac{x}{2}$, $x\in[0,1]$, and Lemma~\ref{Lemma:weights-in-D-hat}
give
    \begin{equation*}
    \begin{split}
    |f(z)|&\le
    M_\infty(r_{n+1},f)\lesssim\left(\frac{1}{\int_{r_{n+1}}^1\om(s)\,ds}\right)^{1/q}\\
    &\le \left(\frac{1}{\int_{1-2^{-(n+2)}}^1\om(s)\,ds}\right)^{1/q}
    \lesssim\left(\frac{1}{\int_{1-2^{-n}}^1\om(s)\,ds}\right)^{1/q}\\
    &\le\left(\frac{1}{\int_{r_n}^1\om(s)\,ds}\right)^{1/q}\le \left(\frac{1}{\int_{|z|}^1\om(s)\,ds}\right)^{1/q},
    \end{split}
    \end{equation*}
and hence
    \begin{equation*}
    M_\infty(r,f)\lesssim\left(\frac{1}{\int_{r}^1\om(s)\,ds}\right)^{1/q},\quad 0<r<1.
    \end{equation*}
This and the identity
$\psi_{\widetilde{\om}}(r)=\frac{1}{1-\a}\psi_\om(r)$\index{$\widetilde{\om}(r)$}
of Lemma~\ref{le:RAp}(iii), with $\a=p/q<1$ and $r=0$, yield
    $$
    \|f\|^p_{A^p_\om}\lesssim\int_0^1\frac{\om(r)\,dr}{\left(\int_{r}^1\om(s)\,ds\right)^{p/q}}
    =\int_0^1\widetilde{\om}(r)\,dr=\frac{q}{q-p}\left(\int_{0}^1\om(s)\,ds\right)^{\frac{q-p}{q}}<\infty.
    $$
This finishes the proof.
\end{proof}

The proof of the above result implies that the union of two
$A^p_\omega$-zero sets is not an $A^p_\om$-zero set. Going further,
we obtain the following
result.

\begin{corollary}\label{co:unioneszeros}
Let $0<p<\infty$ and $\om\in\DD$. Then the union of two
$A^p_\omega$-zero sets is an $A^{p/2}_\om$-zero set. However,
there are two $\cap_{p<q} A^p_\om$-zero sets such that their union
is not an $A^{q/2}_\om$-zero set.
\end{corollary}

Since the angular distribution of zeros plays a role in a
description of the zero sets of functions in the classical
weighted Bergman space~$A^p_\alpha$, it is natural to expect that
the same happens also in $A^p_\om$, when $\om\in\DD$. However, we do
not venture into generalizing the theory, developed among others
by Korenblum~\cite{Kor}, Hedenmalm~\cite{HedStPeter} and
Seip~\cite{S1,S2}, and based on the use of densities defined in
terms of partial Blaschke sums, Stolz star domains and
Beurling-Carleson characteristic of the corresponding boundary
set.

\section{Integral operators}\label{sec:integral}
The main aim of this section is to characterize those symbols
$g\in\H(\D)$ such that the integral operator
    \begin{displaymath}
    T_g(f)(z)=\int_{0}^{z}f(\zeta)\,g'(\zeta)\,d\zeta,\quad
    z\in\D,
    \end{displaymath}
is bounded or compact from $A^p_\om$ to $A^q_\om$, when $\om\in\DD$. The choice
$g(z)=z$ gives the usual Volterra operator and the Ces\`{a}ro
operator is obtained when $g(z)=-\log(1-z)$. The bilinear operator $\left(f,g\right)\rightarrow \int f\,g'$  was introduced by  A. Calder\'on in harmonic analysis in the $60$'s for his research on commutators of singular integral operators \cite{Calderon65} which leads to the study of  \lq\lq paraproducts\rq\rq.
Regarding the complex function theory, Pommerenke  considered the operator $T_g$ \cite{Pom} to study the space $BMOA$ proving
that $T_g: H^2\to H^2$ is bounded if and only $g\in \BMOA$. We recall that $\BMOA$
consists of functions in the Hardy space $H^1$ that have
\emph{bounded mean oscillation}
on the boundary $\T$~\cite{Ba86(2),GiBMO}.
 We will use the norm given by
    $$
    \|g\|^2_{\BMOA}=\sup_{a\in\D}\frac{\int_{S(a)}|g'(z)|^2(1-|z|^2)\,dA(z)}{1-|a|}+|g(0)|^2.
    $$
Later, Aleman and Cima~\cite{AC} proved that $T_g:H^p\to H^p$ is bounded if and only if $g\in \BMOA$.
The analogue holds for $A^p_\om$, $\om\in\R$, if and only if $g\in\B$ ~\cite{AS}. Recently, the spectrum
of~$T_g$  has been studied on the Hardy space $H^p$~\cite{AlPe} and on the classical
weighted Bergman space $A^p_\alpha$~\cite{AlCo}. The following family of spaces
of analytic functions will appear in the description of those symbols $g$
such that $T_g:\,A^p_\om\to A^q_\om$ is bounded.
\subsection{Non-conformally Invariant Spaces}
We say that $g\in\H(\D)$ belongs to $\CC^{q,\,p}(\om^\star)$,
$0<p,q<\infty$, if the measure $|g'(z)|^2\om^\star(z)\,dA(z)$ is a
$q$-Carleson measure for $A^p_\om$.  If
$q\ge p$ and $\om\in\DD$, then Theorem~\ref{th:cm} shows that
these spaces only depend on the quotient $\frac{q}{p}$.
Consequently, for $q\ge p$ and $\om\in\DD$, we simply write
$\CC^{q/p}(\om^\star)$ instead of $\CC^{q,\,p}(\om^\star)$. Thus,
if $\alpha\ge 1$ and $\om\in\DD$, then
$\CC^{\alpha}(\om^\star)$ consists of those $g\in\H(\D)$ such
that
    \begin{equation}\label{calpha}
    \|g\|^2_{\CC^{\alpha}(\om^\star)}=|g(0)|^2+\sup_{I\subset\T}\frac{\int_{S(I)}|g'(z)|^2\om^\star(z)\,dA(z)}
    {\left(\om\left(S(I)\right)\right)^{\alpha}}<\infty.\index{$\CC^\a(\om^\star)$}
    \end{equation}

Unlike~$\B$, the space
$\CC^1(\om^\star)$ can not be described
by a simple growth condition on the maximum modulus of $g'$ if
$\om\in\DD$. This follows by Proposition~\ref{pr:blochcpp} (below) and the
fact that $\log(1-z)\in A^p_\om$ for all $\om\in\DD$.

The spaces $\BMOA$ and $\B$ are
conformally invariant. This property has been used, among other
things, in describing those symbols $g\in\H(\D)$ for which $T_g$
is bounded on $H^p$ or $A^p_\alpha$. However, the space
$\CC^1(\omega^\star)$ is not necessarily
conformally invariant, and therefore
different techniques must be employed in the case of $A^p_\om$
with $\om\in\DD$.

Recall that $h:\,[0,1)\to (0,\infty)$ is
essentially increasing on $[0,1)$ if there exists a constant $C>0$
such that $h(r)\le C h(t)$ for all $0\le r\le t<1$.

\begin{proposition}\label{pr:blochcpp}
\begin{itemize}
\item[\rm(A)] If $\om\in\DD$, then
$\CC^1(\omega^\star)\subset\cap_{0<p<\infty}A^p_\omega$.\index{$\CC^1(\om^\star)$}

\item[\rm(B)] If $\om\in\DD$, then $\BMOA\subset
\CC^1(\omega^\star)\subset\B$.

\item[\rm(C)] If $\om\in\R$, then $\CC^1(\omega^\star)=\B$.

\item[\rm(D)] If $\om\in\I$, then
$\CC^1(\omega^\star)\subsetneq\B$.

\item[\rm(E)] If $\om\in\I$ and both $\om(r)$ and
$\frac{\psi_\om(r)}{1-r}$ are essentially increasing on $[0,1)$,
then
$\BMOA\subsetneq\CC^1(\omega^\star)$.\index{$\BMOA$}\index{$\B$}
\end{itemize}
\end{proposition}

\begin{proof} (A). Let $g\in\CC^1(\omega^\star)$.\index{$\CC^1(\om^\star)$} Theorem~\ref{th:cm} shows that
$|g'(z)|^2\omega^{\star}(z)\,dA(z)$ is a $p$-Carleson measure for
$A^p_\omega$ for all $0<p<\infty$. In particular,
$|g'(z)|^2\omega^{\star}(z)\,dA(z)$ is a finite measure and hence
$g\in A^2_\omega$ by \eqref{eq:LP2}. Therefore \eqref{HSB} yields
    \begin{equation*}
    \|g\|_{A^4_\omega}^4=4^2\int_{\D}|g(z)|^{2}|g'(z)|^2\omega^\star(z)\,dA(z)+|g(0)|^4\lesssim\|g\|_{A^2_\omega}^2+|g(0)|^4,
    \end{equation*}
and thus $g\in A^4_\omega$. Continuing in this fashion, we deduce
$g\in A^{2n}_\omega$ for all $n\in\N$, and the assertion follows.

(B). If $g\in\BMOA$,\index{$\BMOA$} then
$|g'(z)|^2\log\frac{1}{|z|}\,dA(z)$ is a classical Carleson
measure~\cite{Garnett1981} (or \cite[Section~8]{GiBMO}), that is,
    $$
    \sup_{I\subset\T}\frac{\int_{S(I)}|g'(z)|^2\log\frac{1}{|z|}\,dA(z)}{|I|}<\infty.
    $$
Therefore
    \begin{equation*}
    \begin{split}
    \int_{S(I)}|g'(z)|^2\om^\star(z)\,dA(z)
    &\le\int_{S(I)}|g'(z)|^2\log\frac{1}{|z|}\left(\int_{|z|}^1\om(s)s\,ds\right)\,dA(z)\\
    &\le\left(\int_{1-|I|}^1\om(s)s\,ds\right)\int_{S(I)}|g'(z)|^2\log\frac{1}{|z|}\,dA(z)\\
    &\lesssim\left(\int_{1-|I|}^1\om(s)s\,ds\right)|I|\asymp\om\left(S(I)\right),
    \end{split}
    \end{equation*}
which together with Theorem~\ref{th:cm} gives
$g\in\CC^1(\omega^\star)$ for all
$\om\in\DD$.\index{$\CC^1(\om^\star)$}

Let now $g\in\CC^1(\omega^\star)$ with $\om\in\DD$. It is
well known that $g\in\H(\D)$ is a Bloch function if and only if
    $$
    \int_{S(I)}|g'(z)|^2(1-|z|^2)^\gamma\,dA(z)\lesssim|I|^\gamma,\quad
    I\subset\T,
    $$
for some (equivalently for all) $\gamma>1$. Fix
$\b=\b(\omega)>0$ and $C=C(\b,\om)>0$ as in
Lemma~\ref{Lemma:weights-in-D-hat}(ii). Then \eqref{3} and
Lemma~\ref{Lemma:weights-in-D-hat}(ii) yield
    \begin{equation*}
    \begin{split}
    \int_{S(I)}|g'(z)|^2(1-|z|)^{\b+1}\,dA(z)
    &=\int_{S(I)}|g'(z)|^2\om^\star(z)\frac{(1-|z|)^{\b+1}}{\om^\star(z)}\,dA(z)\\
    &\asymp\int_{S(I)}|g'(z)|^2\om^\star(z)\frac{(1-|z|)^{\b}}{\int_{|z|}^1\om(s)s\,ds}\,dA(z)\\
    &\lesssim\frac{|I|^\b}{\int_{1-|I|}^1\om(s)s\,ds}\int_{S(I)}|g'(z)|^2\om^\star(z)\,dA(z)\\
    &\lesssim |I|^{\b+1},\quad |I|\le\frac12,
    \end{split}\index{$\CC^1(\om^\star)$}
    \end{equation*}
and so $g\in\B$.

(C). By Part (B) it suffices to show that
$\B\subset\CC^1(\omega^\star)$ for $\om\in\R$. To see
this, let $g\in\B$ and $\om\in\R$. Let us consider the weight
$\tilde{\om}(r)=\frac{\widehat{\om}(r)}{1-r}$. Since $\om\in\R$, $\tilde{\om}(r)$ is a continuous weight such that
$$C_1\le \frac{\psi_{\tilde{\om}}(r)}{1-r}\le C_2,\quad 0<r<1.$$
 A calculation shows that
$\tilde{h}(r)=\frac{\int_r^1\tilde{\om}(s)\,ds}{(1-r)^\a}$, $\a=\frac{1}{C_2}$, is decreasing on
$[0,1)$. So, $h(r)=\frac{\int_r^1\om(s)\,ds}{(1-r)^\a}$ is essentially decreasing on
$[0,1)$
 This together with \eqref{3} gives
    \begin{equation*}\index{$\CC^1(\om^\star)$}
    \begin{split}
    \int_{S(I)}|g'(z)|^2\om^\star(z)\,dA(z)
    &=\int_{S(I)}|g'(z)|^2\frac{\om^\star(z)}{(1-|z|)^{\alpha+1}}(1-|z|)^{\alpha+1}\,dA(z)\\
    &\asymp\int_{S(I)}|g'(z)|^2\frac{\int_{|z|}^1\om(s)s\,ds}{(1-|z|)^{\alpha}}(1-|z|)^{\alpha+1}\,dA(z)\\
    &\lesssim\frac{\int_{1-|I|}^1\om(s)\,ds}{|I|^\alpha}\int_{S(I)}|g'(z)|^2(1-|z|)^{\alpha+1}\,dA(z)\\
    &\lesssim\om\left(S(I)\right),\quad |I|\le\frac12,
    \end{split}
    \end{equation*}
and therefore $g\in\CC^1(\omega^\star)$.

(D). Let $\om\in\I$, and assume on the contrary to the assertion
that $\B\subset\CC^1(\omega^\star)$. Ramey and
Ullrich~\cite[Proposition~5.4]{RU} constructed
$g_1,g_2\in\B$\index{$\B$} such that
$|g'_1(z)|+|g'_2(z)|\ge(1-|z|)^{-1}$ for all $z\in\D$. Since
$g_1,g_2\in\CC^1(\omega^\star)$ by the antithesis,
\eqref{3} yields
    \begin{equation}\label{56}\index{$\CC^1(\om^\star)$}
    \begin{split}
    \|f\|_{A^2_\om}^2&\gtrsim\int_{\D}|f(z)|^2\left(|g'_1(z)|^2+|g'_2(z)|^2\right)\om^\star(z)\,dA(z)\\
    &\ge\frac{1}{2}\int_{\D}|f(z)|^2\left(\vert g'_1(z)\vert +\vert
    g'_2(z)\vert\right)^2\om^\star(z)\,dA(z)\\
    &\ge\frac{1}{2}\int_{\D}|f(z)|^2\frac{\om^\star(z)}{(1-|z|)^2}\,dA(z)
    \asymp\int_{\D}|f(z)|^2\frac{\int_{|z|}^1\om(s)\,ds}{(1-|z|)}\,dA(z)\\
    &=\int_{\D}|f(z)|^2\frac{\psi_\om(|z|)}{1-|z|}\,\om(z)\,dA(z)
    \end{split}
    \end{equation}
for all $f\in\H(\D)$. If
$\int_{\D}\frac{\psi_\om(|z|)}{1-|z|}\,\om(z)\,dA(z)=\infty$, we
choose $f\equiv1$ to obtain a contradiction. Assume now that
    $
    \int_{\D}\frac{\psi_\om(|z|)}{1-|z|}\,\om(z)\,dA(z)<\infty,
    $
and replace $f$ in \eqref{56} by the test function $F_{a,2}$ from
Lemma~\ref{testfunctions1}. Then \eqref{eq:tf2} and
Lemma~\ref{Lemma:weights-in-D-hat} yield
    \begin{equation*}
    \om^\star(a)\gtrsim\int_0^1\frac{(1-|a|)^{\gamma+1}}{(1-|a|r)^\gamma}\frac{\psi_\om(r)}{1-r}\,\om(r)\,dr
    \gtrsim(1-|a|)\int_{|a|}^1\frac{\psi_\om(r)}{1-r}\,\om(r)\,dr,
    \end{equation*}
and hence
    $$
    \int_{|a|}^1\frac{\psi_\om(r)}{1-r}\,\om(r)\,dr\lesssim\int_{|a|}^1\om(r)\,dr,\quad
    a\in\D.
    $$
By letting $|a|\to1^-$, Bernouilli-l'H\^{o}pital theorem and the
assumption $\om\in\I$ yield a contradiction.
\par (E) Recall that
$\BMOA\subset\CC^1(\om^\star)$\index{$\CC^1(\om^\star)$}
by Part~(B). See \cite[Proposition $5.2$]{PelRat} for the remaining inclusion. In fact, there is
constructed a lacunary series $g\in \CC^1(\om^\star)\setminus H^2$.
\end{proof}

\begin{proposition}\label{PropConformallyInvariant}
Let $\om\in\I$ such that both $\om(r)$ and
$\frac{\psi_\om(r)}{1-r}$ are essentially increasing on $[0,1)$,
and
    \begin{equation}\label{41}
    \int_r^1\omega(s)s\,ds\lesssim\int_{\frac{2r}{1+r^2}}^1\omega(s)s\,ds,\quad0\le
    r<1.
    \end{equation}
Then $\CC^1(\omega^\star)$ is not conformally
invariant.
\end{proposition}

\begin{proof}
Let $\omega\in\I$ be as in the assumptions. An standard calculation and Lemma~\ref{Lemma:weights-in-D-hat} gives that
$g\in\CC^1(\omega^\star)$ if and only if
    $$
    \sup_{b\in\D}\frac{(1-|b|)^2}{\omega(S(b))}\int_\D\frac{|g'(z)|^2}{|1-\overline{b}z|^2}\omega(S(z))\,dA(z)<\infty.
    $$ Let
$g\in\CC^1(\omega^\star)\setminus H^2$ be the function constructed
in the proof of Proposition~\ref{pr:blochcpp}(E). Then
    \begin{equation}\label{43}\index{$\CC^1(\om^\star)$}
    \begin{split}
    &\sup_{b\in\D}\frac{(1-|b|)^2}{\omega(S(b))}\int_\D\frac{|(g\circ\vp_a)'(z)|^2}{|1-\overline{b}z|^2}\omega(S(z))\,dA(z)\\
    &\ge\frac{(1-|a|)^2}{\omega(S(a))}\int_\D\frac{|g'(\zeta)|^2}{|1-\overline{a}\vp_a(\zeta)|^2}
    \omega(S(\vp_a(\zeta)))\,dA(\z)\\
    &\ge\int_{D(0,|a|)}|g'(\z)|^2(1-|\z|)\left(\frac{\omega(S(\vp_a(\zeta)))}{\omega(S(a))}\frac{|1-\overline{a}\z|^2}{1-|\z|}\right)dA(\z),
    \end{split}
    \end{equation}
where
    \begin{equation*}
    \begin{split}
    \frac{\omega(S(\vp_a(\zeta)))}{\omega(S(a))}\frac{|1-\overline{a}\z|^2}{1-|\z|}
    &=\frac{(1-|\vp_a(\zeta)|)\int_{|\vp_a(\zeta)|}^1\omega(s)s\,ds}{(1-|a|)\int_{|a|}^1\omega(s)s\,ds}
    \frac{|1-\overline{a}\z|^2}{1-|\z|}\\
    &\gtrsim\frac{\int_{\frac{2|a|}{1+|a|^2}}^1\omega(s)s\,ds}{\int_{|a|}^1\omega(s)s\,ds}\gtrsim1,\quad
    |\z|\le|a|,
    \end{split}
    \end{equation*}
by Lemma~\ref{Lemma:weights-in-D-hat} and \eqref{41}. Since $g\not\in
H^2$, the assertion follows by letting $|a|\to1^-$ in \eqref{43}.
\end{proof}
\subsection{Boundedness of the integral operator. Case $\mathbf{q=p}$}
We shall use the following preliminary result.
\begin{lemma}\label{le:tgminfty}
Let $0<p,q<\infty$ and $\omega\in\DD$.
 If $T_g:A^p_\omega\to A^q_\omega$ is bounded, then
\begin{equation}\label{Eq:Radialqp}
    M_\infty(r,g')\lesssim\frac{(\omega^\star(r))^{\frac{1}{p}-\frac{1}{q}}}{1-r},\quad
    0<r<1.
    \end{equation}
\end{lemma}

\begin{proof}
 Let $0<p,q<\infty$ and $\omega\in\DD$, and assume that
$T_g:A^p_\omega\to A^q_\omega$ is bounded. Consider the functions
    $$
    f_{a,p}(z)=\frac{(1-|a|)^{\frac{\gamma+1}{p}}}{(1-\overline{a}z)^{\frac{\gamma+1}{p}}\om\left(S(a)\right)^{\frac1p}},\quad
    a\in\D.
    $$
By Lemma \ref{testfunctions1} there is
$\gamma>0$ such that
$\sup_{a\in\D}\|f_{a,p}\|_{A^p_\om}<\infty$. Since
    \begin{equation*}
    \begin{split}
    \|h\|_{A^q_\omega}^q&\ge\int_{\D\setminus
    D(0,r)}|h(z)|^q\omega(z)\,dA(z)
    \gtrsim M_q^q(r,h)\int_r^1\omega(s)\,ds,\quad r\ge\frac12,
    \end{split}
    \end{equation*}
for all $h\in A^q_\omega$, we obtain
    \begin{equation*}\index{$f_{a,p}$}
    \begin{split}
    M^q_q(r,T_g(f_{a,p}))&\lesssim\frac{\|T_g(f_{a,p})\|_{A^q_\omega}^q}{\int_r^1\omega(s)\,ds}
    \le\frac{\|T_g\|^q_{(A^p_\om,A^q_\om)} \cdot\left(\sup_{a\in\D}\|f_{a,p}\|^q_{A^p_\om}\right)}{\int_r^1\omega(s)\,ds}\\
    &\lesssim\frac{1}{\int_r^1\omega(s)\,ds},\quad r\ge\frac12,
    \end{split}
    \end{equation*}
for all $a\in\D$. This together with the well-known relations
$M_\infty(r,f)\lesssim M_q(\rho,f)(1-r)^{-1/q}$ and
$M_q(r,f')\lesssim M_q(\rho,f)/(1-r)$, $\rho=(1+r)/2$,
Lemma~\ref{Lemma:weights-in-D-hat} and \eqref{3} yield
    \begin{equation*}
    \begin{split}\index{$f_{a,p}$}
    |g'(a)|&\asymp(\omega^\star(a))^\frac1p|T_g(f_{a,p})'(a)|
    \lesssim(\omega^\star(a))^\frac1p\frac{M_q\left(\frac{1+|a|}{2},\left(T_g(f_{a,p})\right)'\right)}{\left(1-|a|\right)^{\frac1q}}\\
    &\lesssim (\omega^\star(a))^\frac1p\frac{M_q((3+|a|)/4,T_g(f_{a,p}))}{\left(1-|a|\right)^{1+\frac1q}}
    \asymp\frac{(\omega^\star(a))^{\frac1p-\frac1q}}{1-|a|},\quad |a|\ge\frac12.
    \end{split}
    \end{equation*}
The assertion follows from this inequality.

\end{proof}

\begin{theorem}\label{th:integralq=p}
 If $\om\in\DD$, $g\in\H(\D)$ and $0<p<\infty$, then $T_g:A^p_\om\to A^p_\om$ is bounded if and only if $g\in \CC^1(\om^\star)$.
\end{theorem}
\begin{proof}
 If $p=2$ the equivalence follows from
Theorem~\ref{ThmLittlewood-Paley}, the definition of
$\CC^1(\om^\star)$\index{$\CC^1(\om^\star)$} and \eqref{calpha}. The rest of the proof
 is divided in four cases.

\medskip
Let $p>2$ and assume that
$g\in\CC^1(\omega^\star)$. Since
 $L^{p/2}_\omega
\backsimeq \left(L^{\frac{p}{p-2}}_\omega\right)^\star$,
Theorem~\ref{th:normacono} shows that $T_g:A^p_\omega\to
A^p_\omega$ is bounded if and only if
    \begin{equation*}
    \left|
    \int_\D\,h(u)\left(\int_{\Gamma(u)}|f(z)|^2|g'(z)|^2\,dA(z)\right)\omega(u)\,dA(u)
    \right|\lesssim\|h\|_{L^{\frac{p}{p-2}}_\omega}\|f\|^2_{A^p_\om},\quad\text{ $h\in L^{\frac{p}{p-2}}_\omega$.}
    \end{equation*}
Bearing in mind Theorems~\ref{th:cm} and \ref{co:maxbou},
    \begin{equation*}
    \begin{split}
    &\left|
    \int_\D\,h(u)\left(\int_{\Gamma(u)}|f(z)|^2|g'(z)|^2\,dA(z)\right)\omega(u)\,dA(u)\right|
   \\ & \le\int_\D|f(z)|^2|g'(z)|^2\left(\int_{T(z)}|h(u)|\omega(u)\,dA(u)\right)\,dA(z)\\
    &\asymp\int_\D|f(z)|^2|g'(z)|^2\omega^\star(z)\left(\frac{1}{\omega\left(S(z)\right)}\int_{T(z)}|h(u)|\omega(u)
    \,dA(u)\right)\,dA(z)
    \\&  \lesssim\int_\D|f(z)|^2|g'(z)|^2\omega^\star(z)M_\om(|h|)(z)\,dA(z)
     \\ &  \lesssim \left(\int_\D|f(z)|^{p}|g'(z)|^2\omega^\star(z)\,dA(z)\right)^{\frac{2}{p}}
   \cdot\left(\int_\D\left(M_\om(|h|)(z)\right)^{\frac{p}{p-2}}|g'(z)|^2\omega^\star(z)\,dA(z)\right)^{\frac{p-2}{p}}
   \\ &\lesssim \|f\|^2_{A^p_\om}\|h\|_{L^{\frac{p}{p-2}}_\omega},
    \end{split}
    \end{equation*}
  so  $T_g:A^p_\omega\to
A^p_\omega$ is bounded.
    \medskip
\par Reciprocally, let  $p>2$  and assume that $T_g:A^p_\omega\to A^p_\omega$ is
bounded. By  Theorem~\ref{th:normacono}
this is equivalent to
    \begin{equation*}
    \|T_g(f)\|_{A^p_\omega}^p
    \asymp\int_\D\left(\int_{\Gamma(u)}|f(z)|^2|g'(z)|^2\,dA(z)\right)^\frac{p}2\omega(u)\,dA(u)
    \lesssim\|f\|_{A^p_\omega}^p
    \end{equation*}
for all $f\in A^p_\om$. By using this together with
\eqref{3}, Fubini's theorem, H\"older's
inequality and Lemma~\ref{le:funcionmaximalangular}, we obtain
    \begin{eqnarray*}\index{$N(f)$}\index{non-tangential maximal function}
    &&\int_{\D}|f(z)|^p|g'(z)|^2\om^\star(z)\,dA(z)
    \asymp\int_{\D}|f(z)|^p|g'(z)|^2\om(T(z))\,dA(z)\\
    &&=\int_\D\int_{\Gamma(u)}|f(z)|^p|g'(z)|^2dA(z)\om(u)\,dA(u)\\
    &&\le\int_\D N(f)(u)^{p-2}\int_{\Gamma(u)}|f(z)|^2|g'(z)|^2\,dA(z)\om(u)\,dA(u)\\
    &&\le\left(\int_\D
    N(f)(u)^{p}\om(u)\,dA(u)\right)^\frac{p-2}{p}\\
    &&\quad\cdot\left(\int_\D\left(\int_{\Gamma(u)}|f(z)|^2|g'(z)|^2dA(z)\right)^\frac{p}{2}\om(u)\,dA(u)\right)^\frac{2}{p}
    \lesssim\|f\|_{A^p_\om}^p
    \end{eqnarray*}
for all $f\in A^p_\om$. Therefore $|g'(z)|^2\omega^\star(z)dA(z)$
is a $p$-Carleson measure for $A^p_\omega$, and thus
$g\in\CC^1(\om^\star)$ by the
definition.
\smallskip\par
Let now $0<p<2$, and assume that
$g\in\CC^1(\omega^\star)$. Then
\begin{equation*}
    \begin{split}
    \|T_g(f)\|_{A^p_\omega}^p
    &\asymp \int_\D\left(\int_{\Gamma(u)}|f(z)|^2|g'(z)|^2\,dA(z)\right)^\frac{p}2\omega(u)\,dA(u)\\
    &\le\int_\D N(f)(u)^{\frac{p(2-p)}{2}}\cdot\left(\int_{\Gamma(u)}|f(z)|^{p}|g'(z)|^2\,dA(z)\right)^\frac{p}2\omega(u)\,dA(u)\\
    &\le\left(\int_\D
    N(f)(u)^p\omega(u)\,dA(u)\right)^\frac{2-p}{2}\\
    &\quad\cdot\left(\int_\D\int_{\Gamma(u)}|f(z)|^{p}|g'(z)|^2\,dA(z)\omega(u)\,dA(u)\right)^\frac{p}2\\
    &\lesssim\|f\|_{A^p_\omega}^{\frac{p(2-p)}{2}}\left(\int_\D|f(z)|^{p}|g'(z)|^2\omega(T(z))\,dA(z)\right)^\frac{p}2\\
    &\asymp\|f\|_{A^p_\omega}^{\frac{p(2-p)}{2}}\left(\int_{\D}|f(z)|^{p}|g'(z)|^2\omega^\star(z)\,dA(z)\right)^\frac{p}2
    \lesssim\|f\|_{A^p_\omega}^p.
    \end{split}
    \end{equation*}

Let now $0<p<2$, and assume that $T_g:A^p_\omega\to A^p_\omega$ is
bounded. Then Lemma~\ref{le:tgminfty} and its proof imply
$g\in\mathcal{B}$ and
    \begin{equation}\label{eq:tgbloch}
    \|g\|_{\mathcal{B}}\lesssim\|T_g\|.
    \end{equation}
Choose $\gamma>0$ large enough, and consider the functions
$F_{a,p}=\left(\frac{1-|a|^2}{1-\overline{a}z}\right)^{\frac{\gamma+1}{p}}$
of Lemma~\ref{testfunctions1}.\index{$F_{a,p}$} Let
$1<\a,\b<\infty$ such that $\b/\a=p/2<1$, and let $\a'$ and $\b'$
be the conjugate indexes of $\a$ and $\b$. Then
\eqref{3}, Fubini's theorem, H\"older's
inequality, \eqref{eq:tf1} and \eqref{normacono} yield
    \begin{equation}\index{$F_{a,p}$}
    \begin{split}\label{eq:tgb1}
    &\int_{S(a)}|g'(z)|^2\omega^\star(z)\,dA(z)\\
    &\asymp\int_\D\left(\int_{S(a)\cap\Gamma(u)}|g'(z)|^2|F_{a,p}(z)|^2\,dA(z)\right)^{\frac1\a+\frac1{\a'}}\omega(u)\,dA(u)\\
    &\le\left(\int_\D\left(\int_{\Gamma(u)}|g'(z)|^2|F_{a,p}(z)|^2\,dA(z)\right)^\frac{\b}{\a}\omega(u)\,dA(u)\right)^\frac1\b\\
    &\quad\cdot\left(\int_\D\left(\int_{\Gamma(u)\cap
    S(a)}|g'(z)|^2\,dA(z)\right)^\frac{\b'}{\a'}\omega(u)\,dA(u)\right)^\frac1{\b'}\\
    &\asymp\|T_g(F_{a,p})\|_{A^p_\omega}^\frac{p}{\b}\|S_g(\chi_{S(a)})\|_{L_\omega^\frac{\b'}{\a'}}^\frac{1}{\a'},\quad
    a\in\D,
    \end{split}
    \end{equation}
where
    \begin{equation*}
    S_g(\varphi)(u)=\int_{\Gamma(u)}|\varphi(z)|^2|g'(z)|^2\,dA(z),\quad
    u\in\D\setminus\{0\},
    \end{equation*}
for any bounded function $\varphi$ on $\D$. Since $\b/\a=p/2<1$,
we have $\frac{\b'}{\a'}>1$ with the conjugate exponent
$\left(\frac{\b'}{\a'}\right)'=\frac{\b(\a-1)}{\a-\b}>1$.
Therefore
      \begin{equation}
      \begin{split}\label{eq:tgb2}
      \|S_g(\chi_{S(a)})\|_{L_\omega^\frac{\b'}{\a'}}
      =\sup_{\|h\|_{L_\omega^{\frac{\b(\a-1)}{\a-\b}}}\le1}
      \left|\int_\D h(u)S_g(\chi_{S(a)})(u)\omega(u)\,dA(u)\right|.
      \end{split}
      \end{equation}
By using Fubini's theorem, \eqref{3},
H\"older's inequality and Theorem~\ref{co:maxbou}, we deduce
     \begin{equation}
     \begin{split}\label{eq:tgb3}
     &\left|\int_\D h(u)S_g(\chi_{S(a)})(u)\omega(u)\,dA(u)\right|\\
     &\le\int_\D|h(u)|\int_{\Gamma(u)\cap
     S(a)}|g'(z)|^2\,dA(z)\,\omega(u)\,dA(u)\\
     &\lesssim\int_{S(a)}M_\omega(|h|)(z)|g'(z)|^2\omega^\star(z)\,dA(z)\\
     &\le\left(\int_{S(a)}|g'(z)|^2\omega^\star(z)\,dA(z)\right)^\frac{\a'}{\b'}\\
     &\quad\cdot\left(\int_\D
     M_\omega(|h|)^{\left(\frac{\b'}{\a'}\right)'}|g'(z)|^2\omega^\star(z)\,dA(z)\right)^{1-\frac{\a'}{\b'}}\\
     &\lesssim\left(\int_{S(a)}|g'(z)|^2\omega^\star(z)\,dA(z)\right)^\frac{\a'}{\b'}\\
     &\quad\cdot\left(\sup_{a\in\D}\frac{\int_{S(a)}|g'(z)|^2\omega^\star(z)\,dA(z)}{\omega(S(a))}\right)^{1-\frac{\a'}{\b'}}
     \|h\|_{L_\omega^{\left(\frac{\b'}{\a'}\right)'}}.
     \end{split}
     \end{equation}
By replacing $g(z)$ by $g_r(z)=g(rz)$, $0<r<1$, and combining
\eqref{eq:tgb1}--\eqref{eq:tgb3}, we obtain
    \begin{equation*}\index{$F_{a,p}$}
    \begin{split}
    \int_{S(a)}|g_r'(z)|^2\omega^\star(z)\,dA(z)&\lesssim\|T_{g_r}(F_{a,p})\|_{A^p_\omega}^\frac{p}{\b}
    \left(\int_{S(a)}|g_r'(z)|^2\omega^\star(z)\,dA(z)\right)^\frac{1}{\b'}\\
    &\quad\cdot\left(\sup_{a\in\D}\frac{\int_{S(a)}|g_r'(z)|^2\omega^\star(z)\,dA(z)}{\omega(S(a))}\right)^{\frac{1}{\a'}\left(1-\frac{\a'}{\b'}\right)}.
    \end{split}
    \end{equation*}
We now claim that there exists a constant $C=C(\om)>0$ such that
    \begin{equation}\label{eq:dilatadas}\index{$F_{a,p}$}
    \sup_{0<r<1}\|T_{g_r}(F_{a,p})\|_{A^p_\omega}^p\le
    C\|T_{g}\|_{A^p_\omega}^p\omega(S(a)),\quad a\in\D.
    \end{equation}
Taking this for granted for a moment, we deduce
    \begin{equation*}
    \left(\frac{\int_{S(a)}|g_r'(z)|^2\omega^\star(z)\,dA(z)}{\omega(S(a))}\right)^\frac{1}{\b}
    \lesssim\|T_{g}\|_{A^p_\omega}^\frac{p}{\b}
    \left(\sup_{a\in\D}\frac{\int_{S(a)}|g_r'(z)|^2\omega^\star(z)\,dA(z)}{\omega(S(a))}\right)^{\frac{1}{\a'}\left(1-\frac{\a'}{\b'}\right)}
    \end{equation*}
for all $0<r<1$ and $a\in\D$. This yields
    $$
    \frac{\int_{S(a)}|g_r'(z)|^2\omega^\star(z)\,dA(z)}{\omega(S(a))}\lesssim\|T_{g}\|^2,\quad
    a\in\D,
    $$
and so
    \begin{equation*}
    \sup_{a\in\D}\frac{\int_{S(a)}|g(z)|^2\omega^\star(z)\,dA(z)}{\omega(S(a))}\le\sup_{a\in\D}\liminf_{r\to
1^-}\left(\frac{\int_{S(a)}|g_r'(z)|^2\omega^\star(z)\,dA(z)}{\omega(S(a))}\right)
    \lesssim  \|T_{g}\|^2
    \end{equation*}
by Fatou's lemma. Therefore
$g\in\CC^1(\om^\star)$ by
Theorem~\ref{th:cm}.

It remains to prove~\eqref{eq:dilatadas}. To do this, let
$a\in\D$. If $|a|\le r_0$, where $r_0\in(0,1)$ is fixed, then the
inequality in~\eqref{eq:dilatadas} follows by Theorem~\ref{th:normacono}, the change of
variable $rz=\z$, the fact
    \begin{equation}\label{conos}
    \Gamma(ru)\subset \Gamma(u),\quad 0<r<1,
    \end{equation}
and the assumption that $T_g:A^p_\om\to A^p_\om$ is bounded. If
$a\in\D$ is close to the boundary, we consider two separate cases.

Let first $\frac12<|a|\le\frac1{2-r}$. Then
    $$
    |1-\overline{a}z|\le\left|1-\overline{a}\frac{z}{r}\right|+\frac{1-r}{2-r}\le
    2\left|1-\overline{a}\frac{z}{r}\right|,\quad |z|\le r.
    $$
Therefore Theorem~\ref{th:normacono}, \eqref{conos} and
\eqref{eq:tf2} yield
    \begin{equation}\label{78}\index{$F_{a,p}$}
    \begin{split}
    \|T_{g_{r}}(F_{a,p})\|_{A^p_\omega}^p
    &\asymp \int_\D \left(\int_{\Gamma(u)}r^2|g'(rz)|^2
    \left|F_{a,p}(z)\right|^2\,dA(z)\right)^{\frac{p}{2}}\om(u)\,du\\
    &=\int_\D \left(\int_{\Gamma(ru)}|g'(z)|^2
    \left|F_{a,p}\left(\frac{z}{r}\right)\right|^2
    \,dA(z)\right)^{\frac{p}{2}}\om(u)\,du\\
    &\le2^{\gamma+1}\int_\D \left(\int_{\Gamma(ru)}|g'(z)|^2
    \left|F_{a,p}(z)\right|^2
    \,dA(z)\right)^{\frac{p}{2}}\om(u)\,du\\
    &\le2^{\gamma+1}\int_\D \left(\int_{\Gamma(u)}|g'(z)|^2
    \left|F_{a,p}(z)\right|^2
    \,dA(z)\right)^{\frac{p}{2}}\om(u)\,du\\
    &\asymp\|T_{g}(F_{a,p})\|_{A^p_\omega}^p\lesssim\|T_{g}\|_{A^p_\omega}^p\omega(S(a)),
    \end{split}
    \end{equation}
and hence
    \begin{equation}\index{$F_{a,p}$}
    \begin{split}\label{eq:dil1}
    \|T_{g_{r}}(F_{a,p})\|_{A^p_\omega}^p\lesssim \|T_{g}\|_{A^p_\omega}^p\omega(S(a))
    ,\quad\frac12<|a|\le\frac1{2-r}.
    \end{split}
    \end{equation}

Let now $|a|>\max\{\frac1{2-r},\frac12\}$. Then, by
Theorem~\ref{th:normacono}, \eqref{eq:tgbloch} and
Lemma~\ref{testfunctions1}, we deduce
    \begin{equation}\index{$F_{a,p}$}
    \begin{split}\label{eq:dil2}
    \|T_{g_{r}}(F_{a,p})\|_{A^p_\omega}^p
    &\asymp\int_\D\left(\int_{\Gamma(u)}r^2|g'(rz)|^2
    \left|F_{a,p}(z)\right|^2\,dA(z)\right)^{\frac{p}{2}}\om(u)\,du\\
    &\le M_\infty(r,g')^p \int_\D \left(\int_{\Gamma(u)}\left|F_{a,p}(z)\right|^2\,dA(z)\right)^{\frac{p}{2}}\om(u)\,du
\\ &\lesssim
 M_\infty\left( 2-\frac{1}{|a|},g'\right)^p(1-|a|)^p\left\|\left(\frac{1-|a|^2}{1-\overline{a}z}\right)^{\frac{\gamma+1}{p}-1}\right\|_{A^p_\omega}^p
\\ & \lesssim \|g\|^p_{\mathcal{B}}\,\omega(S(a))\lesssim\|T_{g}\|_{A^p_\omega}^p\,\omega(S(a))
    \end{split}
    \end{equation}
for $\gamma>0$ large enough. This together with \eqref{eq:dil1}
gives \eqref{eq:dilatadas}. The proof of (i) is now complete.
\end{proof}
\subsection{Boundedness of the integral operator. Case $q\ge p$.}
If
$\alpha>1$ and $\om\in\DD$,  $g\in\CC^{\alpha}(\omega^\star)$
if and only if \cite[Proposition $4.7$]{PelRat}
    $$
    M_\infty(r,g')\lesssim\frac{(\omega^\star(r))^{\frac{\alpha-1}{2}}}{1-r},\quad
    0<r<1.   $$
    So using  analogous  ideas to those employed in the proof of Theorem~\ref{th:integralq=p}  we can prove the following.
    \begin{theorem}
    Let $0<p<q<\infty$, $\om\in\DD$ and $g\in\H(\D)$.
    \begin{enumerate}
    \item[\rm(i)]If $0<p<q$ and $\frac1p-\frac1q<1$, then the
following conditions are equivalent:
    \begin{enumerate}
   \item \,$T_g:A^p_\om\to A^q_\om$ is bounded;
    \item $\displaystyle M_\infty(r,g')\lesssim\frac{(\omega^\star(r))^{\frac1p-\frac1q}}{1-r},\quad
    r\to1^-;
    $
    \item $\displaystyle \sup_{I\subset\T}\frac{\int_{S(I)}|g'(z)|^2\om^\star(z)\,dA(z)}
    {\left(\om\left(S(I)\right)\right)^{\alpha}}<\infty.$
    \end{enumerate}
     \item[\rm(ii)] If $\frac1p-\frac1q\ge 1$, then $T_g:A^p_\om\to A^q_\om$ is bounded if and only if $g$ is constant.
     \end{enumerate}
    \end{theorem}
\subsection{Boundedness of the integral operator. Case $0<q<p$.}
We shall use Corollary~\ref{cor:FactorizationBergman}  on factorization of $A^p_\om$-functions  in order to study the remaining case.
\begin{theorem}
If $0<q<p<\infty$, $g\in\H(\D)$ and $\om\in\widetilde{\I}\cup\R$,
    $T_g:A^p_\om\to A^q_\om$ is bounded if and only if
     $g\in A^s_\omega$, where $\frac{1}{s}=\frac1q-\frac1p$.
\end{theorem}
\begin{proof}
The sufficiency can be proved arguing as in Proposition \ref{pr:cmqmenorp} and it is valid for any radial weight $\om$.
Let first $g\in A^s_{\omega}$, where $s=\frac{pq}{p-q}$. Then
Theorem~\ref{th:normacono}, H\"older's inequality and
Lemma~\ref{le:funcionmaximalangular} yield
    \begin{equation}\label{58}\index{$N(f)$}\index{non-tangential maximal function}
    \begin{split}
    \|T_g(f)\|_{A^q_\omega}^q&\asymp\int_\D\left(\int_{\Gamma(u)}|f(z)|^2|g'(z)|^2\,dA(z)\right)^\frac{q}{2}\omega(u)\,dA(u)\\
    &\le\int_\D(N(f)(u))^q\left(\int_{\Gamma(u)}|g'(z)|^2\,dA(z)\right)^\frac{q}{2}\omega(u)\,dA(u)\\
    &\le\left(\int_\D(N(f)(u))^p\omega(u)\,dA(u)\right)^\frac{q}{p}\\
    &\quad\cdot\left(\int_\D\left(\int_{\Gamma(u)}|g'(z)|^2\,dA(z)\right)^\frac{pq}{2(p-q)}\omega(u)\,dA(u)\right)^\frac{p-q}{p}\\
    &\le C_1^{q/p}C_2(p,q,\om)\|f\|_{A^p_\omega}^q\|g\|_{A^s_\omega}^q.
    \end{split}
    \end{equation}
Thus $T_g:A^p_\omega\to A^q_\omega$ is bounded.
\medskip
\par In order to prove the converse we
we will use ideas
from~\cite[p.~170--171]{AC}, where~$T_g$ acting on Hardy spaces is
studied. We begin with the following result whose proof relies on
Corollary~\ref{cor:FactorizationBergman}.

\begin{proposition}\label{PropSmallIndeces}
Let $0<q<p<\infty$ and $\omega\in\widetilde{\I}\cup\R$, and let
$T_g:A^p_\omega\to A^q_\omega$ be bounded. Then $T_g: A^{
\hat{p}}_\omega\to A^{ \hat{q}}_\omega$ is bounded for any $
\hat{p}<p$ and $ \hat{q}<q$ with $\frac{1}{ \hat{q}}-\frac{1}{
\hat{p}}=\frac{1}{q}-\frac{1}{p}$. Further, if $0<p\le 2$, then
there exists $C=C(p,q,\omega)>0$ such that
    \begin{equation}\label{eq:tqq<p1}
    \limsup_{ \hat{p}\to p^-}\|T_g\|_{\left(A^{ \hat{p}}_\om,A^{ \hat{q}}_\om\right)}\le C\|T_g\|_{\left(A^{p}_\om,A^{q}_\om\right)}.
    \end{equation}
\end{proposition}

\begin{proof} Theorem~\ref{Thm:FactorizationBergman} shows that for any $f\in
A^{\hat{p}}_\om$, there exist $f_1\in A^p_\om$ and $f_2\in
A^{\frac{ \hat{p}p}{p- \hat{p}}}_\om$ such that
    \begin{equation}\label{eq:fact}
    f=f_1f_2\quad\text{and}\quad\|f_1\|_{A^p_\om}\cdot\|f_2\|_{A^{\frac{ \hat{p}p}{p- \hat{p}}}_\om}\le C_3\|f\|_{A^{ \hat{p}}_\om}
    \end{equation}
for some constant $C_3=C_3(p, \hat{p},\om)>0$. We observe that
$T_g(f)=T_F(f_2)$, where $F=T_g(f_1)$. Since $T_g:A^p_\omega\to
A^q_\omega$ is bounded,
    \begin{equation}\label{eq:fact2}
    \|F\|_{ A^q_\omega}=\|T_g(f_1)\|_{
    A^q_\omega}\le\|T_g\|_{\left(A^{p}_\om,A^{q}_\om\right)}\|f_1\|_{A^p_\om}<\infty,
    \end{equation}
and hence $F\in A^q_\omega$. Then \eqref{58} and the identity
$\frac{1}{q}=\frac{1}{ \hat{q}}-\frac{1}{\frac{ \hat{p}p}{p-
\hat{p}}}$ yield
    \begin{equation*}
    \|T_g(f)\|_{ A^{ \hat{q}}_\omega}=\|T_F(f_2)\|_{ A^{ \hat{q}}_\omega}\le C_1^{\frac{1}{ \hat{p}}-\frac{1}{p}}C_2
    \|f_2\|_{A^{\frac{ \hat{p}p}{p- \hat{p}}}_\om}\|F\|_{ A^q_\omega},
    \end{equation*}
where $C_2=C_2(q,\om)>0$. This together with \eqref{eq:fact} and
\eqref{eq:fact2} gives
    \begin{equation}\label{j11}
    \begin{split}
    \|T_g(f)\|_{A^{ \hat{q}}_\omega}&\le
    C_1^{\frac{1}{ \hat{p}}-\frac{1}{p}}C_2\|T_g\|_{\left(A^{p}_\om,A^{q}_\om\right)}\|f_1\|_{A^p_\om}
    \cdot\|f_2\|_{A^{\frac{ \hat{p}p}{p- \hat{p}}}_\om}\\
    &\le C_1^{\frac{1}{ \hat{p}}-\frac{1}{p}}C_2\,C_3\,\|T_g\|_{\left(A^{p}_\om,A^{q}_\om\right)}\|f\|_{A^{ \hat{p}}_\om}.
    \end{split}
    \end{equation}
Therefore $T_g: A^{ \hat{p}}_\omega\to A^{ \hat{q}}_\omega$ is
bounded.

To prove~\eqref{eq:tqq<p1}, let $0<p\le 2$ and let $0< \hat{p}<2$
be close enough to $p$ such that
    $$
    \min\left\{\frac{p}{p- \hat{p}}, \frac{ \hat{p}p}{p- \hat{p}}\right\}>2.
    $$
If $f\in A^{ \hat{p}}_\om$, then
Corollary~\ref{cor:FactorizationBergman} shows that
\eqref{eq:fact} holds with $C_3=C_3(p,\omega)$. Therefore the
reasoning in the previous paragraph  and \eqref{j11} give
\eqref{eq:tqq<p1}.
\end{proof}

With this result in hand, we are ready to prove
(aiv)$\Rightarrow$(biv). Let $0<q<p<\infty$ and
$\omega\in\widetilde{\I}\cup\R$, and let $T_g: A^{p}_\omega\to
A^{q}_\omega$ be bounded. Denote
$\frac{1}{s}=\frac{1}{q}-\frac{1}{p}$. By the first part of
Proposition~\ref{PropSmallIndeces}, we may assume that $p\le2$. We
may also assume, without loss of generality, that $g(0)=0$. Define
$t^*=\sup\{t:g\in A^t_\om\}$. Since the constant function $1$
belongs to $A^p_\om$, we have $g=T_g(1)\in A^q_\om$, and hence
$t^*\ge q>0$. Fix a positive integer $m$ such that
$\frac{t^*}{m}<p$. For each $t<t^*$, set $ \hat{p}=
\hat{p}(t)=\frac{t}{m}$, and define $ \hat{q}= \hat{q}(t)$ by the
equation $\frac{1}{s}=\frac{1}{ \hat{q}}-\frac{1}{ \hat{p}}$. Then
$ \hat{p}<p$,  $ \hat{q}<q$ and $T_g: A^{ \hat{p}}_\omega\to A^{
\hat{q}}_\omega$ is bounded by Proposition~\ref{PropSmallIndeces}.
Since $g^m=g^{\frac{t}{ \hat{p}}}\in A^{ \hat{p}}_\omega$, we have
$g^{m+1}=(m+1)T_g(g^m)\in A^{ \hat{q}}_\om$ and
    $$
    \|g^{m+1}\|_{ A^{ \hat{q}}_\omega}\le
    (m+1)\|T_g\|_{\left(A^{ \hat{p}}_\om,A^{ \hat{q}}_\om\right)}\|g^m\|_{A^{ \hat{p}}_\om},
    $$
that is,
    \begin{equation}\label{eq:tgq<p2}
    \|g\|^{m+1}_{ A^{(m+1) \hat{q}}_\omega}\le (m+1)\|T_g\|_{\left(A^{ \hat{p}}_\om,A^{ \hat{q}}_\om\right)}\|g\|^m_{A^{t}_\om}.
    \end{equation}

Suppose first that for some $t<t^*$, we have
    $$
    t\ge (m+1) \hat{q}=\left(\frac{t}{ \hat{p}}+1 \right) \hat{q}= \hat{q}+t\left(1-\frac{ \hat{q}}{s}
    \right).
    $$
Then $s\le t<t^*$, and the result follows from the definition of
$t^*$. It remains to consider the case in which $t<(m+1) \hat{q}$
for all $t<t^*$. By H\"older's inequality, $\|g\|^m_{A^{t}_\om}\le
C_1(m,\om)\|g\|^m_{ A^{(m+1) \hat{q}}_\omega}$. This and
\eqref{eq:tgq<p2} yield
    \begin{equation}\label{j12}
    \|g\|_{ A^{(m+1) \hat{q}}_\omega}\le
    C_2(m,\om)\|T_g\|_{\left(A^{ \hat{p}}_\om,A^{ \hat{q}}_\om\right)},
    \end{equation}
where $C_2(m,\om)=C_1(m,\om)(m+1)$. Now, as $t$ increases to
$t^*$, $ \hat{p}$ increases to $\frac{t^*}{m}$ and $ \hat{q}$
increases to $\frac{t^*s}{t^*+ms}$, so by \eqref{j12} and
\eqref{eq:tqq<p1} we deduce
    \begin{equation*}
    \begin{split}
    \|g\|_{ A^{\frac{(m+1)t^*s}{t^*+ms}}_\omega}
    &\le\limsup_{t\to t^*}\|g\|_{ A^{(m+1) \hat{q}}_\omega}
    \le C_2(m,\om)\limsup_{ \hat{p}\to
    p^-}\|T_g\|_{\left(A^{ \hat{p}}_\om,A^{ \hat{q}}_\om\right)}\\
    &\le C(p,q,m,\om)\|T_g\|_{\left(A^{p}_\om,A^{q}_\om\right)}<\infty.
    \end{split}
    \end{equation*}
The definition of $t^*$ implies $\frac{(m+1)t^*s}{t^*+ms}\le t^*$,
and so $t^*\ge s$. This finishes the proof.
\end{proof}

\par The main results of this section are gathered here.

\begin{theorem}
Let $0<p,q<\infty$, $\om\in\DD$ and $g\in\H(\D)$.
    \begin{itemize}
    \item[\rm(i)] The following conditions are equivalent:
    \begin{enumerate}
    \item[\rm(ai)]\, $T_g:A^p_\om\to A^p_\om$ is bounded;
    \item[\rm(bi)]\,
    $g\in\CC^1(\om^\star)$.
    \end{enumerate}
    \item[\rm(ii)] If $0<p<q$ and $\frac1p-\frac1q<1$, then the
following conditions are equivalent:
    \begin{enumerate}
    \item[\rm(aii)]\,$T_g:A^p_\om\to A^q_\om$ is bounded;
    \item[\rm(bii)] $\displaystyle M_\infty(r,g')\lesssim\frac{(\omega^\star(r))^{\frac1p-\frac1q}}{1-r},\quad
    r\to1^-;
    $
    \item[\rm(cii)]
    \,
    $g\in\CC^{q/p}(\om^\star)$.
    \end{enumerate}
    \item[\rm(iii)] If $\frac1p-\frac1q\ge 1$, then $T_g:A^p_\om\to A^q_\om$ is bounded if and only if $g$ is constant.
    \item[\rm(iv)] If $0<q<p<\infty$ and $\om\in\widetilde{\I}\cup\R$,
    then the following conditions are equivalent:
    \begin{enumerate}
    \item[\rm(aiv)]\,$T_g:A^p_\om\to A^q_\om$ is bounded;
    \item[\rm(biv)] $g\in A^s_\omega$, where $\frac{1}{s}=\frac1q-\frac1p$.
    \end{enumerate}
     \end{itemize}
\end{theorem}
\section{Composition operators}\label{sec:composition}
Each analytic self-map $\vp$ of $\D$ induces the composition operator $C_\vp(f)= f\circ\vp$
acting on $\H(\D)$. With regard to the theory of composition operators, we refer to~\cite{CowenMac95,Shapiro93,Zhu}.

Let $\z\in\vp^{-1}(z)$ denote the set of the points
$\{\z_n\}$ in $\D$, organized by increasing moduli, such that
$\vp(\z_n)=z$ for all $n$, with each point repeated according to its multiplicity.
For a radial weight $\om$ and an analytic self-map $\vp$ of $\D$
we define the generalized Nevanlinna counting function as
    $$
    N_{\vp,\om^\star}(z)=\sum_{\z\in\vp^{-1}(z),}\om^\star\left(z\right),\quad z\in\D\setminus\{\vp(0)\}.
    $$
Using the characterization of the $q$-Carleson measures for $A^p_\om$ provided in Theorem~\ref{th:cm}, Theorem~\ref{Theorem:CarlesonMeasures}
and a description of bounded differentiation operators from $A^p_\om$ to $L^q_\mu$
 \cite{PelRatMathAnn}, it has recently been proved
the following result~\cite{PelRatToeplitz}.

\begin{theorem}\label{Theorem:introduction-bounded-composition-operators}
Let $0<p,q<\infty$, $\omega\in\DD$ and $v$ be a radial weight, and let $\vp$ be an
analytic self-map of $\D$.
\begin{itemize}
\item[\rm(a)] If $p>q$, then the following assertions are equivalent:
\begin{enumerate}
\item[\rm(i)] $C_\vp:A^p_\om\to A^q_v$ is bounded;
\item[\rm(ii)] $C_\vp:A^p_\om\to A^q_v$ is compact;
\item[\rm(iii)] $N\left(\frac{N_{\vp,v^\star}}{\om^\star}\right)\in L^\frac{p}{p-q}_\om$.
\end{enumerate}
\item[\rm(b)] If $q\ge p$, then $C_\vp:A^p_\om\to A^q_v$ is bounded if and only if
    $$
    \limsup_{|z|\to1^-}\frac{N_{\vp,v^\star}(z)}{\om^\star(z)^\frac{q}{p}}<\infty.
    $$
\item[\rm(c)] If $q\ge p$, then $C_\vp:A^p_\om\to A^q_v$ is compact if and only if
    $$
    \lim_{|z|\to1^-}\frac{N_{\vp,v^\star}(z)}{\om^\star(z)^\frac{q}{p}}=0.
    $$
\end{itemize}
\end{theorem}
We observe that  condition~(iii) in the classical case $C_\vp:A^p_\a\to A^q_\b$ gives a characterization of bounded (and compact) operators that differs from the one in the existing literature~\cite{SmithYang98}. Here we shall prove an extension of this last result to the class of regular weights.

\begin{theorem}\label{th:psmallerq}Let $0<q<p<\infty$, $\omega\in\R$ and $v$ be a radial weight, and let $\vp$ be an
analytic self-map of $\D$. Then the following assertions are equivalent:
\begin{itemize}
\item[\rm(i)] $C_\vp: A^p_\om\to A^q_v$ is bounded;
\item[\rm(ii)] $C_\vp: A^p_\om\to A^q_v$ is compact;
\item[\rm(iii)] The function
\begin{equation*}
z\mapsto\frac{N_{\vp,v^\star}(z)}{\om^\star(z)}
\end{equation*}
belongs to $ L^{\frac{p}{p-q}}_\om$; \item[\rm(iv)] The function
\begin{equation*}
z\mapsto\frac{\int_{\Delta(z,r)}\frac{N_{\vp,v^\star}(\zeta)}{\om^\star(\zeta)}\,dA(\zeta)}{(1-|z|)^2}
\end{equation*}
belongs to $ L^{\frac{p}{p-q}}_\om$ for some (equivalently for
all) fixed $r\in(0,1)$.
\end{itemize}
\end{theorem}

\subsection{Preliminary results}

A key result in the proof of Theorem~\ref{th:psmallerq} is the local good behavior of the generalized
Nevanlinna counting function \cite[Lemma~$14$]{PelRatToeplitz}.
\begin{lemma}\label{Nsubharmonic}
Let $\vp$ be an
analytic self-map of $\D$ and $v$ a radial weight. Then
     $N_{\vp,v^\star}$ is subharmonic on $\D\setminus\{\vp(0)\}$.
\end{lemma}
Next, using  the subharmonicity of $|f|^p$, the definition of the class $\Inv$ and the fact that $\inf_{z\in K} \om(z)>0$ for any compact subset $K\subset\D$,
it can be deduced the following.
\begin{lemma}\label{le:uniInv}
Let $0<p<\infty$ and $\om\in\Inv$. Then the norm convergence in ${A^p_\om}$ implies the uniform convergence on compact subsets of $\D$.
\end{lemma}

\par We shall use the following result on composition operators acting on weighted Bergman spaces induced by weights that are not necessarily radial .

\begin{proposition}\label{pr:compinv}
Let $0<q,p<\infty$,  $\om\in\Inv$ such that the polynomials are
dense in $A^p_\om$, and   $v$ be a weight. If $n\in\N$,
then the following assertions are valid:
\begin{itemize}
\item[\rm(i)] $C_\vp: A^p_\om\to A^q_v$ is bounded if and only if
$C_\vp: A^{np}_\om\to A^{nq}_v$ is bounded. Moreover,
    $$
    \|C_\vp\|_{(A^{np}_\om,A^{nq}_v)}\asymp\|C_\vp\|^{1/n}_{(A^p_\om,A^q_v)}.
    $$
\item[\rm(ii)] If the norm convergence in $A^q_v$ implies the
uniform convergence on compact subsets of $\D$, then $C_\vp:
A^p_\om\to A^q_v$ is compact if and only if  $C_\vp: A^{np}_\om\to
A^{nq}_v$ is compact.
\end{itemize}
\end{proposition}

\begin{proof} (i) Let first $C_\vp: A^p_\om\to A^q_v$ be bounded and $f\in
A^{np}_\om$. Then $f^n\in A^p_\om$ and
    \begin{equation*}
    \begin{split}
    \|C_\vp(f)\|^{nq}_{A^{nq}_v}&=\int_\D |f\circ\vp(z)|^{nq}v(z)\,dA(z)=\int_\D |f^n\circ\vp(z)|^{q}v(z)\,dA(z)\\
    &=\|C_\vp(f^n)\|^{q}_{A^{q}_v}\le\|C_\vp\|^q_{(A^p_\om,A^q_v)}\|f^n\|^{q}_{A^p_\om}
    =\|C_\vp\|^q_{(A^p_\om,A^q_v)}\|f\|^{nq}_{A^{np}_\om},
    \end{split}
    \end{equation*}
so $C_\vp: A^{np}_\om\to A^{nq}_v$ is bounded and
$\|C_\vp\|_{(A^{np}_\om,A^{nq}_v)}\le
\|C_\vp\|^{1/n}_{(A^p_\om,A^q_v)}$.

Conversely, let $C_\vp: A^{np}_\om\to A^{nq}_v$ be bounded and
$f\in A^{p}_\om$. Now $n$ applications of
Theorem~\ref{Thm:FactorizationBergman} show that $f$ can be represented in the
form $f=\prod_{k=1}^nf_k$, where each $f_k\in A^{np}_\om$ and
    $$
    \prod_{k=1}^n\|f_k\|_{A^{np}_\om}\le C(n,p,\om)\|f\|_{A^p_\om}.
    $$
Therefore H\"{o}lder's inequality gives
    \begin{equation*}
    \begin{split}
    \|C_\vp(f)\|^{q}_{A^{q}_v}&=\int_\D
    \left|\left(\left(\prod_{k=1}^nf_k\right)\circ\vp\right)(z)\right|^{q}v(z)\,dA(z)\\
    &=\int_\D \prod_{k=1}^n|\left(f_k\circ\vp\right)(z)|^{q}v(z)\,dA(z)\le\prod_{k=1}^n\|C_\vp(f_k)\|^{q}_{A^{nq}_v}\\
    &\le\|C_\vp\|^{nq}_{(A^{np}_\om,A^{nq}_v)}\prod_{k=1}^n\|f_k\|^{q}_{A^{np}_\om}
    \le
    C(n,p,q,\om)\|C_\vp\|^{nq}_{(A^{np}_\om,A^{nq}_v)}\|f\|^q_{A^p_\om}.
    \end{split}
    \end{equation*}
So $C_\vp: A^{p}_\om\to A^{q}_v$ is bounded and
$\|C_\vp\|^{1/n}_{(A^{p}_\om,A^{q}_v)}\lesssim
\|C_\vp\|_{(A^{np}_\om,A^{nq}_v)}$.

(ii) We may assume that $\vp$ is not constant. Let first $C_\vp:
A^p_\om\to A^q_v$ be compact. To see that also $C_\vp:
A^{np}_\om\to A^{nq}_v$ is compact, take $\{f_j\}\subset
A^{np}_\om$ such that $\sup_{j}\|f_j\|_{A^{np}_\om}<\infty$. Since
$\om\in\Inv$ by the assumption, Lemma~\ref{le:uniInv} and Montel's
theorem imply the existence of a subsequence $\{f_{j_{k}}\}$ such
that $f_{j_{k}}$ converges uniformly on compact subsets of $\D$ to
some $f\in\H(\D)$, and further $f\in A^{np}_\om$ by Fatou's lemma.
Therefore the sequence $\{g_{k}\}=\{f_{j_{k}}-f\}$ converges
uniformly to $0$ on compact subsets of $\D$ and
$\sup_{k}\|g_k\|_{A^{np}_\om}<\infty$. Hence $\{g^n_k\}$ converges
uniformly to $0$ compact subsets of $\D$ and
$\sup_{k}\|g^n_k\|_{A^{p}_\om}<\infty$. Now, since $C_\vp:
A^p_\om\to A^q_v$ is compact there is a subsequence
$\{g^n_{k_m}\}$ and $G\in\H(\D)$ such that
    \begin{equation}\label{eq:j5}
    \|C_\vp(g^n_{k_m}-G)\|^{q}_{A^{q}_v}\to 0,\quad m\to\infty.
    \end{equation}
Now, by the hypotheses on $v$,
$g^n_{k_m}\circ\vp-G\circ\vp$ converges uniformly to $0$ on
compact subsets of $\D$, and since $\vp$ is not constant, this and
the uniform convergence of $\{g^n_k\}$ to zero imply $G\equiv0$.
So, by \eqref{eq:j5},
    $$
    \|C_\vp(g_{k_m})\|^{nq}_{A^{nq}_v}\to 0,\quad m\to\infty,
    $$
and hence $C_\vp: A^{np}_\om\to A^{nq}_v$ is compact.

Conversely, let $C_\vp: A^{np}_\om\to A^{nq}_v$ be compact and
take $\{f_j\}\subset A^{p}_\om$ such that
$\sup_{j}\|f_j\|_{A^{p}_\om}<\infty$. As earlier, since
$\om\in\Inv$, we may use Lemma~\ref{le:uniInv} and Montel's
theorem to find a subsequence $\{f_{j_{k}}\}$ such that
$f_{j_{k}}$ converges uniformly on compact subsets of $\D$ to some
$f\in\H(\D)$, that in fact belongs to $A^{p}_\om$ by Fatou's lemma.
Therefore $\{g_{k}\}=\{f_{j_{k}}-f\}$ convergence uniformly to $0$
on compact subsets of $\D$ and
$\sup_{k}\|g_k\|_{A^{p}_\om}<\infty$. By $n$ applications of
\cite[Theorem~3.1]{PelRat}, each function $g_k$ can be factorized
to $g_k=\prod_{m=1}^ng_{k,m}$, where each $g_{k,m}\in A^{np}_\om$
and
    $$
    \prod_{m=1}^n\|g_{k,m}\|_{A^{np}_\om}\le C(n,p,\om)\|g_k\|_{A^p_\om}.
    $$
Since $\sup_{k}\|g_k\|_{A^{p}_\om}<\infty$, this implies
$\sup_{k}\|g_{k,m}\|_{A^{np}_\om}<\infty$ for all
$m=1,2,\ldots,n$. Using that $C_\vp: A^{np}_\om\to A^{nq}_v$ is bounded, we
get functions $G_1,\dots, G_m\in
A^{np}_\om$ and subsequences $\{{g_{k_{l},m}}\}$ such that
\begin{equation}\label{eq:j22}
  \|{{g_{k_{l},m}}}\circ\vp-G_m\|_{A^{q}_v}\to0,\quad
  l\to\infty,\quad m=1,2,\ldots,n.
  \end{equation}
Since the norm convergence in $A^{q}_v$ implies the uniform
convergence on compact subsets of $\D$, and $\vp$ is not constant,
the uniform convergence of $g_k$ to zero and \eqref{eq:j22} imply
that at least one of the functions $G_1,\dots, G_m$ must be
identically zero. Without loss of generality, we may assume that
$G_1\equiv 0$. Then, by H\"{o}lder's inequality, we deduce
    \begin{equation*}
    \begin{split}
    \|C_\vp({g_k}_{l})\|^{q}_{A^{q}_v}
    &=\int_\D
    \left|\left(\left(\prod_{m=1}^n{{g_{k_{l},m}}}\right)\circ\vp\right)(z)\right|^{q}v(z)\,dA(z)\\
    &\le\prod_{k=1}^n\|C_\vp({{g_{k_{l},m}}})\|^{q}_{A^{nq}_v}
    \le\|C_\vp({{g_{k_{l},1}}})\|^{q}_{A^{nq}_v}\|C_\vp\|^{(n-1)q}_{(A^{np}_\om,A^{nq}_v)}\prod_{k=2}^n
    \|{{g_{k_{l},m}}}\|^{q}_{A^{p}_\om}\\
    &\le C(n,p,q,\om)\|C_\vp({{g_{k_{l},1}}})-G_1\|^{q}_{A^{nq}_v},
\end{split}\end{equation*}
which together with \eqref{eq:j22} finishes the proof.
\end{proof}

We also need an  atomic decomposition of $A^p_\om$-functions, $\om\in\R$.
Recall that $A=\{z_k\}_{k=0}^\infty\subset\D$ is
uniformly discrete if it
is separated in the hyperbolic metric, it is an $\e$-net if $\D=\bigcup_{k=0}^\infty \Delta(z_k,\e)$, and
finally, it is a
$\delta$-lattice if it is a $5\delta$-net and uniformly
discrete with constant $\gamma=\delta/5$.

\begin{proposition}\label{pr:atomicdec}
Let $1<p<\infty$, $\om\in\R$ and
$\{z_k\}_{k=1}^\infty\subset\D\setminus\{0\}$ be an
$\e$-net. Then the following assertions hold:
\begin{itemize}
\item[\rm(i)] If $f\in A^p_\om$, then there exist
$\{c_j\}_{j=1}^\infty\in l^p$ and $M=M(\om)>0$ such that
    \begin{equation}\label{eq:de1}
    f(z)=\sum_{j=1}^\infty\frac{c_j}
    {\left(\om\left(\Delta(z_j,\epsilon)\right)\right)^{1/p}}\left(\frac{1-|z_j|^2}{1-\overline{z}_jz}\right)^{M}
    \end{equation}
and $\|\{c_j\}_{j=1}^\infty\|_{l^p}\lesssim\|f\|_{A^p_\om}$.
\item[\rm(ii)] If $\{c_j\}_{j=1}^\infty\in l^p$, then there exists
$M=M(\om)>0$ such that the function defined by the infinite sum in
\eqref{eq:de1} converges uniformly on compact subsets of $\D$ to
an analytic function $f\in A^p_\om$ and
$\|f\|_{A^p_\om}\lesssim\|\{c_j\}_{j=1}^\infty\|_{l^p}$.
\end{itemize}
\end{proposition}

\begin{proof}
(i) Let $1<p<\infty$, $\om\in\R$ and $f\in A^p_\om$. Then $\om\in
B_p(\eta)$ for each $\eta=\eta(p,\om)$ large enough by Lemma~\ref{le:RAp}. That is,
$\frac{\om(z)}{(1-|z|)^\eta}$ satisfies \cite[(4.2)]{LuecInd85}
with  $\beta=0$, $\gamma=\left(1+\frac{p}{p'}\right)\eta$ and
$\alpha=\eta$. Consequently, \cite[Theorem~4.1]{LuecInd85} implies
the existence of $\{c_j\}_{j=1}^\infty\in l^p$ such that
    \begin{equation*}
    f(z)=\sum_{j=1}^\infty
    \frac{c_j}{\left(\om\left(\Delta(z_j,\epsilon)\right)\right)^{1/p}}\left(\frac{1-|z_j|^2}{1-\overline{z}_jz}\right)^{\eta+2}
    \end{equation*}
and $\|\{c_j\}_{j=1}^\infty\|_{l^p}\lesssim\|f\|_{A^p_{\om}}$.
Hence (i) is proved with $M=\eta+2$.

(ii) Let $\{c_j\}_{j=1}^\infty\in l^p$ be given. By the proof of
(i) we know that $\om\in B_p(\eta)$ for each $\eta$ large enough.
The assertion follows by \cite[Theorem~4.1]{LuecInd85}.
\end{proof}
 An atomic-decomposition for $A^p_\om$-functions, $0<p\le 1$ and $\om\in\R$, can also be obtained
 by using the results by Constantin~\cite{Cons:IE07} (see also \cite[Theorem $2.2$]{AlCo}). However,
 we do not get into this question for a matter of simplicity and because we are able to prove Theorem~\ref{th:psmallerq} just by using Proposition~\ref{pr:atomicdec}.

\subsection{Proof of Theorem~\ref{th:psmallerq}}
We shall prove (iv)$\Rightarrow$(iii)$\Rightarrow$(i)$\Rightarrow$(iv)$\Rightarrow$(ii)$\Rightarrow$(i).
\par (iv)$\Rightarrow$(iii). If $0<r<1$ is fixed, then
$N_{\vp,\om^\star}$ is subharmonic in each pseudohyperbolic disc
that is sufficiently close to the boundary by
Lemma~\ref{Nsubharmonic}. The implication follows by this fact
because $\om^\star$ is essentially constant on pseudohyperbolic discs.

(iii)$\Rightarrow$(i). Let first $2\le q<\infty$, and let $0<r<1$
be fixed. Then the function $|f|^{q-2}|f'|^2$ is subharmonic. By
using this and arguing as in the proof of
\cite[Lemma~2.4]{SmithYang98}, we deduce
    \begin{equation}\label{subharmonic}
    \begin{split}
    |f(\zeta)|^{q-2}|f'(\zeta)|^2
    &\lesssim \frac{1}{(1-|\zeta|)^2}\int_{\De(\zeta,r^2)}|f(z)|^{q-2}|f'(z)|^2\,dA(z)\\
    &\lesssim \frac{1}{(1-|\zeta|)^4}\int_{\De(\zeta,r)}|f(z)|^{q}\,dA(z),\quad \zeta\in\D.
    \end{split}
    \end{equation}
Now Theorem~\ref{ThmLittlewood-Paley}, a change of variable, \eqref{subharmonic} and
Fubini's theorem 
give
    \begin{equation}\label{eq:qmenorp}
    \begin{split}
    \|C_\vp(f)\|_{A^q_{v}}^q&\asymp\int_\D|f(\vp(z))|^{q-2}|f'(\vp(z))|^2|\vp'(z)|^2v^\star(z)\,dA(z)\\
    &\quad+v(\D)|f(\vp(0))|^q\\
    &=\int_\D|f(\zeta)|^{q-2}|f'(\zeta)|^2N_{\vp,v^\star}(\zeta)\,dA(\zeta)+v(\D)|f(\vp(0))|^q\\
    &\lesssim\int_\D\left(\frac{1}{(1-|\zeta|)^4}\int_{\De(\zeta,r)}|f(z)|^{q}\,dA(z)\right)
    N_{\vp,v^\star}(\zeta)\,dA(\zeta)\\
    &\quad+v(\D)|f(\vp(0))|^q\\
    &\asymp\int_\D\left(\frac{1}{(1-|z|)^2}\int_{\De(z,r)}\frac{N_{\vp,v^\star}(\zeta)}{(1-|\zeta|)^2}\,dA(\zeta)\right)
    |f(z)|^{q}\,dA(z)\\
    &\quad+v(\D)|f(\vp(0))|^q.
    \end{split}
    \end{equation}
Let $M[f]$ denote the Hardy-Littlewood maximal function defined by
    \begin{equation*}
    M[f](z)=\sup_{\delta>0}\frac{1}{A(\Delta(z,\delta))}\int_{\Delta(z,\delta)}|f(\zeta)|\,dA(\zeta),\quad
    z\in\D,
    \end{equation*}
for each $f\in L^1$. The maximal function $M[f]$ is bounded on
$L^p$ when $p>1$. Therefore \eqref{eq:qmenorp}, the assumption
$\om\in\R$, H\"older's inequality and \eqref{3} yield
    \begin{equation*}
    \begin{split}
    \|C_\vp(f)\|_{A^q_{v}}^q
    &\lesssim
    \int_\D\frac{1}{\om(z)^{\frac{p-q}{p}}}\left(\frac{1}{(1-|z|)^2}\int_{\De(z,r)}\frac{N_{\vp,v^\star}(\zeta)}{(1-|\zeta|)^2\om^{q/p}(\zeta)}\,dA(\zeta)
    \right)\\
    &\quad\cdot|f(z)|^{q}\om(z)\,dA(z)+v(\D)|f(\vp(0))|^q\\
    &\lesssim \int_\D\frac{1}{\om(z)^{\frac{p-q}{p}}}M\left[  \frac{N_{\vp,v^\star}}{(1-|\zeta|)^2\om^{q/p}}    \right](z)|f(z)|^{q}\om(z)\,dA(z)\\
    &\quad+v(\D)|f(\vp(0))|^q\\
    &\lesssim \|f\|^q_{A^p_\om}\left\Vert M\left[  \frac{N_{\vp,v^\star}}{(1-|\zeta|)^2\om^{q/p}}    \right] \right\Vert^{\frac{p}{p-q}}_{L^{\frac{p}{p-q}}}\\
    &\lesssim \|f\|^q_{A^p_\om}\left\Vert  \frac{N_{\vp,v^\star}}{(1-|\zeta|)^2\om^{q/p}}     \right\Vert^{\frac{p}{p-q}}_{L^{\frac{p}{p-q}}}\\
    &\asymp\|f\|^q_{A^p_\om}\left\Vert \frac{N_{\vp,v^\star}}{\om^{\star}}     \right\Vert^{\frac{p}{p-q}}_{L^{\frac{p}{p-q}}_\om}
    \lesssim\|f\|^q_{A^p_\om}.
    \end{split}
    \end{equation*}
Thus $C_\vp: A^p_\om\to A^q_v$ is bounded provided $q\ge 2$. If
$0<q<2$, we may choose $n\in\N$ such that $nq\ge2$. Then $C_\vp:
A^{np}_\om\to A^{nq}_v$ is bounded by the previous argument and so
is $C_\vp: A^{p}_\om\to A^{q}_v$ by
Proposition~\ref{pr:compinv}.

(i)$\Rightarrow$(iv). By Proposition~\ref{pr:compinv}(i) we may
assume that $q\ge2$. Next, bearing in mind
Proposition~\ref{pr:atomicdec}, we pick up an $\epsilon$-net
$\{z_k\}_{k=0}^\infty\subset\D\setminus\{0\}$ and $M=M(\om)>0$
large enough such that for any $\{c_j\}_{j=1}^\infty\in l^p$ the
function $f$ defined by (\ref{eq:de1}) converges uniformly on
compact subsets of $\D$ to an analytic function $f\in A^p_\om$
such that $\|f\|_{A^p_\om}\lesssim\|\{c_j\}_{j=1}^\infty\|_{l^p}$.
For simplicity, let us write
$h_j(z)={\left(\om\left(\Delta(z_j,\epsilon)\right)\right)^{-1/p}}\left(\frac{1-|z_j|^2}{1-\overline{z}_jz}\right)^{M}$.
Let us consider the classical Rademacher functions $\{r_j(t)\}$ and set
$f_t(z)=\sum_{j=1}^\infty r_j(t)c_jh_j(z)$. Since $C_\vp:
A^p_\om\to A^q_v$ is bounded by the assumption,
    $$
    \|C_\vp(f_t)\|^q_{A^q_v}\le\|C_\vp\|_{(A^p_\om,A^q_v)}^q\|f_t\|^q_{A^p_\om}
    \lesssim\|C_\vp\|_{(A^p_\om,A^q_v)}^q\|\{c_j\}_{j=1}^\infty\|^q_{l^p},
    $$
from which an integration with respect to $t$ gives
    $$
    \sum_{j=1}^\infty|c_j|^q\|C_\vp(h_j)\|^q_{A^q_v}
    \le\left\|\left(\sum_{j=1}^\infty|c_j|^2|C_\vp(h_j)|^2\right)^{1/2}\right\|^q_{L^q_v}
    \lesssim\|\{c_j\}_{j=1}^\infty\|^q_{l^p},
    $$
where in the first inequality we used the hypothesis $q\ge 2$.
Therefore Theorem~\ref{ThmLittlewood-Paley} and a change of
variable give
    $$
    \sum_{j=1}^\infty|c_j|^q\int_\D|h_j(\zeta)|^{q-2}|h'_j(\zeta)|^2N_{\vp,v^\star}(\zeta)\,dA(\zeta)
    \lesssim \|\{c_j\}_{j=1}^\infty\|^q_{l^p}.
    $$
Now, a calculation shows that
    $$
    \sum_{j=1}^\infty\frac{|c_j|^q}{\left(\om\left(\Delta(z_j,\epsilon)\right)\right)^{q/p}(1-|z_j|^2)^2}\int_{\Delta(z_j,r)}
    N_{\vp,v^\star}(\zeta)\,dA(\zeta)
    \lesssim\|\{c_j\}_{j=1}^\infty\|^q_{l^p}
    $$
for any fixed $0<r<1$, and so the sequence
    $$
    \left\{\frac{\int_{\Delta(z_j,r)}
    N_{\vp,v^\star}(\zeta)\,dA(\zeta)}{\left(\om\left(\Delta(z_j,\epsilon)\right)\right)^{q/p}(1-|z_j|^2)^2}\right\}_{j=1}^\infty
    $$
belongs to $\left(l^{p/q}\right)^\star$. Since $\om\in\R$, this is
equivalent to
    $$
    \sum_{j=1}^\infty\left(\frac{\int_{\Delta(z_j,r)}
    \frac{N_{\vp,v^\star}(\zeta)}{\om^\star(\zeta)}\,dA(\zeta)}{(1-|z_j|^2)^2}
    \right)^{\frac{p}{p-q} } \om\left(\Delta(z_j,\epsilon)\right)<\infty.
    $$
Finally, for a given $0<s<1$, by choosing $0<\ep<s<r<1$
appropriately, we deduce
    \begin{equation*}
    \begin{split}
    &\int_\D\left(\frac{1}{(1-|z|)^2}\int_{\De(z,s)}
    \frac{N_{\vp,v^\star}(\zeta)}{\om^\star(\zeta)}\,dA(\zeta)\right)^{\frac{p}{p-q}}\om(z)\,dA(z)\\
    &\le\sum_{j=1}^\infty\int_{\Delta(z_j,\ep)}
    \left(\frac{1}{(1-|z|)^2}\int_{\De(z,s)}
    \frac{N_{\vp,v^\star}(\zeta)}{\om^\star(\zeta)}\,dA(\zeta)\right)^{\frac{p}{p-q}}\om(z)\,dA(z)\\
    &\lesssim\sum_{j=1}^\infty  \left(
    \frac{\int_{\Delta(z_j,r)}
    \frac{N_{\vp,v^\star}(\zeta)}{\om^\star(\zeta)}\,dA(\zeta)}{(1-|z_j|^2)^2}
    \right)^{\frac{p}{p-q}}\om\left(\Delta(z_j,\epsilon)\right)<\infty,
    \end{split}
    \end{equation*}
and thus (iv) is satisfied.

Since trivially (ii)$\Rightarrow$(i), it suffices to show
(iv)$\Rightarrow$(ii) to complete the proof.
 By  Proposition~\ref{pr:compinv}(ii) we may assume that $q\ge2$.
 Since $\om\in\R$, it is enough to prove that for
each $\{f_n\}\in A^p_\om$ that converges to $0$ uniformly on
compact subsets of $\D$ and $K=\sup_n\|f_n\|_{A^p_\om}<\infty$ we
have
   \begin{equation}\label{eq:j3}
    \lim_{n\to\infty}\|C_\vp f_n\|_{A^q_v}=0.
    \end{equation}
To see this, let $\ep>0$. Choose $r_0$ such that $\vp(0)\in
D(0,r_0)$ and
   $$
   \left(\int_{r_0<|z|<1}\left(\frac{1}{(1-|z|)^2}\int_{\Delta(z,r)}
   \frac{N_{\vp,v^\star}(\zeta)}{\om^\star(\zeta)}\right)^{\frac{p}{p-q}}\om(z)\,dA(z)\right)^{(p-q)/p}<\ep.
   $$
Further, let $n_0\in\N$ such that $|f_n(z)|<\ep^{1/q}$ for all
$n\ge n_0$ and $z\in\overline{D(0,r_0)}$. Then \eqref{eq:qmenorp}
shows that for all $n\ge n_0$ we have
    \begin{equation*}
    \begin{split}
    \|C_\vp(f_n)\|_{A^q_{v}}^q
    &\lesssim \int_\D\left(\frac{1}{(1-|z|)^2}\int_{D(z,r)}
    \frac{N_{\vp,v^\star}(\zeta)}{(1-|\zeta|)^2}\,dA(\zeta)\right)|f(z)|^{q}\,dA(z)\\
    &\quad+v(\D)|f(\vp(0))|^q\\
    &\lesssim\int_\D\left(\frac{1}{(1-|z|)^2}\int_{D(z,r)}
    \frac{N_{\vp,v^\star}(\zeta)}{\om^\star(\zeta)}\,dA(\zeta)\right)|f_n(z)|^{q}\om(z)\,dA(z)+\ep v(\D)\\
    &\lesssim\ep\int_{D(0,r_0)}\left(\frac{1}{(1-|z|)^2}\int_{D(z,r)}
    \frac{N_{\vp,v^\star}(\zeta)}{\om^\star(\zeta)}\,dA(\zeta)\right)\om(z)\,dA(z)+\|f_n\|^q_{A^p_\om}\\
    &\quad\cdot\left(\int_{r_0<|z|<1}
    \left(\frac{1}{(1-|z|)^2}\int_{D(z,r)}
    \frac{N_{\vp,v^\star}(\zeta)}{\om^\star(\zeta)}\,dA(\zeta)\right)^{\frac{p}{p-q}}\om(z)\,dA(z)\right)^{(p-q)/p}\\
    &\quad+\ep v(\D)\lesssim \ep,
    \end{split}
    \end{equation*}
which gives \eqref{eq:j3} and thus completes the proof.

\subsection*{Acknowledgements}
\par These notes are related to the course \lq\lq Weighted Hardy-Bergman spaces\rq\rq I delivered in the
 the Summer School  \lq\lq Complex and Harmonic Analysis and Related Topics\rq\rq at the University
of Eastern Finland, June $2014$, that is essentially based on several joint projects together with Jouni~R\"atty\"a. I would  like to thank the organizers for inviting me to participate in the meeting
and for their great hospitality. I am also very grateful to  all the participants for the nice research environment
we enjoyed throughout  these days.
 
 \end{document}